\documentclass[10pt, reqno]{amsart}
\usepackage[utf8]{inputenc}

\usepackage{graphicx}
\usepackage{comment}
\usepackage{fourier}
\usepackage[T1]{fontenc}
\usepackage[margin=1in]{geometry}
\usepackage{parskip}
\usepackage{amsthm}
\usepackage{amsmath}
\usepackage{amssymb}
\usepackage{mathtools}
\usepackage{apptools}
\usepackage{setspace}
\usepackage{url}
\usepackage{tabularx}
\usepackage[
backend=biber,
style=alphabetic,
sorting=nyt,
maxnames=10,
giveninits=true
]{biblatex}
\newcommand {\R}{\mathbb{R}}

\newcommand {\Z}{\mathbb{Z}}
\newcommand {\N}{\mathbb{N}}

\newcommand {\C}{\mathbb{C}}

\newcommand{\la}{\lambda}
\newcommand{\re}{\operatorname{Re}}
\newcommand{\im}{\operatorname{Im}}
\newcommand{\sumast}{\mathop{\sum\nolimits^{\mathrlap{\ast}}}}

\newcommand{\Mod}[1]{\ (\mathrm{mod}\ #1)}
\newtheorem{thm}{Theorem}[section]

\newtheorem{lemma}[thm]{Lemma}

\newtheorem{cor}{Corollary}[thm]
\newtheorem*{remark}{Remark}
\newtheoremstyle{named}{}{}{\itshape}{}{\bfseries}{.}{.5em}{\thmnote{#3}}
\theoremstyle{named}

\AtAppendix{\counterwithin{lemma}{section}}
\numberwithin{equation}{section}

\usepackage{multicol}
\usepackage{xcolor}
\usepackage{hyperref}
\hypersetup{
    colorlinks,
    linkcolor={red!50!black},
    citecolor={blue!50!black},
    urlcolor={blue!80!black}
}

\title{Hybrid Subconvexity Bound for $L\left(\frac{1}{2},\mathrm{Sym}^2 f\otimes\rho\right)$ via the Delta Method}
\author{Wing Hong Leung}
\date{}

\begin{document}

\setcounter{page}{1}

\maketitle

\section*{Abstract}

Let $P$ be a prime and $k$ be an even integer. Let $f$ be a full level holomorphic cusp form of weight $k$ and $\rho$ be a primitive level $P$ holomorphic cusp form with arbitrary nebentypus and fixed weight $\kappa$. We prove a hybrid subconvexity bound for $L\left(\frac{1}{2},\mathrm{Sym}^2 f\otimes \rho\right)$ when $P^{\frac{1}{4}+\eta}<k<P^{\frac{21}{17}-\eta}$ for any $0<\eta<\frac{67}{136}$. This extends the range of $P$ and $k$ achieved by Holowinsky, Munshi and Qi \cite{HybridSub}. The result is established using a new variant of the delta method.

\section{Introduction and Statement of Results}

Subconvexity bounds for $L$-functions have been a big area of study due to their connections to equidistribution problems and other arithmetic properties. In the context of $L\left(\frac{1}{2},\mathrm{Sym}^2 f\otimes\rho\right)$, with $f$ a full level holomorphic Hecke eigen cusp form of weight $k$, and $\rho$ a level $P$ holomorphic Hecke eigen cusp form with arbitrary nebentypus and fixed weight $\kappa$, establishing a subconvexity bound refers to showing the existence of some $\delta>0$ such that \begin{align*}
    L\left(\frac{1}{2},\mathrm{Sym}^2 f\otimes\rho\right)\ll_\varepsilon \left(P^3k^4\right)^{\frac{1}{4}-\delta+\varepsilon}
\end{align*}
for any $\varepsilon>0$. Although the size $\delta$ usually does not matter in applications, achieving any such nontrivial bound is difficult. In contrast, the above bound with $\delta=0$ is a consequence of the Phragm$\acute{\text{e}}$n Lindel$\ddot{\text{o}}$f principle and other standard complex analysis techniques.

One motivation to study this particular $L$-function is its striking connection with the Quantum Unique Ergodicity (QUE) conjecture introduced by Rudnick and Sarnak \cite{RS}. This conjecture has gathered huge attention and various techniques have been developed to attack the problem. In many cases, one can relate the QUE conjecture to the subconvexity bounds for $L\left(\frac{1}{2},\mathrm{Sym}^2 f\right)$ and $L\left(\frac{1}{2},\mathrm{Sym}^2 f\otimes\rho\right)$, with $\rho$ being a fixed form.

In 2013, Rizwanur Khan \cite{RizwanSub} employed a conditional amplifier in a first moment method to establish a subconvexity bound for $L\left(\frac{1}{2},\mathrm{Sym}^2 f\otimes\rho\right)$ when $f$ is fixed. Inspired by his work, R. Holowinsky, R. Munshi and Z. Qi \cite{HybridSub} used a first moment method to obtain a hybrid subconvexity bound when $f$ and $\rho$ both vary. Having both forms vary allowed them to remove the conditional amplifier. To be precise, if $f$ and $\rho$ are as above, they showed \cite[Cor 2.3]{HybridSub} that there exists explicit $\delta>0$ such that for any $\varepsilon>0$, \begin{align*}
    L\left(\frac{1}{2},\mathrm{Sym}^2 f\otimes\rho\right)\ll_\varepsilon \left(P^3k^4\right)^{\frac{1}{4}-\delta+\varepsilon}
\end{align*}
if $P^{\frac{13}{64}+\eta}<k<P^{\frac{3}{8}-\eta}$ for some $0<\eta<\frac{11}{128}$.

In recent years, various delta methods have been developed and extensively applied to solve boundary value problems. Kumar, Munshi and Singh \cite{MunshiHybrid} recently showed hybrid subconvexity for $L\left(\frac{1}{2},\mathrm{Sym}^2 f\otimes\rho\right)$ with $f$ level $P_1$, $\rho$ level $P_2$, both with fixed weight, when $P_2^{1/2+\varepsilon}<P_1<P^{3/2-\varepsilon}$. In this paper, we approach the subconvexity problem with a new variant of the delta method which we develop here. We demonstrate how our delta method can be applied to study boundary value problems, and in this context we achieve a hybrid subconvexity bound for $L\left(\frac{1}{2},\mathrm{Sym}^2 f\otimes\rho\right)$, extending the range achieved by \cite{HybridSub}.

\subsection{Reformation of the Delta Method}

In Appendix \ref{DeltaAppendix}, we prove the following delta method which we use to tackle the hybrid subconvexity problem. Let $\varepsilon>0$, $n$ be an integer such that $|n|\ll N\rightarrow\infty$ and let $C>N^\varepsilon$ be a parameter. Let $U\in C_c^\infty(\R), W\in C_c^\infty([-2,-1]\bigcup [1,2])$ be a fixed non-negative even function such that $U(x)=1$ for $-2\leq x\leq 2$. Then we have \begin{align*}
        \delta(n=0)=\frac{1}{\mathcal{C}}\sum_{c\geq1}\frac{1}{c}\sum_{\alpha\Mod{c}}e\left(\frac{\alpha n}{c}\right)h\left(\frac{c}{C},\frac{n}{cC}\right),
    \end{align*}
    with $\mathcal{C}=\displaystyle\sum_{c\geq1}W\left(\frac{c}{C}\right)\sim C$ and \begin{align*}
        h\left(x,y\right)=W\left(x\right)U\left(x\right)U\left(y\right)-W(y)U(x)U(y).
    \end{align*}
    In particular, $h$ is a fixed smooth function satisfying $h(x,y)\ll \delta(|x|,|y|\ll 1).$

One can observe that our delta method is not too different from the delta method of Duke-Friedlander-Iwaniecs \cite[\S 3]{DFIdelta}, \cite[Thm 1]{heathbrown} together with a "conductor lowering trick" which truncates the size of the equation. One could also draw comparison to Kloosterman's refinement of the circle method instead \cite{kloostermanrefinement}, \cite{munshi3}. Indeed, our version of the delta method is inspired by the start of the proof of the D-F-I's delta method \cite[\S 3]{DFIdelta}, with the additional freedom of choice of parameters seen in the "conductor lowering" function $U$. See Appendix \ref{DeltaAppendix} for details.

The main advantage of our delta method, for our application here, is that it is very explicit and our function $U$ truncates $|n|\ll cCN^\varepsilon$, as compared to the usual truncation of $|n|\ll C^2N^\varepsilon$ if one adds an artificial "conductor lowering" to the D-F-I delta method. As a result, the reader will see this simplifies some of the initial dual summation steps. On the other hand, if one uses the Fourier inversion version of the D-F-I delta method, i.e. \begin{align*}
    \delta(n=0) = \frac{1}{C}\sum_{1\leq c\leq C} \frac{1}{c}\sumast_{\alpha\Mod{c}}e\left(\frac{an}{c}\right) \int_{\R} g(c,x) e\left(\frac{nx}{cC}\right)dx.
\end{align*}
While we have $g(c,x)\ll C^\varepsilon$, it is an implicit function and the dependence on q can create technical difficulties. The Fourier inversion version of our delta method comes with $\tilde{U}(x)$ (the Fourier transform of $U$) instead, and is a fixed function.

\subsection{Hybrid Subconvexity}

Let $f$ be a full level holomorphic Hecke eigen cusp form of weight $k$ and $\rho$ be a primitive level $P$ holomorphic Hecke eigen cusp form with arbitrary nebentypus and fixed weight $\kappa$ as above. By the approximate functional equation (\ref{AFE}) and a smooth dyadic subdivision, we will prove the following lemma in Section \ref{SetupSection}.

\begin{lemma}\label{SetupLemma}
    Let $\varepsilon>0$ and let $f, \rho$ as above with $P>k^\varepsilon$. Then for any $A>0$, we have \begin{align*}
        L\left(\frac{1}{2},\mathrm{Sym}^2 f\otimes \rho\right)\ll_\varepsilon \sup_{\frac{1}{2}\leq N\leq P^{3/2+\varepsilon}k^2}\sum_{(d, P)=1} \left|\sum_{n\geq1} \frac{A_{F}(1, n) \la_\rho(d n)}{\sqrt{N}} V\left(\frac{d^{3} n}{N}\right)\right|+P^{-A},
    \end{align*}
    for some $P^\varepsilon$-inert function $V$ supported on $[1,2]$, depending on $f$ and $\rho$.
\end{lemma}

Define \begin{align}\label{SNsetup}
    S_d(N)=\sum_n\la_\rho(dn)A(1,n)V\left(\frac{d^3n}{N}\right).
\end{align}
Then Lemma \ref{SetupLemma} reduces the problem of subconvexity to bounding $S_d(N)$. In Section \ref{MainthmSection1} and \ref{MainthmSection2} we apply our delta method in Lemma \ref{DeltaMethod} to prove the following main theorem.

\begin{thm}\label{mainthm}
    Let $\varepsilon>0$ and let $f, \rho$ be as above with $P>k^\varepsilon$. Let $N\leq P^{3/2+\varepsilon}k^2$, $(d,P)=1$, we have \begin{align}\label{trivialBound}
        S_d(N)\ll\frac{NP^\varepsilon}{d^3}.
    \end{align}
    Moreover, if $P^\frac{1}{4}<k\leq P^\frac{6}{5}, k>d^3, P>d^\frac{8}{3}$ and $N\geq (Pk)^\frac{4}{3}d^2+P^\frac{8}{5}k^\frac{4}{5}d^\frac{12}{5}+P^2d^2$, then we have \begin{align}\label{mainBound1}
        S_d(N)\ll P^\varepsilon\left(\frac{P^\frac{1}{10}N^\frac{19}{20}}{d^\frac{12}{5}}+\frac{P^\frac{1}{6}N^\frac{3}{4}k^\frac{2}{3}}{d^\frac{3}{2}}+\frac{P^\frac{1}{6}N^\frac{5}{6}k^\frac{1}{3}}{d^2}+\frac{\sqrt{P}N^\frac{3}{4}}{d^2}\right).
    \end{align}
    If $P^\frac{6}{5}\leq k< P^\frac{21}{17}, k>d^3, P>d^{24}$ and $P>d^{24} \text{ and } N\geq (Pk)^\frac{4}{3}d^2+P^\frac{5}{3}k^\frac{19}{18}d^\frac{7}{3}+P^2d^2$, we have \begin{align}\label{mainBound2}
        S_d(N)\ll P^\varepsilon\frac{P^\frac{1}{12}N^\frac{3}{4}k^\frac{53}{72}}{d^\frac{19}{12}}.
    \end{align}
\end{thm}

Applying the bounds in Theorem \ref{mainthm} appropriately, we prove the following corollary in Section \ref{maincorSect}.

\begin{cor}\label{maincor}
    Let $\varepsilon>0$ and let $f, \rho$ as above. Then for $P^\frac{1}{4}<k< P^\frac{21}{17}$, we have
    \begin{align*}
        L\left(\frac{1}{2},\mathrm{Sym}^2 f\otimes \rho\right)\ll_\varepsilon P^\varepsilon\left(P^{\frac{31}{40}}k^{\frac{9}{10}}+P^{\frac{13}{24}}k^{\frac{7}{6}}+P^{\frac{2}{3}}k+P^{\frac{11}{24}}k^{\frac{89}{72}}\right).
    \end{align*}
    In particular, for $P^{\frac{1}{4}+\eta}<k<P^{\frac{17}{21}-\eta}$ with $0<\eta<\frac{67}{136}$, there exists computable $\delta>0$ such that for any $\varepsilon$, \begin{align*}
        L\left(\frac{1}{2},\mathrm{Sym}^2 f\otimes \rho\right)\ll_\varepsilon (P^3k^4)^{\frac{1}{4}-\delta+\varepsilon}.
    \end{align*}
\end{cor}

\subsection*{Acknowledgments}

The author thanks R. Holowinsky for suggesting the problem, guidance and encouragements.

\section{Notations}

Let $\varepsilon>0$, $a, b\in \C$, $c\in\Z$, $P\rightarrow\infty$.

\begin{spacing}{1.5} 
 \begin{tabbing} 
 \= Simbolo \= Separador \= Significado \+ \kill 
 {\bf Symbol} \> \> {\bf Meaning} \\ \\
 $a\ll b$ \> \> There exists constant $c>0$ such that $|a|\leq c|b|$\\
 $a\sim b$ \> \> $b\ll a\ll b$\\
 $a\sim_\varepsilon b$ \> \> $bP^\varepsilon\ll a\ll bP^\varepsilon$\\
 $(a,b)$ \> \> The greatest common divisor of $a$ and $b$\\
 $e(x)$ \> \> $e^{2\pi ix}$\\
 $\displaystyle\sumast_{\alpha\Mod{c}}$ \> \> Sum over $\alpha\in (\Z/c)^*$\\
 $S(a,b;c)$ \> \> The Kloosterman sum $\displaystyle\sumast_{\alpha\Mod{c}}e\left(\frac{a\alpha+b\overline{\alpha}}{c}\right)$
 \end{tabbing}
\end{spacing}

We say $f$ is $X$-inert if it is smooth and satisfies for any $x=(x_1,...,x_n)\in\R^n$, $j_1,...,j_n\in\N$, \begin{align*}
    \text{either } x_1^{j_1}\cdots x_n^{j_n}f^{(j_1,...,j_n)}(x)\ll X^{j_1+\cdots+j_n} \text{ or } f(x)\ll P^{-A} \text{ for any } A>0.
\end{align*}
Note that a similar definition is introduced by \cite{inert}, we include $f(x)\ll P^{-A}$ for any $A>0$ for convenience.

\section{Sketch of Proof of Theorem \ref{mainthm} in the simplified main case}

Here we demonstrate the main idea of the proof of Theorem \ref{mainthm} by sketching out the simplified main case when $P^\frac{1}{4}<k\leq P^\frac{6}{5}$. For simplicity, we take $\rho$ to be of trivial nebentypus and focus on $N=P^\frac{3}{2}k^2, d=1$. In this case, we want to beat the bound \begin{align*}
    S_1(P^\frac{3}{2}k^2)\ll P^\frac{3}{2}k^2.
\end{align*}
Start by applying our delta method to separate the oscillations. Apply Corollary \ref{DeltaCor} with $1\leq C^2\leq P^\frac{3}{2}k^2$, we roughly get \begin{align*}
    S_1(P^\frac{3}{2}k^2)\sim& \frac{1}{C}\sum_{c\leq C}\frac{1}{c}\sum_{m\sim P^\frac{3}{2}k^2}\la_\rho(m)V\left(\frac{m}{P^\frac{3}{2}k^2}\right)\sum_{n\sim P^\frac{3}{2}k^2}A(1,n)\sumast_{\alpha\Mod{c}}e\left(\frac{\alpha(n-m)}{c}\right)h\left(\frac{c}{C},\frac{n-m}{cC}\right),
\end{align*}
with $h$ being a fixed smooth function satisfying $h(x,y)\ll\delta(|x|,|y|\ll1)$. Splitting the case whether $P|c$ or not, we focus on the main case $(c,P)=1$ and $c\sim C$. Performing Voronoi summation on the $m$-sum yields \begin{align*}
    S_1(P^\frac{3}{2}k^2)\sim& \frac{1}{C^3\sqrt{P}}\sum_{c\sim C} \sum_m\la_{\overline{\rho}}(m)\sum_{n\sim P^\frac{3}{2}k^2}A(1,n)S(n,\overline{P}m;c)\\
    &\times \int_0^\infty V\left(\frac{y}{P^\frac{3}{2}k^2}\right)h\left(\frac{c}{C},\frac{n-y}{cC}\right)J_{\kappa-1}\left(\frac{4\pi\sqrt{my}}{c\sqrt{P}}\right)dy.
\end{align*}
One can see in the proof in Section \ref{VoronoiSection} that having the fixed function $h$ with fixed compact support on both variables makes it easier to extract the asymptotics from the integral, leading us to obtain \begin{align*}
    S_1(P^\frac{3}{2}k^2)\sim\frac{1}{P^\frac{5}{4}k}\sum_\pm\sum_{c\sim C}\sum_{m\ll\frac{P^\frac{5}{2}k^2}{C^2}}\la_{\overline{\rho}}(m)\sum_{n\sim P^\frac{3}{2}k^2}A(1,n)S(n,\overline{P}m;c)e\left(\pm\frac{2\sqrt{mn}}{c\sqrt{P}}\right).
\end{align*}

Next we perform Voronoi summation on the $n$-sum. The integral transform that appears is quite complicated and this is the most computationally heavy part of the proof. There are also some non-generic "transition range" cases which create obstacles and prevents us from getting subconvexity bound for larger values of $k$. Focusing on the main case $m\sim \frac{P^\frac{5}{2}k^2}{C^2}$ and choosing $C>P^\frac{3}{4}\sqrt{k}$, analysis of the integral transform yields roughly \begin{align*}
    S_1(P^\frac{3}{2}k^2)\sim\frac{Ck}{\sqrt{P}}\sum_\pm\sum_{c\sim C}\sum_{m\sim\frac{P^\frac{5}{2}k^2}{C^2}}\la_{\overline{\rho}}(m)\sum_{n\sim Ck^2}\frac{A(n,1)}{n}e\left(-\frac{P\overline{m}n}{c}\right)e\left(\mp\frac{ck^2m}{4\pi^2Pn}\right).
\end{align*}
At this point the trivial bound is \begin{align*}
    S_1(P^\frac{3}{2}k^2)\ll P^\varepsilon\frac{Ck}{\sqrt{P}}C\frac{P^\frac{5}{2}k^2}{C^2}\ll P^{2+\varepsilon}k^3.
\end{align*}
We have to save $\sqrt{P}k$ and a little bit more.

Finally, we apply Cauchy-Schwartz and take out the $n$-sum, giving us \begin{align*}
    S_1(P^\frac{3}{2}k^2)\ll \sqrt{\frac{C}{P}}\left(\sum_{n\sim Ck^2}\sum_{c_1\sim C}\sum_{c_2\sim C}\mathop{\sum\sum}_{m_1,m_2\sim\frac{P^\frac{5}{2}k^2}{C^2}}\la_{\overline{\rho}}(m_1)\overline{\la_{\overline{\rho}}(m_2)}e\left(-\frac{P\overline{m_1}n}{c_1}+\frac{P\overline{m_2}n}{c_2}\mp\frac{c_1k^2m_1}{4\pi^2Pn}\pm\frac{c_2k^2m_2}{4\pi^2Pn}\right)\right)^\frac{1}{2}.
\end{align*}
The two extreme cases are the diagonal $\Delta$, i.e. $c_1=c_2,m_1=m_2$, and the worst offdiagonal $\mathcal{O}$, when $|c_1m_1-c_2m_2|\sim\frac{P^\frac{5}{2}k^2}{C}$. For the diagonal, we have the bound \begin{align*}
    \Delta\ll P^\varepsilon\sqrt{\frac{C}{P}}\sqrt{Ck^2C\frac{P^\frac{5}{2}k^2}{C^2}}\ll \sqrt{C}P^{\frac{3}{4}+\varepsilon}k^2.
\end{align*}
For the offdiagonal, we apply Poisson summation to the $n$-sum and apply stationary phase to extract the asymptotics from the Fourier transform, which in the worst case scenario $|c_1m_1-c_2m_2|\sim\frac{P^\frac{5}{2}k^2}{C}$ gives us \begin{align*}
    \mathcal{O}\ll\sqrt{\frac{C}{P}}\left(\frac{C^2k}{P^\frac{3}{4}}\sum_{n\sim \frac{P^\frac{3}{2}}{C}}\sum_{c_1\sim C}\sum_{c_2\sim C}\mathop{\sum\sum}_{m_1,m_2\sim\frac{P^\frac{5}{2}k^2}{C^2}}\la_{\overline{\rho}}(m_1)\overline{\la_{\overline{\rho}}(m_2)}\delta\left(n\equiv P\overline{m_1}c_2-P\overline{m_2}c_1\Mod{c_1c_2}\right)\right)^\frac{1}{2}.
\end{align*}
Focusing on the generic case $(c_1,c_2)=1$, we have that $n,c_1,c_2$ determines $m_1,m_2$ mod $c_1,c_2$ respectively. Therefore we get the bound \begin{align*}
    \mathcal{O}\ll P^\varepsilon \sqrt{\frac{C}{P}\frac{C^2k}{P^\frac{3}{4}}\frac{P^\frac{3}{2}}{C}C^2\left(1+\frac{P^\frac{5}{2}k^2}{C^3}\right)^2}\ll P^\varepsilon\left(\frac{C^2\sqrt{k}}{P^\frac{1}{8}}+\frac{P^\frac{19}{8}k^\frac{5}{2}}{C}\right).
\end{align*}
Combining the above bounds for $\Delta$ and $\mathcal{O}$, we have roughly \begin{align*}
    S_1(P^\frac{3}{2}k^2)\ll P^\varepsilon\left(\sqrt{C}P^\frac{3}{4}k^2+\frac{C^2\sqrt{k}}{P^\frac{1}{8}}+\frac{P^\frac{19}{8}k^\frac{5}{2}}{C}\right).
\end{align*}
Optimizing these three terms by choosing $C=P^\frac{13}{12}k^\frac{1}{3}+P^\frac{5}{6}k^\frac{2}{3}$, we get the bound \begin{align*}
    S_1(P^\frac{3}{2}k^2)\ll P^\varepsilon\left(P^\frac{31}{24}k^\frac{13}{6}+P^\frac{37}{24}k^\frac{11}{6}\right).
\end{align*}
One can see that such a bound beats $P^\frac{3}{2}k^2$ when $P^\frac{1}{4}<k<P^\frac{5}{4}$.

In the actual proof, however, there is one case in which we are not able to achieve the same subconvexity range $P^\frac{1}{4}<k<P^\frac{5}{4}$. That case comes from the new dual length of $m$ being of certain size, which puts the GL(3) Voronoi integral transform in its "transition range". We were not able to treat this non-generic case in the same fashion as the main case and therefore our range of subconvexity is ultimately reduced to $P^\frac{1}{4}<k<P^\frac{21}{17}$.

\section{Preliminaries}

\subsection{Holomorphic Hecke Cusp forms}

For $\rho$ a primitive holomorphic Hecke eigen cusp form of squarefree level $N$ with nebentypus $\chi$ and fixed weight $\kappa$, we have the Fourier expansion \begin{align*}
    \rho(z)=\sum_{n\geq1}a_\rho(n)e(nz).
\end{align*}
We can rewrite $a_\rho(n)$ as \begin{align*}
    a_\rho(n)=\la_\rho(n)a_\rho(1)n^\frac{k-1}{2},
\end{align*}
where $\la_\rho(n)$ are the Hecke eigenvalues, also known as the normalized Fourier coefficients. They satisfy the Hecke relations \begin{align*}
    \la_\rho(m)\la_\rho(n)=\sum_{\substack{d|(m,n)\\(d,N)=1}}\chi(d)\la_\rho(mn/d^2).
\end{align*}
By M$\ddot{\text{o}}$bius inversion, this is equivalent to \begin{align}\label{GL2Hecke}
    \la_\rho(mn)=\sum_{\substack{d|(m,n)\\(d,N)=1}}\mu(d)\chi(d)\la_\rho(m/d)\la_\rho(n/d).
\end{align}
Moreover, these normalized Fourier coefficients satisfy the bound \begin{align}\label{DeligneBound}
    |\la_\rho(n)|\ll n^{\varepsilon}
\end{align} for any positive integer $n$ by the work of Deligne. Moreover, they satisfy the following Voronoi summation formula by \cite[Thm A4]{KMV}.

\begin{lemma}[GL(2) Voronoi Summation]\label{GL2Voronoi}
    Let $a,c\neq0$ be integers such that $(a,c)=1$, $N_1=(c,N)$, $N_2=N/(c,N)$ and let $F\in C_c^\infty(\R^+)$. Let $\chi=\chi_{N_1}\chi_{N_2}$ be the unique decomposition into characters mod $N_1$ and $N_2$ respectively. There exists $\eta_\rho\in\C$ such that $|\eta_\rho|=1$ and the following holds, \begin{align*}
        \sum_{n\geq1}\la_\rho(n)e\left(\frac{an}{c}\right)F(n)=2\pi i^\kappa\chi_{N_1}(\overline{a})\chi_{N_2}(-c)\frac{\eta_\rho}{c\sqrt{N_2}}\sum_{n\geq1}\la_{\overline{\rho}}(n)e\left(-\frac{\overline{a N_2}n}{c}\right)\int_0^\infty F(y)J_{\kappa-1}\left(\frac{4\pi\sqrt{ny}}{c\sqrt{N_2}}\right)dy.
    \end{align*}
\end{lemma}

\subsection{Symmetric Square \texorpdfstring{$L$}{L}-function}

Let $f$ be a full level holomorphic Hecke eigen cusp form of weight $k$ and let $\la_f(n)$ be the $n$-th Hecke eigenvalue for any $n\geq1$. The symmetric square $L$-function is defined by \begin{align*}
    L(s,\text{Sym}^2f)=\zeta(2s)\sum_{n=1}^\infty\la_f(n^2)n^{-s}.
\end{align*}
There exists a $L$-function $L(s,F)$ associated to some automorphic representation $F$ of GL$(3,\Z)$ such that \begin{align}
    L(s,\text{Sym}^2f)=\sum_{n=1}^\infty A_F(1,n)n^{-s},
\end{align}
where $A_F(m,n)$ are the normalized Fourier coefficients of $F$ satisfying \begin{gather*}
    A_{F}(1, n)=A_{F}(n, 1)=\sum_{m \ell=n} \lambda_{f}\left(m^{2}\right) \\
    A_{F}(m, n)=A_{F}(-m, n)=A_{F}(m,-n)=A_{F}(-m,-n), \\
    A_{F}(m, 0)=A_{F}(0, n)=0,
\end{gather*}
and the Hecke relations \cite[(7.7)]{gl3} \begin{align}\label{HeckeRelations}
    A_{F}(m, n)=\sum_{d \mid(m, n)} \mu(d) A_{F}(m / d, 1) A_{F}(1, n / d).
\end{align}
These relations together with (\ref{DeligneBound}) imply that, \begin{align}\label{SymDeligneBound}
    |A_F(m,n)|\ll (mn)^\varepsilon
\end{align}
for any $m,n$. Moreover, these normalized Fourier coefficient satisfy the following Voronoi summation formula \cite[Thm 1.18]{gl3}.

\begin{lemma}[GL(3) Voronoi Summation]\label{GL3Voronoi}
    Let $\psi\in C_c^\infty(\R^+)$ and let $\tilde{\psi}$ be its Mellin transform, i.e. \begin{align*}
        \tilde{\psi}(s):=\int_0^\infty \psi(y)y^{s-1}dy.
    \end{align*}
    Let $a,c\neq0\in\Z$ such that $(a,c)=1$. Define \begin{align*}
        G_k^\pm(s)=\left(\frac{\Gamma\left(\frac{2+s}{2}\right)}{\Gamma\left(\frac{1-s}{2}\right)}\mp i\frac{\Gamma\left(\frac{1+s}{2}\right)}{\Gamma\left(-\frac{s}{2}\right)}\right)\frac{\Gamma\left(\frac{k+s}{2}\right)\Gamma\left(\frac{k+1+s}{2}\right)}{\Gamma\left(\frac{k-s-1}{2}\right)\Gamma\left(\frac{k-s}{2}\right)}
    \end{align*}
    and for $\sigma>-1$, define \begin{align*}
        \Psi_\pm(x)=\frac{1}{4\pi^{5/2}i}\int_{(\sigma)}(\pi^3x)^{-s}G_k^\pm(s)\tilde{\psi}(-s)ds.
    \end{align*}
    Then we have \begin{align*}
        \sum_{n\geq1} A_{F}(1, n) e\left(\frac{na}{c}\right) \psi(n)=& c \sum_\pm\sum_{n_{1} \mid c} \sum_{n_{2}\geq1} \frac{A_{F}\left(n_{2}, n_{1}\right)}{n_{1} n_{2}} S\left(\overline{a}, \pm n_{2} ; c / n_{1}\right) \Psi_\pm\left(\frac{n_{2} n_{1}^{2}}{c^{3}}\right).
    \end{align*}
\end{lemma}

Let us record the convexity bound here. For $0<\sigma<1$, we have \begin{align}\label{Sym2Convexity}
    L\left(\sigma+i t, \mathrm{Sym}^{2} f\right) \ll_{\varepsilon}\left((|t|+1)(|t|+k)^{2}\right)^{\frac{1-\sigma}{2}+\varepsilon}.
\end{align}

\subsection{Rankin-Selberg \texorpdfstring{$L$}{L}-function and Approximate Functional Equation}

Let $f$ and $\rho$ as above, we define the Rankin-Selberg $L$-function \begin{align*}
    L(s,\text{Sym}^2f\otimes\rho)=\sum_{m\geq1}\sum_{n\geq1}A_F(m,n)\la_\rho(n)(m^2n)^{-s}.
\end{align*}
If $\rho$ is holomorphic with weight $\kappa$, its gamma factor is \begin{align*}
    \begin{aligned}
        \gamma\left(s, \operatorname{Sym}^{2} f \otimes \rho\right)=(2 \pi)^{-3 s} \Gamma\left(s+\kappa-\frac{1}{2}\right) \Gamma\left(s+k+\kappa-\frac{3}{2}\right) \Gamma\left(s+\left|k-\kappa-\frac{1}{2}\right|\right).
    \end{aligned}
\end{align*}
Then the complete $L$-function \begin{align*}
    \Lambda\left(s, \operatorname{Sym}^{2} f \otimes \rho\right)=N^{\frac{3s}{2}} \gamma\left(s, \operatorname{Sym}^{2} f \otimes \rho\right) L\left(s, \operatorname{Sym}^{2} f \otimes \rho\right)
\end{align*}
satisfies \begin{align*}
    \Lambda\left(s, \operatorname{Sym}^{2} f \otimes\rho\right)=\varepsilon\left(\operatorname{Sym}^{2} f \otimes \rho\right) \Lambda\left(1-s, \operatorname{Sym}^{2} f \otimes \rho\right)
\end{align*}
for some root number $\varepsilon\left(\operatorname{Sym}^{2} f \otimes \rho\right)$ satisfying $\left|\varepsilon\left(\operatorname{Sym}^{2} f \otimes \rho\right)\right|=1$.

Using the Hecke relations (\ref{HeckeRelations}), we get \begin{align*}
    \begin{aligned}
        L\left(s, \operatorname{Sym}^{2} f \otimes \rho\right)=& \sum_{d\geq1} \mu(d) d^{-3 s} \sum_{m\geq1} A_{F}(m, 1) m^{-2 s} \sum_{n\geq1} A_{F}(1, n) \lambda_{\rho}(d n) n^{-s} \\
        =& L\left(2 s, \operatorname{Sym}^{2} f\right) \sum_{d\geq1} \mu(d) d^{-3 s} \sum_{n\geq1} A_{F}(1, n) \lambda_{\rho}(d n) n^{-s} \\
        =& L\left(2 s, \operatorname{Sym}^{2} f\right) \left(1-\la_\rho(P)P^{-s}\right) \sum_{(d, N)=1} \mu(d) d^{-3 s} \sum_{n\geq1} A_{F}(1, n) \lambda_{\rho}(d n) n^{-s}.
    \end{aligned}
\end{align*}
Define \begin{align*}
    L_N(s,\rho)=\prod_{p|N}\left(1-\la_\rho(p)p^{-s}\right)^{-1},
\end{align*}
together with \cite[Thm 5.3, Prop. 5.4]{IwaniecKowalski}, we have the approximate functional equation \begin{align}\label{AFE}
    \begin{aligned}
        L\left(\frac{1}{2}, \operatorname{Sym}^{2} f \otimes \rho\right) =& \sum_{(d, N)=1} \sum_{n\geq1} \frac{\mu(d) A_{F}(1, n) \la_\rho(d n)}{\sqrt{d^{3} n}} V_\frac{1}{2}\left(\frac{d^{3} n}{N^{\frac{3}{2}}}\right) \\
        &+\varepsilon\left(\operatorname{Sym}^{2} f \otimes g\right) \sum_{(d, N)=1} \sum_{n\geq1} \frac{\mu(d) A_{F}(1, n) \la_\rho(d n)}{\sqrt{d^{3} n}} V_\frac{1}{2}\left(\frac{d^{3} n}{N^{\frac{3}{2}}}\right),
\end{aligned}
\end{align}
where \begin{align*}
    V_\frac{1}{2}(y)=\frac{1}{2 \pi i} \int_{(3)}\left(\cos \left(\frac{\pi u}{4 A}\right)\right)^{-24 A} \frac{\gamma\left(\frac{1}{2}+u, \operatorname{Sym}^{2} f \otimes \rho\right)}{\gamma\left(\frac{1}{2}, \operatorname{Sym}^{2} f \otimes \rho\right)}\frac{L\left(1+2 u, \operatorname{Sym}^{2}f\right)}{L_N\left(\frac{3}{2}+3u,\rho\right)} y^{-u} \frac{d u}{u},
\end{align*}
for any $A>0$. By shifting the contour to $\re(u)=A$ for $y$ large and $\re(u)=-\sigma$ with $0<\sigma<1/2$ for $y$ small, passing through $u=0$ with residue $L(1,\text{Sym}^2 f)L_N(3/2,\rho)$, together with Stirling's approximation (Lemma \ref{Stirling}) and the convexity bound (\ref{Sym2Convexity}), gives us \begin{align*}
    y^jV_\frac{1}{2}^{(j)}(y)\ll_{\varepsilon,j,A} (kN)^\varepsilon\left(1+\frac{y}{k^2}\right)^{-A}
\end{align*}
for any $j\geq0$.

For more details, see \cite[\S 4]{HybridSub}.

\section{Setup: Proof of Lemma \ref{SetupLemma}}\label{SetupSection}

Here we prove Lemma \ref{SetupLemma}, which reduces the subconvexity problem into bounding the sum $S_d(N)$. By the approximate functional equation (\ref{AFE}), we have \begin{align*}
    L\left(\frac{1}{2}, \operatorname{Sym}^{2} f \otimes \rho\right) \ll\sum_{(d, P)=1} \left|\sum_{n} \frac{A_{F}(1, n) \la_\rho(d n)}{\sqrt{d^{3} n}} V_\frac{1}{2}\left(\frac{d^{3} n}{P^{\frac{3}{2}}}\right)\right|,
\end{align*}
with \begin{align*}
    y^jV_\frac{1}{2}^{(j)}(y)\ll_{\varepsilon,j,A} (kP)^\varepsilon\left(1+\frac{y}{k^2}\right)^{-A}
\end{align*}
for any $j,A\geq0$. Applying a smooth dyadic subdivision on the $d^3n$, there exists some fixed $\varphi\in C_c^\infty(\R^+)$ supported on $[1,2]$ such that for any $B>0$, \begin{align*}
    L\left(\frac{1}{2}, \operatorname{Sym}^{2} f \otimes \rho\right) \ll\sup_{\frac{1}{2}\leq N\leq P^{3/2+\varepsilon}k^2}\sum_{(d, P)=1} \left|\sum_{n} \frac{A_{F}(1, n) \la_\rho(d n)}{\sqrt{d^{3} n}} V_\frac{1}{2}\left(\frac{d^{3} n}{P^{\frac{3}{2}}}\right)\varphi\left(\frac{d^3n}{N}\right)\right|+(Pk)^{-B}
\end{align*}
Define \begin{align*}
    V(x)=\frac{1}{\sqrt{x}}V_\frac{1}{2}\left(\frac{Nx}{P^\frac{3}{2}}\right)\varphi(x),
\end{align*}
then the derivative properties of $V_\frac{1}{2}$, support of $\varphi$ together with $P>k^\varepsilon$ implies that $V$ is a $P^\varepsilon$-inert function supported on $[1,2]$ and \begin{align*}
    L\left(\frac{1}{2}, \operatorname{Sym}^{2} f \otimes \rho\right) \ll\sum_{(d, P)=1} \left|\sum_{n} \frac{A_{F}(1, n) \la_\rho(d n)}{\sqrt{N}} V\left(\frac{d^{3} n}{N}\right)\right|+P^{-A}
\end{align*}
for any $A>0$. This concludes the proof of Lemma \ref{SetupLemma}.

\section{Analysis of \texorpdfstring{$S_d(N)$}{Sd(N)}}\label{MainthmSection1}

In the rest of the paper, many of the inequalities have the implicit constant depending on $\varepsilon$. We will no longer specify so with $\ll_\varepsilon$ to avoid heavy notations.

Now we prove Theorem \ref{mainthm} in the following sections. For $d^2P^\varepsilon\leq N\leq P^{3/2+\varepsilon}k^2$, $(d,P)=1$, recall the definition in (\ref{SNsetup}) that \begin{align}
    S_d(N)=\sum_nA(1,n)\la_\rho(dn)V\left(\frac{d^3n}{N}\right),
\end{align}
with $V$ being a $P^\varepsilon$-inert function supported on $[1,2]$.

\subsection{Trivial Bound}

Here we prove the first statement of Theorem \ref{mainthm}. Using (\ref{SymDeligneBound}) and (\ref{DeligneBound}), we get \begin{align*}
    S_d(N)\leq &\sum_n|A(1,n)||\la_\rho(dn)|\left|V\left(\frac{d^3n}{N}\right)\right|\leq\frac{NP^\varepsilon}{d^3}.
\end{align*}

\subsection{Application of the Delta Method}

Now we apply our delta method to separate the oscillations. Take a fixed even function $U\in C_c^\infty(\R)$ and a fixed function $V_0\in C_c^\infty((1/2,5/2))$ such that $V_0(x)=1$ for $1\leq x\leq 2$. Applying Corollary \ref{DeltaCor} with $1\leq C^2\leq \frac{N}{d^2}$, splitting $c=ab$ depending on $(\alpha,c)$, there exists $\mathcal{C}\sim C$ and a fixed smooth function $h$ satisfying $h(x,y)\ll \delta(|x|,|y|\ll1)$ such that \begin{align*}
    S_d(N)=&\frac{1}{\mathcal{C}}\sum_a\sum_b \frac{1}{ab}\sum_m\la_\rho(m)V_0\left(\frac{d^2m}{N}\right)\sum_nA(1,n)V\left(\frac{d^3n}{N}\right)\\
    &\times\sumast_{\alpha\Mod{b}}e\left(\frac{\alpha(dn-m)}{b}\right)h\left(\frac{ab}{C},\frac{dn-m}{abC}\right).
\end{align*}
Splitting up the $b$-sum into whether $P|b$ or not, we get \begin{align}\label{SNDeltaSplit}
    S_d(N)=S_{0}+S_{1},
\end{align}
where for $v=0,1$, \begin{align}\label{SvDef}
    S_v=&\frac{1}{\mathcal{C}}\sum_a\sum_{(P^{1-v},b)=1}\frac{1}{abP^v}\sum_m\la_\rho(m)V_0\left(\frac{d^2m}{N}\right)\sum_nA(1,n)V\left(\frac{d^3n}{N}\right)\sumast_{\alpha\Mod{bP^v}}e\left(\frac{\alpha(dn-m)}{bP^v}\right)h\left(\frac{abP^\nu}{C},\frac{dn-m}{abP^vC}\right).
\end{align}

\subsection{Voronoi Summations}\label{VoronoiSection}

We first apply Voronoi summation to the $m$-sum with Lemma \ref{GL2Voronoi}. Together with a change of variable this yields \begin{align}\label{Msum}
    &\sum_m\la_\rho(m)e\left(-\frac{\alpha m}{bP^v}\right)V_0\left(\frac{d^2m}{N}\right)h\left(\frac{abP^\nu}{C},\frac{dn-m}{abP^vC}\right)\nonumber\\
    =&2\pi i^\kappa\frac{N}{bd^2P^\frac{1+v}{2}}\overline{\chi(-\alpha)}^v\chi(-b)^{1-v}\eta_\rho\sum_m\la_{\overline{\rho}}(m)e\left( \frac{\overline{\alpha P^{1-v}}m}{bP^v}\right)\int_0^\infty V_0(y)h\left(\frac{abP^\nu}{C},\frac{dn-\frac{Ny}{d^2}}{abP^vC}\right)J_{\kappa-1} \left(\frac{4\pi\sqrt{Nmy}}{bdP^\frac{1+v}{2}}\right)dy.
\end{align}

By repeated integration by parts using the derivative properties (2) in Lemma \ref{besselBoundsDerivatives} and bounding the Bessel function by $O(1)$ or (\ref{BesselAsymptotics}), we have for $j\geq0$, \begin{align*}
    &\int_0^\infty V_0(y)h\left(\frac{abP^\nu}{C},\frac{dn-\frac{Ny}{d^2}}{abP^vC}\right)J_{\kappa-1}\left(\frac{4\pi\sqrt{Nmy}}{bdP^\frac{1+v}{2}}\right)dy\\
    =&\int_0^\infty \frac{d^j}{dy^j}\left(V_0(y)y^{-\frac{\kappa+j-1}{2}}h\left(\frac{abP^\nu}{C},\frac{dn-\frac{Ny}{d^2}}{abP^vC}\right)\right)\left(\frac{bdP^\frac{1+v}{2}}{2\pi\sqrt{Nmy}}\right)^jy^\frac{\kappa+j-1}{2}J_{\kappa+j-1}\left(\frac{4\pi\sqrt{Nmy}}{bdP^\frac{1+v}{2}}\right)dy\\
    \ll&_{\kappa,j} \left(\left(1+\frac{N}{abd^2P^vC}\right)\left(\frac{bdP^\frac{1+v}{2}}{2\pi\sqrt{Nmy}}\right)\right)^j.
\end{align*}
Hence by the restriction $abd^2P^vC\leq C^2d^2\leq N$, we get arbitrary saving for the above integral unless $$m\ll\frac{P^{1-v+\varepsilon}N}{a^2C^2d^2}.$$
By another change of variable $dn-\frac{N}{d^2}y=abP^vCt$ together with (\ref{BesselAsymptotics}) and Taylor expansions, we get (\ref{Msum}) is equal to \begin{align*}
    &\frac{aC}{P^\frac{1-v}{2}}\overline{\chi(-\alpha)}^v\chi(-b)^{1-v}\eta_\rho\sum_{m\ll\frac{P^{1-v+\varepsilon}N}{a^2C^2d^2}}\la_{\overline{\rho}}(m)e\left( \frac{\overline{\alpha P^{1-v}}m}{bP^v}\right)\\
    &\times\int_{-\infty}^\infty h\left(\frac{abP^\nu}{C},t\right)V_0\left(\frac{d^2(dn-abP^vCt)}{N}\right)J_{\kappa-1 }\left(\frac{4\pi\sqrt{m(dn-abP^vCt)}}{bP^\frac{1+v}{2}}\right)dt+O\left(P^{-A}\right)\\
    =&\frac{C^3}{ab^2P^\frac{1+3v}{2}}\chi(-1)\overline{\chi(\alpha)}^v\chi(b)^{1-v}\eta_\rho\\
    &\times\sum_{\eta=\pm1}\sum_{m\ll\frac{P^{1-v+\varepsilon}N}{a^2C^2d^2}}\la_{\overline{\rho}}(m)e\left( \frac{\overline{\alpha P^{1-v}}m}{bp^v}\right)e\left(\eta\frac{2\sqrt{dmn}}{bP^\frac{1+v}{2}}\right)\left(1+\frac{(Nm)^\frac{1}{4}}{\sqrt{bd}P^\frac{1+v}{4}}\right)^{-1}W_{v}^\eta\left(m,n,b\right)+O\left(P^{-A}\right),
\end{align*}
for some $P^\varepsilon$-inert function $  W_{v}^{\eta}$.

\begin{remark}
    Here we make use of our conductor lowering function $h$ to enable an easy analysis of the Bessel transform integral. With $h$ truncating the size of the equation to be $cCP^\varepsilon$, a simple change of variable yields the asymptotics. If one were to apply the D-F-I delta method with a conductor lowering function truncating the size of the equation to be $C^2P^\varepsilon$ instead, then obtaining the asymptotics would be a bit more involved.
\end{remark}

Putting this back into $S_{v}$ we get for any $A>0$, \begin{align}
    S_{v}=&\frac{C\chi(-1)\eta_\rho}{\mathcal{C}P^\frac{1+v}{2}}\sum_{\eta=\pm}\sum_a\sum_{\substack{(P^{1-v},b)=1\\abP^\nu\ll C}}\frac{1}{b}\chi(b)^{1-v}\sum_{m\leq\frac{P^{1-v+\varepsilon}N}{a^2C^2d^2}}\la_{\overline{\rho}}(m)\sum_nA(1,n)V\left(\frac{d^3n}{N}\right)\nonumber\\
    &\times\sumast_{\alpha\Mod{bP^v}}\overline{\chi(\alpha)}^ve\left(\frac{\alpha dn+ \overline{\alpha P^{1-v}}m}{bP^v}\right)e\left(\eta\frac{2\sqrt{dmn}}{bP^\frac{1+v}{2}}\right)\left(1+\frac{(Nm)^\frac{1}{4}}{\sqrt{bd}P^\frac{1+v}{4}}\right)^{-1}  W_{v}^{\eta}\left(m,n,b\right)+O\left(P^{-A}\right).
\end{align}
As we get arbitrary saving unless $m\ll\frac{P^{1-v+\varepsilon}N}{a^2C^2d^2}$, we can perform a smooth dyadic subdivision to get for any $A>0$, \begin{align}\label{SuvDyadic}
    S_{v}\ll P^\varepsilon\sup_{\frac{1}{2}\leq M\leq \frac{P^{1-v+\varepsilon}N}{C^2d^2}}\left|S_{v,M}\right|+O\left(P^{-A}\right),
\end{align}
with \begin{align}\label{SvMDef}
    S_{v,M}=&P^{-\frac{1+v}{2}}\sum_{\eta=\pm}\sum_{a\ll\frac{P^{\frac{1-v}{2}+\varepsilon}\sqrt{N}}{\sqrt{M}Cd}}\sum_{\substack{(P^{1-v},b)=1\\abP^\nu\ll C}}\frac{1}{b}\left(1+\frac{(NM)^\frac{1}{4}}{\sqrt{bd}P^\frac{1+v}{4}}\right)^{-1}\chi(b)^{1-v}\sum_m\la_{\overline{\rho}}(m)\varphi\left(\frac{m}{M}\right)\nonumber\\
    &\times\sum_nA(1,n)V\left(\frac{d^3n}{N}\right)\sumast_{\alpha\Mod{bP^v}}\overline{\chi(\alpha)}^ve\left(\frac{\alpha dn+ \overline{\alpha P^{1-v}}m}{bP^v}\right)e\left(\eta\frac{2\sqrt{dmn}}{bP^\frac{1+v}{2}}\right)\tilde{W}_{v}^{\eta}\left(m,n,b\right),
\end{align}
with some fixed $\varphi\in C_c^\infty(\R^+)$ supported on $(1/2,5/2)$ and $$\tilde{W}_{v}^{\eta}\left(m,n,b\right)=\left(1+\frac{(Nm)^\frac{1}{4}}{\sqrt{bd}P^\frac{1+v}{4}}\right)^{-1}\left(1+\frac{(NM)^\frac{1}{4}}{\sqrt{bd}P^\frac{1+v}{4}}\right)  W_{v}^{\eta}\left(m,n,b\right)$$ is $P^\varepsilon$-inert when $m\sim M, n\sim N, b\leq CP^\varepsilon$. Here we used $\mathcal{C}\sim C$.

Next we perform Voronoi summation on the $n$-sum with Lemma \ref{GL3Voronoi}. Writing $d_b=d/(b,d)$ and $b_d=b/(b,d)$, we get \begin{align*}
    &\sum_nA(1,n)e\left(\frac{\alpha dn}{bP^v}+
    \eta\frac{2\sqrt{dmn}}{bP^\frac{1+v}{2}}\right)V\left(\frac{d^3n}{N}\right)\tilde{W}_{v}^{\eta}\left(m,n,b\right)\\
    =&b_dP^v\sum_\pm\sum_{n_1|b_dP^v}\sum_{n}\frac{A(n,n_1)}{n_1n}S\left(\overline{\alpha d_b},\pm n;\frac{b_dP^v}{n_1}\right)\psi^{\pm,\eta}_{v,d}\left(nn_1^2,m,b\right),
\end{align*}
where \begin{align*}
    \psi^{\pm,\eta}_{v,d}\left(nn_1^2,m,b\right)=\frac{1}{2\pi i}\int_{\left(-\frac{1}{2}\right)}\left(\frac{(b_dP^v)^3}{\pi^3nn_1^2}\right)^sG_k^\pm(s)\int_0^\infty V\left(\frac{d^3y}{N}\right)e\left(\eta\frac{2\sqrt{dmy}}{bP^\frac{1+v}{2}}\right)\tilde{W}_{v}^{\eta}\left(m,y,b\right)y^{-s-1}dyds
\end{align*}
with \begin{align*}
    G_k^\pm(s)=\left(\frac{\Gamma\left(\frac{2+s}{2}\right)}{\Gamma\left(\frac{1-s}{2}\right)}\mp i\frac{\Gamma\left(\frac{1+s}{2}\right)}{\Gamma\left(-\frac{s}{2}\right)}\right)\frac{\Gamma\left(\frac{k+s}{2}\right)\Gamma\left(\frac{k+1+s}{2}\right)}{\Gamma\left(\frac{k-s-1}{2}\right)\Gamma\left(\frac{k-s}{2}\right)}.
\end{align*}
By a change of variable we have, \begin{align}\label{PsiDef}
    \psi^{\pm,\eta}_{v,d}\left(nn_1^2,m,b\right)=\frac{1}{2\pi i}\int_{\left(\frac{1}{2}\right)}\left(\frac{(b_ddP^v)^3}{\pi^3Nnn_1^2}\right)^sG_k^\pm(s)\int_0^\infty V\left(y\right)e\left(\eta\frac{2\sqrt{Nmy}}{bdP^\frac{1+v}{2}}\right)\tilde{W}_{v}^{\eta}\left(m,Nd^{-3}y,b\right)y^{-s-1}dyds.
\end{align}
Putting this back into $S$, we get \begin{align}\label{SuvMBeforeIntegral}
    S_{v,M}=&P^{-\frac{1-v}{2}}\sum_\pm\sum_{\eta=\pm1}\sum_{a\ll\frac{P^{\frac{1-v}{2}+\varepsilon}\sqrt{N}}{\sqrt{M}Cd}}\sum_{\substack{(P^{1-v},b)=1\\abP^\nu\ll C}}\frac{1}{(b,d)}\left(1+\frac{(NM)^\frac{1}{4}}{\sqrt{bd}P^\frac{1+v}{4}}\right)^{-1}\chi(b)^{1-v}\nonumber\\
    &\times\sum_m\la_{\overline{\rho}}(m)\varphi\left(\frac{m}{M}\right)\sum_{n_1|b_dP^v}\sum_{n}\frac{A(n,n_1)}{n_1n}\mathcal{C}_{\chi,v,d}^{ }\left(m,\pm n;bP^v,\frac{b_dP^v}{n_1}\right)\psi^{\pm,\eta}_{v,d}\left(nn_1^2,m,b\right),
\end{align}
where the character sum is \begin{align}\label{CharSumDef}
    \mathcal{C}_{\chi,v,d}^{ }\left(m,\pm n;bP^v,\frac{b_dP^v}{n_1}\right)=\sumast_{\alpha\Mod{bP^v}}\overline{\chi(\alpha)}^ve\left( \frac{\overline{\alpha P^{1-v}}m}{bP^v}\right)S\left(\overline{\alpha d_b},\pm n;\frac{b_dP^v}{n_1}\right).
\end{align}
Note that $\mathcal{C}_{\chi,v,d}^{ }\left(w,x;y,z\right)$ is determined by $w\Mod{y}$ and $x\Mod{z}$.

Applying Lemma \ref{HugeIntegralAnalysis} on $\psi$ gives us arbitrary saving if $nn_1^2$ is sufficiently large, say for example $nn_1^2\gg (Pk)^9$. Recalling (\ref{SuvDyadic}) and (\ref{SuvMBeforeIntegral}), performing a smooth dyadic subdivision on $n\sim \tilde{N}$ and dyadic subdivisions on $n_1,b,(b,d)$, we get

\begin{thm}\label{SvAfterVDThm}
    Let $v=0,1$ and let $A,C,d,M,N>0$ such that $N\ll P^{\frac{3}{2}+\varepsilon}k^2$, $M\ll\frac{P^{1-v+\varepsilon}N}{C^2d^2}$, $1\leq C^2\leq \frac{N}{d^2}$. Let $S_{v,M}$ be as in (\ref{SvMDef}), then we have 
    \begin{align}\label{SvMDyadic2}
        S_{v,M}\ll P^\varepsilon\sum_{\eta=\pm1}\sum_\pm\sup_{\frac{1}{2}\leq\tilde{N}\leq (Pk)^9}\sum_{a\ll\frac{P^{\frac{1-v}{2}+\varepsilon}\sqrt{N}}{\sqrt{M}Cd}}\frac{1}{a}\sup_{1\leq B\ll \frac{C}{aP^v}}\sup_{1\leq B_d\ll d}\sup_{1\leq N_1\leq (Pk)^9}\left|S_{v,M,\tilde{N},N_1,B,B_d}^{\pm,\eta}\right|+O\left(P^{-A}\right),
    \end{align}
    where \begin{align}\label{SvMDyadic3}
        S_{v,M,\tilde{N},N_1,B,B_d}^{\pm,\eta}=&\frac{a}{P^\frac{1-v}{2}}\left(1+\frac{(NM)^\frac{1}{4}}{\sqrt{Bd}P^\frac{1+v}{4}}\right)^{-1}\sum_{\substack{B\leq b\leq 2B\\ B_d\leq (b,d)\leq 2B_d\\
        (P^{1-v},b)=1}}\frac{\chi(b)^{1-v}}{(b,d)}\sum_m\la_{\overline{\rho}}(m)\varphi\left(\frac{m}{M}\right)\nonumber\\
        &\times\sum_{\substack{N_1\leq n_1\leq 2N_1\\
        n_1|b_dP^v}}\sum_{n}\frac{A(n,n_1)}{nn_1}\varphi\left(\frac{n}{\tilde{N}}\right)\mathcal{C}_{\chi,v,d}^{ }\left(m,\pm n;bP^v,\frac{b_dP^v}{n_1}\right)\psi^{\pm,\eta}_{v,d}\left(nn_1^2,m,b\right),
    \end{align}
    with $\psi_{v,d}^{\pm,\eta}$ defined in (\ref{PsiDef}) and $\varphi\in C_c^\infty((1/2,5/2))$ is fixed.
\end{thm}

\subsection{Cauchy Schwartz and Poisson Summation}

We are going to treat $S_{v,M,\tilde{N},N_1,B,B_d}^{\pm,\eta}$ in (\ref{SvMDyadic3}) through an application of Cauchy Schwartz and Poisson summation. However, our treatment will vary slightly depending on whether the size of $M$ lies in the "transition range" of the $\psi$ integral, and so we split our treatment into two cases.

\subsubsection{Non-transitional Case}

We will first treat the non-transitional case, i.e. $M\ll \frac{(Bd)^2P^{1+v-\varepsilon}k^2}{N}$ or $M\gg \frac{(Bd)^2P^{1+v+\varepsilon}k^2}{N}$. This contains the generic case by our ultimate choice of parameters.

Applying Cauchy Schwartz inequality and taking out the $n$-sum, together with (\ref{DeligneBound}) we get \begin{align}
    \left(S_{v,M,\tilde{N},N_1,B,B_d}^{\pm,\eta}\right)^2\ll&\frac{a^2P^\varepsilon}{P^{1-v}\tilde{N}}\left(1+\frac{\sqrt{NM}}{BdP^\frac{1+v}{2}}\right)^{-1}\sum_n\varphi\left(\frac{n}{\tilde{N}}\right)\left|\sum_{\substack{B\leq b\leq 2B\\ B_d\leq (b,d)\leq 2B_d\\
    (P^{1-v},b)=1}}\frac{\chi(b)^{1-v}}{(b,d)}\right.\nonumber\\
    &\times \left.\sum_{\substack{N_1\leq n_1\leq 2N_1\\
    n_1|b_dP^v}}\frac{1}{n_1}\sum_m\la_{\overline{\rho}}(m)\varphi\left(\frac{m}{M}\right)\mathcal{C}_{\chi,v,d}^{ }\left(m,\pm n,y;bP^v,\frac{b_dP^v}{n_1}\right)\psi^{\pm,\eta}_{v,d}\left(nn_1^2,m,b\right)\right|^2.
\end{align}
Opening the square, we get \begin{align}\label{S2AfterCauchy}
    \left(S_{v,M,\tilde{N},N_1,B,B_d}^{\pm,\eta}\right)^2\ll&\frac{a^2P^\varepsilon}{P^{1-v}\tilde{N}}\left(1+\frac{\sqrt{NM}}{BdP^\frac{1+v}{2}}\right)^{-1}\sum_n\varphi\left(\frac{n}{\tilde{N}}\right)\sum_{\substack{B\leq b_1\leq 2B\\ B_d\leq (b_1,d)\leq 2B_d\\
    (P^{1-v},b_1)=1}}\sum_{\substack{B\leq b_2\leq 2B\\ B_d\leq (b_2,d)\leq 2B_d\\
    (P^{1-v},b_2)=1}}\frac{\chi(b_1)^{1-v}\overline{\chi(b_2)}^{1-v}}{(b_1,d)(b_2,d)}\nonumber\\
    &\times \sum_{m_1}\sum_{m_2}\la_{\overline{\rho}}(m_1)\overline{\la_{\overline{\rho}}(m_2)}\varphi\left(\frac{m_1}{M}\right)\varphi\left(\frac{m_2}{M}\right)\sum_{\substack{N_1\leq n_1\leq 2N_1\\
    n_1|b_{1,d}P^v}}\sum_{\substack{N_1\leq n_2\leq 2N_1\\
    n_2|b_{2,d}P^v}}\frac{1}{n_1n_2}\mathcal{C}_{\chi,v,d}^{ }\left(m_1,\pm n;b_1P^v,\frac{b_{1,d}P^v}{n_1}\right)\nonumber\\
    &\times \overline{\mathcal{C}_{\chi,v,d}^{ }\left(m_2,\pm n;b_2P^v,\frac{b_{2,d}P^v}{n_2}\right)}\psi^{\pm,\eta}_{v,d}\left(nn_1^2,m_1,b_2\right)\overline{\psi^{\pm,\eta}_{v,d}\left(nn_2^2,m_2,b_2\right)},
\end{align}
with $b_{j,d}=b_j/(b_j,d)$ for $j=1,2$.

Recall the definition of \begin{align*}
    \mathcal{C}_{\chi,v,d}^{ }\left(m_j,\pm n;b_jP^v,\frac{b_{j,d}P^v}{n_j}\right)=&\sumast_{\alpha_j\Mod{bP^v}}\overline{\chi(\alpha_j)}^ve\left( \frac{\overline{\alpha_j P^{1-v}}m_j}{b_jP^v}\right)S\left(\overline{\alpha_j d_{b_j}},\pm n;\frac{b_{j,d}P^v}{n_j}\right)\\
    =&\sumast_{\alpha_j\Mod{bP^v}}\overline{\chi(\alpha_j)}^ve\left( \frac{\overline{\alpha_j P^{1-v}}m_j}{b_jP^v}\right)\sumast_{\beta_j\Mod{\frac{b_{j,d}P^v}{n_j}}}e\left(n_j\frac{\overline{\alpha_jd_{b_j}\beta_j}\pm \beta_j n}{b_{j,d}P^v}\right)
\end{align*}
for $j=1,2$ as in (\ref{CharSumDef}). Performing Poisson summation on the $n$-sum in (\ref{S2AfterCauchy}) gives \begin{align*}
    &\sum_{n}\varphi\left(\frac{n}{\tilde{N}}\right)e\left(\pm \frac{n}{P^v}\left(\frac{n_1\beta_1}{b_{1,d}}-\frac{n_2\beta_2}{b_{2,d}}\right)\right)\psi^{\pm,\eta}_{v,d}\left(nn_1^2,m_1,b_1\right)\psi^{\pm,\eta}_{v,d}\left(nn_2^2,m_2,b_2\right)\\
    =&\sum_{\gamma\Mod{\frac{b_{1,d}b_{2,d}P^{2v}}{n_1n_2}}}e\left(\pm \frac{\gamma(n_1\beta_1b_{2,d}-n_2\beta_2b_{1,d})}{b_{1,d}b_{2,d}P^{2v}}\right)\sum_n\int_0^\infty \varphi\left(\frac{\gamma+\frac{b_{1,d}b_{2,d}P^{2v}}{n_1n_2}y}{\tilde{N}}\right)\\
    &\times\psi^{\pm,\eta}_{v,d}\left(n_1^2\left(\gamma+\frac{b_{1,d}b_{2,d}P^{2v}}{n_1n_2}y\right),m_1,b_1\right)\psi^{\pm,\eta}_{v,d}\left(n_2^2\left(\frac{b_{1,d}b_{2,d}P^{2v}}{n_1n_2}+y\right),m_2,b_2\right)e(-ny)dy\\
    =&\frac{\tilde{N}n_1n_2}{b_{1,d}b_{2,d}P^{2v}}\sum_n\sum_{\gamma\Mod{\frac{b_{1,d}b_{2,d}P^{2v}}{n_1n_2}}}e\left(\frac{n_1n_2n\gamma}{b_{1,d}b_{2,d}P^{2v}}\pm \frac{\gamma(n_1\beta_1b_{2,d}-n_2\beta_2b_{1,d})}{b_{1,d}b_{2,d}P^{2v}}\right)\\
    &\times\int_0^\infty \varphi\left(y\right)\psi^{\pm,\eta}_{v,d}\left(\tilde{N}n_1^2w,m_1,b_1\right)\psi^{\pm,\eta}_{v,d}\left(\tilde{N}n_2^2w,m_2,b_2\right)e\left(-\frac{\tilde{N}n_1n_2nw}{b_{1,d}b_{2,d}P^{2v}}\right)dw\\
    =&\tilde{N}\sum_n\delta\left(n\equiv\mp\beta_1\frac{b_{2,d}P^v}{n_2}\pm\beta_2\frac{b_{1,d}P^v}{n_1}\Mod{\frac{b_{1,d}b_{2,d}P^{2v}}{n_1n_2}}\right)\\
    &\times\int_0^\infty \varphi\left(w\right)\psi^{\pm,\eta}_{v,d}\left(\tilde{N}n_1^2w,m_1,b_1\right)\psi^{\pm,\eta}_{v,d}\left(\tilde{N}n_2^2w,m_2,b_2\right)e\left(-\frac{\tilde{N}n_1n_2nw}{b_{1,d}b_{2,d}P^{2v}}\right)dw.
\end{align*}
We remind the reader that $\overline{\beta_j}$ is the inverse defined mod $\frac{b_{3-j,d}P^v}{n_{3-j}}$ for $j=1,2$.

Putting this back into (\ref{S2AfterCauchy}), we get \begin{align}\label{S2AfterPoisson}
    \left(S_{v,M,\tilde{N},N_1,B,B_d}^{\pm,\eta}\right)^2\ll&\frac{a^2P^\varepsilon}{P^{1-v}}\left(1+\frac{\sqrt{NM}}{BdP^\frac{1+v}{2}}\right)^{-1}\sum_{n}\sum_{\substack{B\leq b_1\leq 2B\\ B_d\leq (b_1,d)\leq 2B_d\\
    (P^{1-v},b_1)=1}}\sum_{\substack{B\leq b_2\leq 2B\\ B_d\leq (b_2,d)\leq 2B_d\\
    (P^{1-v},b_2)=1}}\frac{\chi(b_1)^{1-v}\overline{\chi(b_2)}^{1-v}}{(b_1,d)(b_2,d)}\nonumber\\
    &\times \sum_{m_1}\sum_{m_2}\la_{\overline{\rho}}(m_1)\overline{\la_{\overline{\rho}}(m_2)}\varphi\left(\frac{m_1}{M}\right)\varphi\left(\frac{m_2}{M}\right)\sum_{\substack{N_1\leq n_1\leq 2N_1\\
    n_1|b_{1,d}P^v}}\sum_{\substack{N_1\leq n_2\leq 2N_1\\
    n_2|b_{2,d}P^v}}\frac{1}{n_1n_2}\mathcal{D}\mathcal{J},
\end{align}
where \begin{align}
    \mathcal{D}=&\sumast_{\alpha_1\Mod{b_1P^v}}\overline{\chi(\alpha_1)}^ve\left( \frac{\overline{\alpha_1 P^{1-v}}m_1}{b_1P^v}\right)\sumast_{\beta_1\Mod{\frac{b_{1,d}P^v}{n_1}}}e\left(n_1\frac{\overline{\alpha_1d_{b_1}\beta_1}}{b_{1,d}P^v}\right)\nonumber\\
    &\times \sumast_{\alpha_2\Mod{b_2P^v}}\chi(\alpha_2)^ve\left(- \frac{\overline{\alpha_2 P^{1-v}}m_2}{b_2P^v}\right)\sumast_{\beta_2\Mod{\frac{b_{2,d}P^v}{n_2}}}e\left(-n_2\frac{\overline{\alpha_2d_{b_2}\beta_2}}{b_{2,d}P^v}\right)\nonumber\\
    &\times \delta\left(n\equiv\mp\beta_1\frac{b_{2,d}P^v}{n_2}\pm\beta_2\frac{b_{1,d}P^v}{n_1}\Mod{\frac{b_{1,d}b_{2,d}P^{2v}}{n_1n_2}}\right)
\end{align}
and \begin{align}\label{UltIntegralDef}
    \mathcal{J}=\int_0^\infty \varphi\left(w\right)\psi^{\pm,\eta}_{v,d}\left(\tilde{N}n_1^2w,m_1,b_1\right)\overline{\psi^{\pm,\eta}_{v,d}\left(\tilde{N}n_2^2w,m_2,b_2\right)}e\left(-\frac{\tilde{N}n_1n_2nw}{b_{1,d}b_{2,d}P^{2v}}\right)dw.
\end{align}

\subsubsection{Transitional case}

For the transitional case $$\frac{(Bd)^2P^{1+v-\varepsilon}k^2}{N}\ll M\ll \frac{(Bd)^2P^{1+v+\varepsilon}k^2}{N},$$
we need to analyse $\psi$ more carefully. Recalling (\ref{PsiDef}) and applying Lemma \ref{HugeIntegralAnalysis} on (\ref{SvMDyadic3}) gives us \begin{align*}
    S_{v,M,\tilde{N},N_1,B,B_d}^{\pm,\eta}=&\frac{a}{P^\frac{1-v}{2}}\left(1+\frac{(NM)^\frac{1}{4}}{\sqrt{Bd}P^\frac{1+v}{4}}\right)^{-1}\sum_{\substack{B\leq b\leq 2B\\ B_d\leq (b,d)\leq 2B_d\\
    (P^{1-v},b)=1}}\frac{\chi(b)^{1-v}}{(b,d)}\sum_m\la_{\overline{\rho}}(m)\varphi\left(\frac{m}{M}\right)\nonumber\\
    &\times\sum_{\substack{N_1\leq n_1\leq 2N_1\\
    n_1|b_dP^v}}\sum_{n}\frac{A(n,n_1)}{nn_1}\varphi\left(\frac{n}{\tilde{N}}\right)\mathcal{C}_{\chi,v,d}^{ }\left(m,\pm n;bP^v,\frac{b_dP^v}{n_1}\right)\\
    &\times\left(I_0\left(\eta\frac{2\sqrt{Nm}}{bdP^\frac{1+v}{2}},\frac{(b_ddP^v)^3}{N\tilde{N}n_1^2w}\right)+\sum_{\sigma=\pm1}\sum_{P^{-\varepsilon}\leq \sigma T\leq k^\frac{1}{3}P^{-\varepsilon} \text{ Dyadic}}I_{T}\left(\eta\frac{2\sqrt{Nm}}{bdP^\frac{1+v}{2}},\frac{(b_ddP^v)^3}{N\tilde{N}n_1^2w}\right)\right),
\end{align*}
where \begin{align}
    I_0(x,y)=k^2\int_0^\infty W_{0,x,y}^\pm\left(k^\frac{1}{3}P^{-\varepsilon}\left(t-\frac{k}{|x|}\right)\right)e\left(\mp\frac{xt}{2\pi}\log\left(\frac{|x|y(k^2+x^2t^2)}{8\pi et}\right)\mp\frac{k}{\pi}\arctan\left(\frac{xt}{k}\right)\right)dt
\end{align}
and for $\sigma T> 0$, \begin{align}
    I_{T}(x,y)=&k^\frac{3}{2}\sqrt{|T|}W_{0,x,y,\sigma T}^\pm\left(\frac{4\pi}{|x|^3y}\left(1+\sigma\sqrt{1-\left(\frac{x^2yk}{4\pi}\right)^2}\right)\right)\nonumber\\
    &\times e\left(\pm\frac{2}{x^2y}\left(1+\sigma\sqrt{1-\left(\frac{x^2yk}{4\pi}\right)^2}\right)\mp\frac{k}{\pi}\arctan\left(\frac{4\pi}{x^2yk}\left(1+\sigma\sqrt{1-\left(\frac{x^2yk}{4\pi}\right)^2}\right)\right)\right)
\end{align}
for some functions $W_{0,x,y}^\pm(t)$ supported on $(-1,1)$, $W_{0,x,y,\sigma T}^\pm(t)$ supported on $\left(\frac{k}{|x|}-\frac{1}{T},\frac{k}{|x|}+\frac{1}{T}\right)$, $TP^\varepsilon$-inert in $t$ and $P^\varepsilon$-inert in $x,y$. Applying the Cauchy Schwartz inequality as above but also taking out $\displaystyle\sum_{\sigma=\pm1}$ and the $T$-sum, we get \begin{align*}
    &\left(S_{v,M,\tilde{N},N_1,B,B_d}^{\pm,\eta}\right)^2\nonumber\\
    \ll&\frac{a^2P^\varepsilon}{P^{1-v}\tilde{N}}\left(1+\frac{\sqrt{NM}}{BdP^\frac{1+v}{2}}\right)^{-1}\sum_{\sigma=\pm1}\sum_{\substack{T=0 \text{ or }\\ P^{-\varepsilon}\leq \sigma T\leq k^\frac{1}{3}P^{-\varepsilon} \text{ Dyadic}}}\sum_n\varphi\left(\frac{n}{\tilde{N}}\right)\\
    &\times \left|\sum_{\substack{B\leq b\leq 2B\\ B_d\leq (b,d)\leq 2B_d\\
    (P^{1-v},b)=1}}\frac{\chi(b)^{1-v}}{(b,d)}\sum_{\substack{N_1\leq n_1\leq 2N_1\\
    n_1|b_dP^v}}\frac{1}{n_1}\sum_m\la_{\overline{\rho}}(m)\varphi\left(\frac{m}{M}\right)\mathcal{C}_{\chi,v,d}^{ }\left(m,\pm n,y;bP^v,\frac{b_dP^v}{n_1}\right)I_{T}\right|^2.
\end{align*}
Opening the square, applying Poisson summation and performing the exact same procedure as above yields \begin{align}\label{S2TransMAfterPoisson}
    \left(S_{v,M,\tilde{N},N_1,B,B_d}^{\pm,\eta}\right)^2\ll&\sup_{\sigma=\pm1}\sup_{\substack{T=0 \text{ or }\\ P^{-\varepsilon}\leq T\leq k^\frac{1}{3}P^{-\varepsilon}}}\frac{a^2P^\varepsilon}{P^{1-v}}\left(1+\frac{\sqrt{NM}}{BdP^\frac{1+v}{2}}\right)^{-1}\sum_{n}\sum_{\substack{B\leq b_1\leq 2B\\ B_d\leq (b_1,d)\leq 2B_d\\
    (P^{1-v},b_1)=1}}\sum_{\substack{B\leq b_2\leq 2B\\ B_d\leq (b_2,d)\leq 2B_d\\
    (P^{1-v},b_2)=1}}\frac{\chi(b_1)^{1-v}\overline{\chi(b_2)}^{1-v}}{(b_1,d)(b_2,d)}\nonumber\\
    &\times \sum_{m_1}\sum_{m_2}\la_{\overline{\rho}}(m_1)\overline{\la_{\overline{\rho}}(m_2)}\varphi\left(\frac{m_1}{M}\right)\varphi\left(\frac{m_2}{M}\right)\sum_{\substack{N_1\leq n_1\leq 2N_1\\
    n_1|b_{1,d}P^v}}\sum_{\substack{N_1\leq n_2\leq 2N_1\\
    n_2|b_{2,d}P^v}}\frac{1}{n_1n_2}\mathcal{D}\mathcal{J}_{\sigma T},
\end{align}
with $\mathcal{D}$ as in (\ref{CharSumDef}) and \begin{align}\label{UltIntegralDefTransM}
    \mathcal{J}_{\sigma T}=\int_0^\infty \varphi\left(w\right)I_{\sigma T}\left(\eta\frac{2\sqrt{Nm_1}}{b_1dP^\frac{1+v}{2}},\frac{(b_{1,d}dP^v)^3}{N\tilde{N}n_1^2w}\right)\overline{I_{\sigma T}\left(\eta\frac{2\sqrt{Nm_2}}{b_2dP^\frac{1+v}{2}},\frac{(b_{2,d}dP^v)^3}{N\tilde{N}n_2^2w}\right)}e\left(-\frac{\tilde{N}n_1n_2nw}{b_{1,d}b_{2,d}P^{2v}}\right)dw.
\end{align}

\subsection{Character Sum}

Now we analyse and simplify the character sum $\mathcal{D}$. The treatment will be slightly different depending on whether $\chi=\chi_0$ the trivial nebentypus or not when $v=1$. In the rest of the paper, we will focus on $\chi=\chi_0$ or $v=0$ as the main case, and write out the adjustment needed for nontrivial nebentypus when $v=1$.

For $\chi=\chi_0$ or $v=0$, summing over $\alpha_j$, we get for $j=1,2$, \begin{align*}
    \sumast_{\alpha_j\Mod{b_jP^v}}e\left(\pm \frac{\overline{\alpha_j P^{1-v}}m_j}{b_jP^v}\pm n_j\frac{\overline{\alpha_jd_{b_j}\beta_j}}{b_{j,d}P^v}\right)=&\sumast_{\alpha\Mod{b_jP^v}}e\left(\pm\frac{\alpha( \overline{P^{1-v}}m_j+\overline{d_{b_j}\beta_j}n_j(b_j,d))}{b_jP^v}\right)\\
    =&\sum_{u_{j,1}u_{j,2}=b_jP^v}\mu\left(u_{j,1}\right)\sum_{\alpha\Mod{u_{j,2}}}e\left(\pm\frac{\alpha( \overline{P^{1-v}}m_j+\overline{d_{b_j}\beta_j}n_j(b_j,d))}{u_{j,2}}\right)\\
    =&\sum_{u_{j,1}u_{j,2}=b_jP^v}\mu\left(u_{j,1}\right)u_{j,2}\delta\left(m_j\equiv  \overline{d_{b_j}\beta_j}P^{1-v}n_j(b_j,d)\Mod{u_{j,2}}\right).
\end{align*}
Inserting this into $\mathcal{D}$ and bounding everything by absolute value, we get \begin{align}\label{CharSumBound}
    \mathcal{D}\ll &\sum_{u_1|b_1P^v}\sum_{u_2|b_2P^v}u_1u_2\sumast_{\beta_1\Mod{\frac{b_{1,d}P^v}{n_1}}}\sumast_{\beta_2\Mod{\frac{b_{2,d}P^v}{n_2}}}\nonumber\\
    &\times\delta\left(m_1\equiv  \overline{d_{b_1}\beta_1}P^{1-v}n_1(b_1,d)\Mod{u_1}, m_2\equiv  \overline{d_{b_2}\beta_2}P^{1-v}n_2(b_2,d)\Mod{u_2},\right.\nonumber\\
    &\hspace{1cm} \left.n\equiv\mp\beta_1\frac{b_{2,d}P^v}{n_2}\pm\beta_2\frac{b_{1,d}P^v}{n_1}\Mod{\frac{b_{1,d}b_{2,d}P^{2v}}{n_1n_2}}\right).
\end{align}

For nontrivial nebentypus $\chi\neq\chi_0$ and $v=1$, summing over $\alpha_1$, we get by the evaluation of Gauss sum \begin{align*}
    \sumast_{\alpha_1\Mod{b_1P}}\overline{\chi(\alpha_1)}e\left(\frac{\overline{\alpha_1}m_1}{b_1P}+n_1\frac{\overline{\alpha_1d_{b_1}\beta_1}}{b_{1,d}P}\right)\ll\sqrt{b_1P}\ll\sqrt{BP}.
\end{align*}
Doing the same thing on the sum of $\alpha_2$ and bounding everything by absolute value gives us \begin{align}\label{charsumchibound}
    \mathcal{D}\ll &BP\sumast_{\beta_1\Mod{\frac{b_{1,d}P}{n_1}}}\sumast_{\beta_2\Mod{\frac{b_{2,d}P}{n_2}}}\delta\left(n\equiv\mp\beta_1\frac{b_{2,d}P}{n_2}\pm\beta_2\frac{b_{1,d}P}{n_1}\Mod{\frac{b_{1,d}b_{2,d}P^{2}}{n_1n_2}}\right).
\end{align}

\section{Final Bound: Finishing the Proof of Theorem \ref{mainthm}}\label{MainthmSection2}

For $\chi=\chi_0$ or $v=0$, recalling (\ref{S2AfterPoisson}), bounding everything by absolute value and applying (\ref{DeligneBound}) and (\ref{CharSumBound}) gives us \begin{align}\label{FinalBoundSetup}
    \left(S_{v,M,\tilde{N},N_1,B,B_d}^{\pm,\eta}\right)^2\ll&\frac{a^2P^\varepsilon}{B_d^2P^{1-v}N_1^2}\left(1+\frac{\sqrt{NM}}{BdP^\frac{1+v}{2}}\right)^{-1}\sum_{\substack{B\leq b_1\leq 2B\\ B_d\leq (b_1,d)\leq 2B_d\\
    (P^{1-v},b_1)=1}}\sum_{\substack{B\leq b_2\leq 2B\\ B_d\leq (b_2,d)\leq 2B_d\\
    (P^{1-v},b_2)=1}}\sum_{u_1|b_1P^v}\sum_{u_2|b_2P^v}u_1u_2\nonumber\\
    &\times\sum_{\substack{N_1\leq n_1\leq 2N_1\\
    n_1|b_{1,d}P^v}}\sum_{\substack{N_1\leq n_2\leq 2N_1\\
    n_2|b_{2,d}P^v}}\sumast_{\beta_1\Mod{\frac{b_{1,d}P^v}{n_1}}}\sumast_{\beta_2\Mod{\frac{b_{2,d}P^v}{n_2}}}\sum_{n}\sum_{m_1\sim M}\sum_{m_2\sim M}\delta\left((*)\right)\left|\mathcal{J}\right|,
\end{align}
where $(*)$ is the set of congruences \begin{align}\label{*}
    &m_1\equiv  \overline{d_{b_1}\beta_1}P^{1-v}n_1(b_1,d)\Mod{u_1},\nonumber\\
    &m_2\equiv  \overline{d_{b_2}\beta_2}P^{1-v}n_2(b_2,d)\Mod{u_2},n\equiv\mp\beta_1\frac{b_{2,d}P^v}{n_2}\pm\beta_2\frac{b_{1,d}P^v}{n_1}\Mod{\frac{b_{1,d}b_{2,d}P^{2v}}{n_1n_2}}.
\end{align}
Separating the case $\left|b_1m_1(b_2,d)^3n_2^2-b_2m_2(b_1,d)^3n_1^2\right|=0$ and performing dyadic subdivision for \newline $\left|b_1m_1(b_2,d)^3n_2^2-b_2m_2(b_1,d)^3n_1^2\right|\neq 0$, we get for any $A>0$, \begin{align}\label{SuvAfterCauchy0split}
    \left(S_{v,M,\tilde{N},N_1,B,B_d}^{\pm,\eta}\right)^2\ll&\frac{a^2P^\varepsilon}{B_d^2P^{1-v}N_1^2}\left(1+\frac{\sqrt{NM}}{BdP^\frac{1+v}{2}}\right)^{-1}\sup_{0\leq H\leq BB_d^3MP^\varepsilon}\mathcal{S}_H,
\end{align}
where \begin{align}
    \mathcal{S}_H:=&\mathcal{S}_{u,v,M,\tilde{N},N_1,B,B_d,H}^{\pm,\eta}\nonumber\\
    =&\sum_{\substack{B\leq b_1\leq 2B\\ B_d\leq (b_1,d)\leq 2B_d\\
    (P^{1-v},b_1)=1}}\sum_{\substack{B\leq b_2\leq 2B\\ B_d\leq (b_2,d)\leq 2B_d\\
    (P^{1-v},b_2)=1}}\sum_{u_1|b_1P^v}\sum_{u_2|b_2P^v}u_1u_2\sum_{\substack{N_1\leq n_1\leq 2N_1\\
    n_1|b_{1,d}P^v}}\sum_{\substack{N_1\leq n_2\leq 2N_1\\
    n_2|b_{2,d}P^v}}\sumast_{\beta_1\Mod{\frac{b_{1,d}P^v}{n_1}}}\sumast_{\beta_2\Mod{\frac{b_{2,d}P^v}{n_2}}}\nonumber\\
    &\times\sum_n\sum_{m_1\sim M}\sum_{m_2\sim M}\delta\left(H\leq \left|b_1m_1(b_2,d)^3n_2^2-b_2m_2(b_1,d)^3n_1^2\right|\leq 2H, (*)\right)\left|\mathcal{J}\right|.
\end{align}

Similarly for $\chi\neq\chi_0$ and $v=1$, (\ref{charsumchibound}) gives us \begin{align}\label{SuvchiAfterCauchy0split}
    \left(S_{v,M,\tilde{N},N_1,B,B_d}^{\pm,\eta}\right)^2\ll&\frac{a^2P^\varepsilon}{B_d^2N_1^2}\left(1+\frac{\sqrt{NM}}{BdP}\right)^{-1}\sup_{0\leq H\leq BB_d^3MP^\varepsilon}\mathcal{S}_{\chi,H},
\end{align}
where \begin{align}
    \mathcal{S}_{\chi,H}:=&\mathcal{S}_{u,v,M,\tilde{N},N_1,B,B_d,\chi,H}^{\pm,\eta}\nonumber\\
    =&BP\sum_{\substack{B\leq b_1\leq 2B\\ B_d\leq (b_1,d)\leq 2B_d}}\sum_{\substack{B\leq b_2\leq 2B\\ B_d\leq (b_2,d)\leq 2B_d}}\sum_{\substack{N_1\leq n_1\leq 2N_1\\
    n_1|b_{1,d}P}}\sum_{\substack{N_1\leq n_2\leq 2N_1\\
    n_2|b_{2,d}P}}\sumast_{\beta_1\Mod{\frac{b_{1,d}P}{n_1}}}\sumast_{\beta_2\Mod{\frac{b_{2,d}P}{n_2}}}\nonumber\\
    &\times\sum_n\sum_{m_1\sim M}\sum_{m_2\sim M}\delta\left(H\leq \left|b_1m_1(b_2,d)^3n_2^2-b_2m_2(b_1,d)^3n_1^2\right|\leq 2H, (**)\right)\left|\mathcal{J}\right|
\end{align}
and $(**)$ is the set of congruences \begin{align}\label{**}
    n\equiv\mp\beta_1\frac{b_{2,d}P}{n_2}\pm\beta_2\frac{b_{1,d}P}{n_1}\Mod{\frac{b_{1,d}b_{2,d}P^{2}}{n_1n_2}}.
\end{align}

For the rest of the paper we focus on the case $\chi=\chi_0$ or $v=0$ without explicitly stating, and then we write up the adjustment needed for the case $\chi\neq\chi_0$ and $v=1$ afterwards.

\subsection{First Case}

We first deal with the case of \begin{align*}
    \frac{(Bd)^2P^{1+v+\varepsilon}}{N}\ll M\ll \frac{(Bd)^2P^{1+v-\varepsilon}k^2}{N}.
\end{align*}

For any $A>0$, recalling (\ref{PsiDef}) and (\ref{UltIntegralDef}), we apply Lemma \ref{HugeIntegralAnalysis2} on $\mathcal{J}$ to get the following:

If $\frac{\sqrt{N}\left|b_1m_1(b_2,d)^3n_2^2-b_2m_2(b_1,d)^3n_1^2\right|}{B^2B_d^3dN_1^2\sqrt{M}P^\frac{1+v}{2}}\ll P^\varepsilon$, \begin{align}\label{IntegralConclusionSmallMSmallH}
    \mathcal{J}\ll&\frac{B_d^3N\tilde{N}N_1^2}{(BdP^v)^3}P^\varepsilon\delta\left(\tilde{N}\sim \frac{B^2d^2\sqrt{M}k^2}{B_d^3P^\frac{1-5v}{2}\sqrt{N}N_1^2}, |n|\ll\frac{B_dP^{\frac{1-v}{2}+\varepsilon}\sqrt{N}}{d^2\sqrt{M}k^2}\right)+P^{-A}\nonumber\\
    \ll&\frac{\sqrt{MN}k^2}{BdP^\frac{1+v}{2}}P^\varepsilon\delta\left(|n|\ll\frac{B_dP^{\frac{1-v}{2}+\varepsilon}\sqrt{N}}{d^2\sqrt{M}k^2}\right)+P^{-A};
\end{align}
and if $\frac{\sqrt{N}\left|b_1m_1(b_2,d)^3n_2^2-b_2m_2(b_1,d)^3n_1^2\right|}{B^2B_d^3dN_1^2\sqrt{M}P^\frac{1+v}{2}}\gg P^\varepsilon$, 
\begin{align}\label{IntegralConclusionSmallMLargeH}
    \mathcal{J}\ll&\frac{B_d^3N\tilde{N}N_1^2}{(BdP^v)^3}\left(\frac{\sqrt{N}\left|b_1m_1(b_2,d)^3n_2^2-b_2m_2(b_1,d)^3n_1^2\right|}{B^2B_d^3dN_1^2\sqrt{M}P^\frac{1+v}{2}}\right)^{-\frac{1}{2}}P^\varepsilon\nonumber\\
    &\times \delta\left(\tilde{N}\sim \frac{B^2d^2\sqrt{M}k^2}{B_d^3P^\frac{1-5v}{2}\sqrt{N}N_1^2}, |n|\sim \frac{N\left|b_1m_1(b_2,d)^3n_2^2-b_2m_2(b_1,d)^3n_1^2\right|}{B^2B_d^2d^3N_1^2P^vMk^2}\right)+P^{-A}\nonumber\\
    \ll&\frac{B_d^\frac{3}{2}N_1N^\frac{1}{4}M^\frac{3}{4}k^2}{\sqrt{d}P^{\frac{1+v}{4}}}\left|b_1m_1(b_2,d)^3n_2^2-b_2m_2(b_1,d)^3n_1^2\right|^{-\frac{1}{2}}P^\varepsilon\nonumber\\
    &\times \delta\left(|n|\sim \frac{N\left|b_1m_1(b_2,d)^3n_2^2-b_2m_2(b_1,d)^3n_1^2\right|}{B^2B_d^2d^3N_1^2P^vMk^2}\right)+P^{-A}.
\end{align}

\subsubsection{Sub Case: Diagonal \texorpdfstring{$n=0$}{n=0}}

By (\ref{IntegralConclusionSmallMLargeH}), the sub case $n=0$ only occurs when $H\ll\frac{B^2B_d^3dN_1^2\sqrt{M}P^{\frac{1+v}{2}+\varepsilon}}{\sqrt{N}}$ including $H=0$, and in such ranges of $H$ (\ref{IntegralConclusionSmallMSmallH}) gives us the contribution to $\mathcal{S}_H$ is \begin{align*}
    \mathcal{S}_{H,0}:=&\sum_{\substack{B\leq b_1\leq 2B\\ B_d\leq (b_1,d)\leq 2B_d\\
    (P^{1-v},b_1)=1}}\sum_{\substack{B\leq b_2\leq 2B\\ B_d\leq (b_2,d)\leq 2B_d\\
    (P^{1-v},b_2)=1}}\sum_{u_1|b_1P^v}\sum_{u_2|b_2P^v}u_1u_2\sum_{\substack{N_1\leq n_1\leq 2N_1\\
    n_1|b_{1,d}P^v}}\sum_{\substack{N_1\leq n_2\leq 2N_1\\
    n_2|b_{2,d}P^v}}\sumast_{\beta_1\Mod{\frac{b_{1,d}P^v}{n_1}}}\sumast_{\beta_2\Mod{\frac{b_{2,d}P^v}{n_2}}}\nonumber\\
    &\times\sum_n\sum_{m_1\sim M}\sum_{m_2\sim M}\delta\left(H\leq \left|b_1m_1(b_2,d)^3n_2^2-b_2m_2(b_1,d)^3n_1^2\right|\leq 2H, n=0, (*)\right)\left|\mathcal{J}\right|\\
    \ll& \frac{\sqrt{MN}k^2}{BdP^\frac{1+v}{2}}\sum_{\substack{B\leq b_1\leq 2B\\ B_d\leq (b_1,d)\leq 2B_d\\
    (P^{1-v},b_1)=1}}\sum_{\substack{B\leq b_2\leq 2B\\ B_d\leq (b_2,d)\leq 2B_d\\
    (P^{1-v},b_2)=1}}\sum_{u_1|b_1P^v}\sum_{u_2|b_2P^v}u_1u_2\sum_{\substack{N_1\leq n_1\leq 2N_1\\
    n_1|b_{1,d}P^v}}\sum_{\substack{N_1\leq n_2\leq 2N_1\\
    n_2|b_{2,d}P^v}}\sumast_{\beta_1\Mod{\frac{b_{1,d}P^v}{n_1}}}\sumast_{\beta_2\Mod{\frac{b_{2,d}P^v}{n_2}}}\\
    &\times\sum_n\sum_{m_1\sim M}\sum_{m_2\sim M}\delta\left(H\leq \left|b_1m_1(b_2,d)^3n_2^2-b_2m_2(b_1,d)^3n_1^2\right|\leq 2H, n=0, (*)\right)+P^{-A}
\end{align*}
for any $A>0$. Recall $(*)$ is the set of congruences in (\ref{*}). Together with $n=0$, $(*)$ gives us \begin{align*}
    \beta_1\frac{b_{2,d}P^v}{n_2}\equiv\beta_2\frac{b_{1,d}P^v}{n_1}\Mod{\frac{b_{1,d}b_{2,d}P^{2v}}{n_1n_2}},
\end{align*}
which implies $\frac{b_{1,d}P^v}{n_1}=\frac{b_{2,d}P^v}{n_2}$ and $\beta_1=\beta_2$, since $\beta_1,\beta_2$ are coprime to $\frac{b_{1,d}P^v}{n_1}$ and $\frac{b_{2,d}P^v}{n_2}$ respectively. As a result, given $n_1,n_2,b_1,\beta_1,u_1,u_2$, then $b_2,\beta_2$ are uniquely determined, $m_1$ is determined mod $u_1$ and $m_2$ is determined mod $u_2$ with the restriction \begin{align*}
    &H\leq \left|b_1m_1(b_2,d)^3n_2^2-b_2m_2(b_1,d)^3n_1^2\right|\leq 2H\\
    \Rightarrow&\left|m_2-\frac{b_1m_1(b_2,d)^3n_2^2}{b_2(b_1,d)^3n_1^2}\right|\ll\frac{H}{BB_d^3N_1^2}\Rightarrow\left|m_2-\frac{m_1(b_2,d)^2n_2}{(b_1,d)^2n_1}\right|\ll\frac{H}{BB_d^3N_1^2}.
\end{align*}
This yields \begin{align}\label{SHSmallM0Bound}
     \mathcal{S}_{H,0}\ll& P^\varepsilon\frac{\sqrt{MN}k^2}{BdP^\frac{1+v}{2}}\frac{B}{B_d}\frac{BP^v}{B_dN_1}\sup_{u_1,u_2\ll BP^v}u_1u_2\left(\frac{M}{u_1}+1\right)\left(\min\left\{\frac{H}{u_2BB_d^3N_1^2},\frac{M}{u_2}\right\}+1\right)\nonumber\\
     \ll&P^\varepsilon\frac{B\sqrt{MN}k^2}{B_d^2dN_1P^\frac{1-v}{2}}\left(M+BP^v\right)\left(\frac{Bd\sqrt{M}P^\frac{1+v}{2}}{\sqrt{N}}+BP^v\right).
\end{align}

Similarly for the case $\chi\neq\chi_0, v=1$, we have the contribution of $n=0$ to $\mathcal{S}_{\chi,H}$ is \begin{align*}
    \mathcal{S}_{\chi,H,0}:=&BP\sum_{\substack{B\leq b_1\leq 2B\\ B_d\leq (b_1,d)\leq 2B_d}}\sum_{\substack{B\leq b_2\leq 2B\\ B_d\leq (b_2,d)\leq 2B_d}}\sum_{\substack{N_1\leq n_1\leq 2N_1\\
    n_1|b_{1,d}P}}\sum_{\substack{N_1\leq n_2\leq 2N_1\\
    n_2|b_{2,d}P}}\sumast_{\beta_1\Mod{\frac{b_{1,d}P}{n_1}}}\sumast_{\beta_2\Mod{\frac{b_{2,d}P}{n_2}}}\nonumber\\
    &\times\sum_n\sum_{m_1\sim M}\sum_{m_2\sim M}\delta\left(H\leq \left|b_1m_1(b_2,d)^3n_2^2-b_2m_2(b_1,d)^3n_1^2\right|\leq 2H, n=0, (**)\right)\left|\mathcal{J}\right|\\
    \ll& \frac{\sqrt{MN}k^2}{d}\sum_{\substack{B\leq b_1\leq 2B\\ B_d\leq (b_1,d)\leq 2B_d}}\sum_{\substack{B\leq b_2\leq 2B\\ B_d\leq (b_2,d)\leq 2B_d}}\sum_{\substack{N_1\leq n_1\leq 2N_1\\
    n_1|b_{1,d}P}}\sum_{\substack{N_1\leq n_2\leq 2N_1\\
    n_2|b_{2,d}P}}\sumast_{\beta_1\Mod{\frac{b_{1,d}P}{n_1}}}\sumast_{\beta_2\Mod{\frac{b_{2,d}P}{n_2}}}\\
    &\times\sum_n\sum_{m_1\sim M}\sum_{m_2\sim M}\delta\left(H\leq \left|b_1m_1(b_2,d)^3n_2^2-b_2m_2(b_1,d)^3n_1^2\right|\leq 2H, n=0, (**)\right)+P^{-A}
\end{align*}
for any $A>0$. Same as above, $(**)$ together with $n=0$ gives us $\frac{b_{1,d}P^v}{n_1}=\frac{b_{2,d}P^v}{n_2}$ and $\beta_1=\beta_2$. Hence given $n_1,n_2,b_1,\beta_1, m_1$, then $b_2,\beta_2$ are uniquely determined, $m_2$ is restricted to satisfy \begin{align*}
    &H\leq \left|b_1m_1(b_2,d)^3n_2^2-b_2m_2(b_1,d)^3n_1^2\right|\leq 2H\\
    \Rightarrow&\left|m_2-\frac{b_1m_1(b_2,d)^3n_2^2}{b_2(b_1,d)^3n_1^2}\right|\ll\frac{H}{BB_d^3N_1^2}\Rightarrow\left|m_2-\frac{m_1(b_2,d)^2n_2}{(b_1,d)^2n_1}\right|\ll\frac{H}{BB_d^3N_1^2}.
\end{align*}
This yields \begin{align}\label{SHchiSmallM0Bound}
     \mathcal{S}_{\chi,H,0}\ll& P^\varepsilon\frac{\sqrt{MN}k^2}{d}\frac{B}{B_d}\frac{BP}{B_dN_1}M\left(\min\left\{\frac{H}{BB_d^3N_1^2},M\right\}+1\right)\nonumber\\
     \ll&P^\varepsilon\frac{B^2PM^\frac{3}{2}\sqrt{N}k^2}{B_d^2dN_1}\left(\frac{BPd\sqrt{M}}{\sqrt{N}}+1\right)\ll P^\varepsilon\frac{B^2PM^\frac{3}{2}\sqrt{N}k^2}{B_d^2dN_1}.
\end{align}
Here we used $M\ll\frac{NP^\varepsilon}{C^2d^2}$ and $BP<CP^\varepsilon$ for $v=1$.

\subsubsection{Sub Case: Offdiagonal $n\neq0$}

Let $\mathcal{S}_{H,\neq0}$ be the contribution of $n\neq 0$ to $\mathcal{S}_H$, i.e. \begin{align}
    \mathcal{S}_{H,\neq 0}:=&\sum_{\substack{B\leq b_1\leq 2B\\ B_d\leq (b_1,d)\leq 2B_d\\
    (P^{1-v},b_1)=1}}\sum_{\substack{B\leq b_2\leq 2B\\ B_d\leq (b_2,d)\leq 2B_d\\
    (P^{1-v},b_2)=1}}\sum_{u_1|b_1P^v}\sum_{u_2|b_2P^v}u_1u_2\sum_{\substack{N_1\leq n_1\leq 2N_1\\
    n_1|b_{1,d}P^v}}\sum_{\substack{N_1\leq n_2\leq 2N_1\\
    n_2|b_{2,d}P^v}}\sumast_{\beta_1\Mod{\frac{b_{1,d}P^v}{n_1}}}\sumast_{\beta_2\Mod{\frac{b_{2,d}P^v}{n_2}}}\nonumber\\
    &\times\sum_{n\neq0}\sum_{m_1\sim M}\sum_{m_2\sim M}\delta\left(H\leq \left|b_1m_1(b_2,d)^3n_2^2-b_2m_2(b_1,d)^3n_1^2\right|\leq 2H, (*)\right)\left|\mathcal{J}\right|,
\end{align}
with $(*)$ the set of congruences in (\ref{*}).

For $H\ll\frac{B^2B_d^3dN_1^2\sqrt{M}P^{\frac{1+v}{2}+\varepsilon}}{\sqrt{N}}$ including $H=0$,  (\ref{IntegralConclusionSmallMSmallH}) gives us for any $A>0$, \begin{align*}
    \mathcal{S}_{H,\neq0}\ll& \frac{\sqrt{MN}k^2}{BdP^\frac{1+v}{2}}\sum_{\substack{B\leq b_1\leq 2B\\ B_d\leq (b_1,d)\leq 2B_d\\
    (P^{1-v},b_1)=1}}\sum_{\substack{B\leq b_2\leq 2B\\ B_d\leq (b_2,d)\leq 2B_d\\
    (P^{1-v},b_2)=1}}\sum_{u_1|b_1P^v}\sum_{u_2|b_2P^v}u_1u_2\sum_{\substack{N_1\leq n_1\leq 2N_1\\
    n_1|b_{1,d}P^v}}\sum_{\substack{N_1\leq n_2\leq 2N_1\\
    n_2|b_{2,d}P^v}}\sumast_{\beta_1\Mod{\frac{b_{1,d}P^v}{n_1}}}\sumast_{\beta_2\Mod{\frac{b_{2,d}P^v}{n_2}}}\\
    &\times\sum_{1\leq |n|\ll \frac{B_dP^{\frac{1-v}{2}+\varepsilon}\sqrt{N}}{d^2\sqrt{M}k^2}}\sum_{m_1\sim M}\sum_{m_2\sim M}\delta\left(H\leq \left|b_1m_1(b_2,d)^3n_2^2-b_2m_2(b_1,d)^3n_1^2\right|\leq 2H,(*)\right)+P^{-A}.
\end{align*}
Note that $n\neq0$ gives arbitrary saving unless \begin{align*}
    M\ll\frac{B_d^2P^{1-v+\varepsilon}N}{d^4k^4}.
\end{align*}
Let $c_0=\left(\frac{b_{1,d}P^v}{n_1},\frac{b_{2,d}P^v}{n_2}\right)$, and for $j=1,2$, let $\frac{b_{j,d}P^v}{n_1}=c_0c_{j,0}c_{j,1}$ with $c_{j,0}|c_0^\infty$ and $(c_{j,1},c_0)=1$. Then $(*)$ gives us $c_0|n$ and \begin{align*}
    \frac{n}{c_0}\equiv& \mp\beta_1c_{2,0}c_{2,1}\pm\beta_2c_{1,0}c_{1,1}\Mod{c_0c_{1,0}c_{2,0}}\\
    \beta_1\equiv& \mp\frac{n}{c_0}\overline{c_{2,0}c_{2,1}}\Mod{c_{1,1}}\\
    \beta_2\equiv&\pm\frac{n}{c_0}\overline{c_{1,0}c_{1,1}}\Mod{c_{2,1}}.
\end{align*}
Hence given $c_0$, we have $c_0|n$, $c_{1,0},c_{2,0}|c_0^\infty$, and given $c_0,n,c_{1,0},c_{2,0}$, we have $c_0c_{1,0}|b_1P^v, c_0c_{2,0}|b_2P^v$. Given the above variables, we determine $\beta_1\Mod{c_{1,1}}$ and given $\beta_1$ we determine $\beta_2$ uniquely. Together with $m_1,m_2$ being determined mod $u_1,u_2$ respectively and the restriction \begin{align*}
    &H\leq \left|b_1m_1(b_2,d)^3n_2^2-b_2m_2(b_1,d)^3n_1^2\right|\leq 2H\\
    \Rightarrow&\left|m_2-\frac{b_1m_1(b_2,d)^3n_2^2}{b_2(b_1,d)^3n_1^2}\right|\ll\frac{H}{BB_d^3N_1^2},
\end{align*}
we get for $H\ll\frac{B^2B_d^3dN_1^2\sqrt{M}P^{\frac{1+v}{2}+\varepsilon}}{\sqrt{N}}$, \begin{align}\label{SHSmallMSmallHnon0Bound}
    \mathcal{S}_{H,\neq 0}\ll& P^\varepsilon\frac{\sqrt{MN}k^2}{BdP^\frac{1+v}{2}}\sum_{c_0\ll\frac{BP^v}{B_dN_1}}\frac{B_dP^{\frac{1-v}{2}}\sqrt{N}}{c_0d^2\sqrt{M}k^2}\sup_{c_{1,0},c_{2,0}\ll \frac{BP^v}{c_0B_dN_1}}\min\left\{\frac{B}{B_d},\frac{BP^v}{c_0c_{1,0}}\right\}\min\left\{\frac{B}{B_d},\frac{BP^v}{c_0c_{2,0}}\right\}c_0c_{1,0}\nonumber\\
    &\times \sup_{u_1,u_2\ll BP^v}u_1u_2\left(\frac{M}{u_1}+1\right)\left(\min\left\{\frac{H}{u_2BB_d^3N_1^2},\frac{M}{u_2}\right\}+1\right)\delta\left(M\ll\frac{B_d^2P^{1-v+\varepsilon}N}{d^4k^4}\right)\nonumber\\
    \ll&P^\varepsilon\frac{BN}{d^3}\left(M+BP^v\right)\left(\frac{Bd\sqrt{M}P^\frac{1+v}{2}}{\sqrt{N}}+BP^v\right)\delta\left(M\ll\frac{B_d^2P^{1-v+\varepsilon}N}{d^4k^4}\right)
\end{align}

For the remaining range $\frac{B^2B_d^3dN_1^2\sqrt{M}P^{\frac{1+v}{2}+\varepsilon}}{\sqrt{N}}\ll H\ll BB_d^3N_1^2MP^\varepsilon$, (\ref{IntegralConclusionSmallMLargeH}) gives us for any $A>0$, \begin{align*}
    \mathcal{S}_{H,\neq0}\ll& \frac{B_d^\frac{3}{2}N_1N^\frac{1}{4}M^\frac{3}{4}k^2}{\sqrt{dH}P^{\frac{1+v}{4}}}\sum_{\substack{B\leq b_1\leq 2B\\ B_d\leq (b_1,d)\leq 2B_d\\
    (P^{1-v},b_1)=1}}\sum_{\substack{B\leq b_2\leq 2B\\ B_d\leq (b_2,d)\leq 2B_d\\
    (P^{1-v},b_2)=1}}\sum_{u_1|b_1P^v}\sum_{u_2|b_2P^v}u_1u_2\sum_{\substack{N_1\leq n_1\leq 2N_1\\
    n_1|b_{1,d}P^v}}\sum_{\substack{N_1\leq n_2\leq 2N_1\\
    n_2|b_{2,d}P^v}}\sumast_{\beta_1\Mod{\frac{b_{1,d}P^v}{n_1}}}\\
    &\times\sumast_{\beta_2\Mod{\frac{b_{2,d}P^v}{n_2}}}\sum_{|n|\sim \frac{NH}{B^2B_d^2d^3N_1^2P^vMk^2}}\sum_{m_1\sim M}\sum_{m_2\sim M}\delta\left(H\leq \left|b_1m_1(b_2,d)^3n_2^2-b_2m_2(b_1,d)^3n_1^2\right|\leq 2H,(*)\right)+P^{-A}.
\end{align*}
Performing a similar procedure as above to analyse the congruences, $(*)$ gives us \begin{align}\label{SHSmallMBigHnon0Bound}
    \mathcal{S}_{H,\neq 0}\ll& P^\varepsilon\frac{B_d^\frac{3}{2}N_1N^\frac{1}{4}M^\frac{3}{4}k^2}{\sqrt{dH}P^{\frac{1+v}{4}}}\sum_{c_0\ll\frac{BP^v}{B_dN_1}}\frac{NH}{c_0B^2B_d^2d^3N_1^2P^vMk^2}\sup_{c_{1,0},c_{2,0}\ll \frac{BP^v}{c_0B_dN_1}}\min\left\{\frac{B}{B_d},\frac{BP^v}{c_0c_{1,0}}\right\}\min\left\{\frac{B}{B_d},\frac{BP^v}{c_0c_{2,0}}\right\}c_0c_{1,0}\nonumber\\
    &\times \sup_{u_1,u_2\ll BP^v}u_1u_2\left(\frac{M}{u_1}+1\right)\left(\min\left\{\frac{H}{u_2BB_d^3N_1^2},\frac{M}{u_2}\right\}+1\right)\nonumber\\
    \ll&P^\varepsilon\frac{N^\frac{5}{4}\sqrt{H}}{d^\frac{7}{2}B_d^\frac{3}{2}N_1P^\frac{1+v}{4}M^\frac{1}{4}}\left(M+BP^v\right)^2\ll P^\varepsilon\frac{\sqrt{B}M^\frac{1}{4}N^\frac{5}{4}}{d^\frac{7}{2}P^\frac{1+v}{4}}\left(M+BP^v\right)^2.
\end{align}

Similarly for the case $\chi\neq\chi_0, v=1$, we have the contribution of $n\neq0$ to $\mathcal{S}_{\chi,H}$ being \begin{align}
    \mathcal{S}_{\chi,H,\neq 0}:=&BP\sum_{\substack{B\leq b_1\leq 2B\\ B_d\leq (b_1,d)\leq 2B_d}}\sum_{\substack{B\leq b_2\leq 2B\\ B_d\leq (b_2,d)\leq 2B_d}}\sum_{\substack{N_1\leq n_1\leq 2N_1\\
    n_1|b_{1,d}P}}\sum_{\substack{N_1\leq n_2\leq 2N_1\\
    n_2|b_{2,d}P}}\sumast_{\beta_1\Mod{\frac{b_{1,d}P}{n_1}}}\sumast_{\beta_2\Mod{\frac{b_{2,d}P}{n_2}}}\nonumber\\
    &\times\sum_{n\neq0}\sum_{m_1\sim M}\sum_{m_2\sim M}\delta\left(H\leq \left|b_1m_1(b_2,d)^3n_2^2-b_2m_2(b_1,d)^3n_1^2\right|\leq 2H, (**)\right)\left|\mathcal{J}\right|.
\end{align}
Applying the same analysis with the same definition of $c_0, c_{j,0}, c_{j,1}$ for $j=1,2$, $(**)$ gives us: given $c_0$, we have $c_0|n$, $c_{1,0},c_{2,0}|c_0^\infty$, and given $c_0,n,c_{1,0},c_{2,0}$, we have $c_0c_{1,0}|b_1P^v, c_0c_{2,0}|b_2P^v$. Given the above variables, we determine $\beta_1\Mod{c_{1,1}}$ and given $\beta_1$ we determine $\beta_2$ uniquely. Given $m_1$, $m_2$ satisfies the restriction \begin{align*}
    &H\leq \left|b_1m_1(b_2,d)^3n_2^2-b_2m_2(b_1,d)^3n_1^2\right|\leq 2H\\
    \Rightarrow&\left|m_2-\frac{b_1m_1(b_2,d)^3n_2^2}{b_2(b_1,d)^3n_1^2}\right|\ll\frac{H}{BB_d^3N_1^2},
\end{align*}
we get for $H\ll\frac{B^2B_d^3dN_1^2\sqrt{M}P^{1+\varepsilon}}{\sqrt{N}}$, \begin{align}\label{SHchiSmallMSmallHnon0Bound}
    \mathcal{S}_{\chi,H,\neq 0}\ll& P^\varepsilon\frac{M^\frac{3}{2}\sqrt{N}k^2}{d}\sum_{c_0\ll\frac{BP}{B_dN_1}}\frac{B_d\sqrt{N}}{c_0d^2\sqrt{M}k^2}\sup_{c_{1,0},c_{2,0}\ll \frac{BP}{c_0B_dN_1}}\min\left\{\frac{B}{B_d},\frac{BP}{c_0c_{1,0}}\right\}\min\left\{\frac{B}{B_d},\frac{BP}{c_0c_{2,0}}\right\}c_0c_{1,0}\nonumber\\
    &\times \left(\min\left\{\frac{H}{BB_d^3N_1^2},M\right\}+1\right)\delta\left(M\ll\frac{B_d^2P^\varepsilon N}{d^4k^4}\right)\nonumber\\
    \ll&P^\varepsilon\frac{B^2PMN}{d^3}\delta\left(M\ll\frac{B_d^2P^\varepsilon N}{d^4k^4}\right),
\end{align}
and for $\frac{B^2B_d^3dN_1^2\sqrt{M}P^{1+\varepsilon}}{\sqrt{N}}\ll H\ll BB_d^3N_1^2MP^\varepsilon$, we have \begin{align}\label{SHchiSmallMBigHnon0Bound}
    \mathcal{S}_{\chi,H,\neq 0}\ll& P^\varepsilon\frac{BB_d^\frac{3}{2}N_1N^\frac{1}{4}M^\frac{3}{4}\sqrt{P}k^2}{\sqrt{dH}}\sum_{c_0\ll\frac{BP^v}{B_dN_1}}\frac{NH}{c_0B^2B_d^2d^3N_1^2PMk^2}\nonumber\\
    &\times \sup_{c_{1,0},c_{2,0}\ll \frac{BP^v}{c_0B_dN_1}}\min\left\{\frac{B}{B_d},\frac{BP}{c_0c_{1,0}}\right\}\min\left\{\frac{B}{B_d},\frac{BP}{c_0c_{2,0}}\right\}c_0c_{1,0}M\left(\min\left\{\frac{H}{BB_d^3N_1^2},M\right\}+1\right)\nonumber\\
    \ll&P^\varepsilon\frac{BM^\frac{7}{4}N^\frac{5}{4}\sqrt{HP}}{d^\frac{7}{2}B_d^\frac{3}{2}N_1}\ll P^\varepsilon\frac{B^\frac{3}{2}PM^\frac{9}{4}N^\frac{5}{4}}{d^\frac{7}{2}}.
\end{align}

\subsubsection{Conclusion of the First Case}

Hence inserting (\ref{SHSmallM0Bound}), (\ref{SHSmallMSmallHnon0Bound}) and (\ref{SHSmallMBigHnon0Bound}) from both sub cases into (\ref{SuvAfterCauchy0split}), we get \begin{align*}
    \left(S_{v,M,\tilde{N},N_1,B,B_d}^{\pm,\eta}\right)^2\ll&\frac{a^2BdP^\varepsilon}{B_d^2P^{\frac{1-3v}{2}}\sqrt{MN}N_1^2}\left(M+BP^v\right)\\
    &\times\left(\left(\frac{B\sqrt{MN}k^2}{B_d^2dN_1P^\frac{1-v}{2}}+\frac{BN}{d^3}\delta\left(M\ll\frac{B_d^2P^{1-v+\varepsilon}N}{d^4k^4}\right)\right)\left(\frac{Bd\sqrt{M}P^\frac{1+v}{2}}{\sqrt{N}}+BP^v\right)\right.\\
    &\left.+\frac{\sqrt{B}M^\frac{1}{4}N^\frac{5}{4}}{d^\frac{7}{2}P^\frac{1+v}{4}}\left(M+BP^v\right)\right).
\end{align*}
Putting this into (\ref{SvMDyadic2}) in Theorem \ref{SvAfterVDThm}, we get for \begin{align*}
    1+\frac{(Bd)^2P^{1+v+\varepsilon}}{N}\ll M\ll \min\left\{\frac{(Bd)^2P^{1+v-\varepsilon}k^2}{N},\frac{P^{1-v+\varepsilon}N}{a^2C^2d^2}\right\},
\end{align*}
we have \begin{align*}
    S_{v,M}\ll& P^\varepsilon\sum_a\sup_{B\ll\frac{C}{aP^v}}\frac{\sqrt{Bd}}{P^{\frac{1-3v}{4}}M^{\frac{1}{4}}N^\frac{1}{4}}\sqrt{M+BP^v}\\
    &\times\left(\left(\frac{\sqrt{B}(MN)^\frac{1}{4}k}{\sqrt{d}P^\frac{1-v}{4}}+\sqrt{\frac{BN}{d^3}}\delta\left(M\ll\frac{a^2B_d^2P^{1-v+\varepsilon}N}{d^4k^4}\right)\right)\left(\frac{\sqrt{Bd}M^\frac{1}{4}P^\frac{1+v}{4}}{N^\frac{1}{4}}+\sqrt{BP^v}\right)\right.\\
    &\left.+\sqrt{\frac{\sqrt{B}M^\frac{1}{4}N^\frac{5}{4}}{d^\frac{7}{2}P^\frac{1+v}{4}}\left(M+BP^v\right)}\right)\\
    \ll&P^\varepsilon\left(\frac{\sqrt{Cd}}{P^{\frac{1-v}{4}}N^\frac{1}{4}}\sqrt{C}\left(\sqrt{\frac{CN}{d^3P^v}}\sqrt{C}\left(1+\frac{C^2d^2P^{1+v+\varepsilon}}{N}\right)^{-\frac{1}{4}}+\sqrt{\frac{\sqrt{C}N^\frac{5}{4}}{d^\frac{7}{2}P^\frac{1+3v}{4}}}\sqrt{C}\left(1+\frac{C^2d^2P^{1+v+\varepsilon}}{N}\right)^{-\frac{1}{8}}\right)\right.\\
    &\left.+\frac{Cd}{P^\frac{1-v}{2}\sqrt{N}}\sqrt{\frac{P^{1-v}N}{C^2d^2}+C}\left(\frac{\sqrt{N}k}{dP^\frac{v}{2}}\sqrt{P^{1-v}+C}+\frac{N^\frac{3}{4}}{d^2P^\frac{v}{2}}\sqrt{\frac{P^{1-v}N}{C^2d^2}+C}\right)\right)\\
    \ll& P^\varepsilon\frac{C^2N^\frac{1}{4}}{dP^\frac{1}{4}}\left(1+\frac{C^2d^2P}{N}\right)^{-\frac{1}{4}}\left(1+\frac{N}{C^2d^2P}\right)^\frac{1}{8}+P^\varepsilon\frac{Ck}{\sqrt{P}}\sqrt{\frac{PN}{C^2d^2}+C}\left(\sqrt{P+C}+\frac{N^\frac{1}{4}}{dk}\sqrt{\frac{PN}{C^2d^2}+C}\right).
\end{align*}
By choosing $C\gg\max\left\{\frac{\sqrt{N}}{d\sqrt{P}},P\right\}$, we can simplify the above bound into \begin{align}\label{SuvMSmallMBound}
    S_{v,M}\ll P^\varepsilon\left(\frac{C^\frac{3}{2}\sqrt{N}}{d^\frac{3}{2}\sqrt{P}}+\frac{\sqrt{CN}k}{d}+\frac{\sqrt{P}N^\frac{5}{4}}{Cd^3}+\frac{C^2k}{\sqrt{P}}\right).
\end{align}

Similarly for the case $\chi\neq\chi_0, v=1$, inserting (\ref{SHchiSmallM0Bound}), (\ref{SHchiSmallMSmallHnon0Bound}) and (\ref{SHchiSmallMBigHnon0Bound}) from both sub cases into (\ref{SuvchiAfterCauchy0split}), we get \begin{align}
    S_{1,M}^{\eta}\ll& P^\varepsilon\sum_a\sup_{B\ll\frac{C}{aP}}\frac{\sqrt{BdP}}{M^{\frac{1}{4}}N^\frac{1}{4}}\sqrt{BPM}\left(\frac{\sqrt{B}(MN)^\frac{1}{4}k}{\sqrt{d}}+\sqrt{\frac{BN}{d^3}}\delta\left(M\ll\frac{a^2B_d^2P^\varepsilon N}{d^4k^4}\right)+\sqrt{\frac{\sqrt{B}M^\frac{5}{4}N^\frac{5}{4}}{d^\frac{7}{2}\sqrt{P}}}\right)\nonumber\\
    \ll&P^\varepsilon\left(\frac{\sqrt{CN}k}{d\sqrt{P}}+\frac{N^\frac{5}{4}}{Cd^3\sqrt{P}}\right).
\end{align}
Note that this is dominated by (\ref{SuvMSmallMBound}).

\subsection{Second Case: Very Small \texorpdfstring{$M$}{M}}

For $M\ll \frac{(Bd)^2P^{1+v+\varepsilon}}{N}$, the treatment is very similar. It is simpler and it is dependent of $H$. Recalling (\ref{PsiDef}) and (\ref{UltIntegralDef}), we apply Lemma \ref{HugeIntegralAnalysis2} on $\mathcal{J}$ to get for any $A>0$, \begin{align}\label{IntegralConclusion0M}
    \mathcal{J}\ll \frac{B_d^3N\tilde{N}N_1^2}{(BdP^v)^3}\delta\left(\tilde{N}\ll\frac{(BdP^v)^3k^2P^\varepsilon}{B_d^3NN_1^2}, |n|\ll\frac{B^2P^{2v}}{\tilde{N}B_d^2N_1^2}P^\varepsilon\right)+P^{-A}.
\end{align}

\subsubsection{Sub Case: Diagonal $n=0$}

By (\ref{IntegralConclusion0M}), the contribution of $n=0$ to $\mathcal{S}_H$ is \begin{align*}
    \mathcal{S}_{H,0}=&\sum_{\substack{B\leq b_1\leq 2B\\ B_d\leq (b_1,d)\leq 2B_d\\
    (P^{1-v},b_1)=1}}\sum_{\substack{B\leq b_2\leq 2B\\ B_d\leq (b_2,d)\leq 2B_d\\
    (P^{1-v},b_2)=1}}\sum_{u_1|b_1P^v}\sum_{u_2|b_2P^v}u_1u_2\sum_{\substack{N_1\leq n_1\leq 2N_1\\
    n_1|b_{1,d}P^v}}\sum_{\substack{N_1\leq n_2\leq 2N_1\\
    n_2|b_{2,d}P^v}}\sumast_{\beta_1\Mod{\frac{b_{1,d}P^v}{n_1}}}\sumast_{\beta_2\Mod{\frac{b_{2,d}P^v}{n_2}}}\nonumber\\
    &\times\sum_n\sum_{m_1\sim M}\sum_{m_2\sim M}\delta\left(H\leq \left|b_1m_1(b_2,d)^3n_2^2-b_2m_2(b_1,d)^3n_1^2\right|\leq 2H, n=0, (*)\right)\left|\mathcal{J}\right|\\
    \ll& \frac{B_d^3N\tilde{N}N_1^2}{(BdP^v)^3}\sum_{\substack{B\leq b_1\leq 2B\\ B_d\leq (b_1,d)\leq 2B_d\\
    (P^{1-v},b_1)=1}}\sum_{\substack{B\leq b_2\leq 2B\\ B_d\leq (b_2,d)\leq 2B_d\\
    (P^{1-v},b_2)=1}}\sum_{u_1|b_1P^v}\sum_{u_2|b_2P^v}u_1u_2\sum_{\substack{N_1\leq n_1\leq 2N_1\\
    n_1|b_{1,d}P^v}}\sum_{\substack{N_1\leq n_2\leq 2N_1\\
    n_2|b_{2,d}P^v}}\sumast_{\beta_1\Mod{\frac{b_{1,d}P^v}{n_1}}}\sumast_{\beta_2\Mod{\frac{b_{2,d}P^v}{n_2}}}\\
    &\times\sum_n\sum_{m_1\sim M}\sum_{m_2\sim M}\delta\left(\tilde{N}\ll\frac{(BdP^v)^3k^2P^\varepsilon}{B_d^3NN_1^2}, n=0, (*)\right)+P^{-A}
\end{align*}
for any $A>0$. Recall that $(*)$ together with $n=0$ gives us \begin{align*}
    &m_1\equiv  \overline{d_{b_1}\beta_1}P^{1-v}n_1(b_1,d)\Mod{u_1},\\
    &m_2\equiv  \overline{d_{b_2}\beta_2}P^{1-v}n_2(b_2,d)\Mod{u_2},\beta_1\frac{b_{2,d}P^v}{n_2}\equiv\beta_2\frac{b_{1,d}P^v}{n_1}\Mod{\frac{b_{1,d}b_{2,d}P^{2v}}{n_1n_2}},
\end{align*}
which implies $\frac{b_{1,d}P^v}{n_1}=\frac{b_{2,d}P^v}{n_2}$ and $\beta_1=\beta_2$, since $\beta_1,\beta_2$ are coprime to $\frac{b_{1,d}P^v}{n_1}$ and $\frac{b_{2,d}P^v}{n_2}$ respectively. As a result, given $n_1,n_2,b_1,\beta_1,u_1,u_2$, then $b_2,\beta_2$ are uniquely determined, $m_1$ is determined mod $u_1$ and $m_2$ is determined mod $u_2$. This yields \begin{align}\label{SH0M0Bound}
     \mathcal{S}_{H,0}\ll& P^\varepsilon\frac{B_d^3NN_1^2}{(BdP^v)^3}\frac{(BdP^v)^3k^2}{B_d^3NN_1^2}\frac{B}{B_d}\frac{BP^v}{B_dN_1}\sup_{u_1,u_2\ll BP^v}u_1u_2\left(\frac{M}{u_1}+1\right)\left(\frac{M}{u_2}+1\right)\nonumber\\
     \ll&P^\varepsilon\frac{B^2P^vk^2}{B_d^2N_1}\left(M+BP^v\right)^2.
\end{align}

Similarly for the case $\chi\neq\chi_0, v=1$, we have the contribution of $n=0$ to $\mathcal{S}_{\chi,H}$ is \begin{align}\label{SHchi0M0Bound}
    \mathcal{S}_{\chi,H,0}\ll P^\varepsilon\frac{B^3P^2M^2k^2}{B_d^2N_1}.
\end{align}

\subsubsection{Sub Case: Offdiagonal $n\neq0$}

Similarly, (\ref{IntegralConclusion0M}) gives us the contribution of $n\neq0$ to $\mathcal{S}_H$ is \begin{align*}
    \mathcal{S}_{H,\neq0}=&\sum_{\substack{B\leq b_1\leq 2B\\ B_d\leq (b_1,d)\leq 2B_d\\
    (P^{1-v},b_1)=1}}\sum_{\substack{B\leq b_2\leq 2B\\ B_d\leq (b_2,d)\leq 2B_d\\
    (P^{1-v},b_2)=1}}\sum_{u_1|b_1P^v}\sum_{u_2|b_2P^v}u_1u_2\sum_{\substack{N_1\leq n_1\leq 2N_1\\
    n_1|b_{1,d}P^v}}\sum_{\substack{N_1\leq n_2\leq 2N_1\\
    n_2|b_{2,d}P^v}}\sumast_{\beta_1\Mod{\frac{b_{1,d}P^v}{n_1}}}\sumast_{\beta_2\Mod{\frac{b_{2,d}P^v}{n_2}}}\nonumber\\
    &\times\sum_{n\neq0}\sum_{m_1\sim M}\sum_{m_2\sim M}\delta\left(H\leq \left|b_1m_1(b_2,d)^3n_2^2-b_2m_2(b_1,d)^3n_1^2\right|\leq 2H, (*)\right)\left|\mathcal{J}\right|\\
    \ll& \frac{B_d^3N\tilde{N}N_1^2}{(BdP^v)^3}\sum_{\substack{B\leq b_1\leq 2B\\ B_d\leq (b_1,d)\leq 2B_d\\
    (P^{1-v},b_1)=1}}\sum_{\substack{B\leq b_2\leq 2B\\ B_d\leq (b_2,d)\leq 2B_d\\
    (P^{1-v},b_2)=1}}\sum_{u_1|b_1P^v}\sum_{u_2|b_2P^v}u_1u_2\sum_{\substack{N_1\leq n_1\leq 2N_1\\
    n_1|b_{1,d}P^v}}\sum_{\substack{N_1\leq n_2\leq 2N_1\\
    n_2|b_{2,d}P^v}}\sumast_{\beta_1\Mod{\frac{b_{1,d}P^v}{n_1}}}\sumast_{\beta_2\Mod{\frac{b_{2,d}P^v}{n_2}}}\\
    &\times\sum_{|n|\ll\frac{B^2P^{2v+\varepsilon}}{\tilde{N}B_d^2N_1^2}}\sum_{m_1\sim M}\sum_{m_2\sim M}\delta\left(\tilde{N}\ll\frac{(BdP^v)^3k^2P^\varepsilon}{B_d^3NN_1^2}, (*)\right)+P^{-A}
\end{align*}
Let $c_0=\left(\frac{b_{1,d}P^v}{n_1},\frac{b_{2,d}P^v}{n_2}\right)$, and for $j=1,2$, let $\frac{b_{j,d}P^v}{n_1}=c_0c_{j,0}c_{j,1}$ with $c_{j,0}|c_0^\infty$ and $(c_{j,1},c_0)=1$. Applying the same analysis as done in the main case implies: given $c_0$, we have $c_0|n$, $c_{1,0},c_{2,0}|c_0^\infty$, and given $c_0,n,c_{1,0},c_{2,0}$, we have $c_0c_{1,0}|b_1P^v, c_0c_{2,0}|b_2P^v$. Given the above variables, we determine $\beta_1\Mod{c_{1,1}}$ and given $\beta_1$ we determine $\beta_2$ uniquely. Also, $m_1,m_2$ being determined mod $u_1,u_2$ respectively. Hence we get the bound \begin{align}\label{SH0Mnon0Bound}
    \mathcal{S}_{H,\neq0}\ll&P^\varepsilon\frac{B_d^3N\tilde{N}N_1^2}{(BdP^v)^3}\sum_{c_0\ll\frac{BP^v}{B_dN_1}}\frac{B^2P^{2v+\varepsilon}}{c_0\tilde{N}B_d^2N_1^2}\sup_{c_{1,0},c_{2,0}\ll \frac{BP^v}{c_0B_dN_1}}\min\left\{\frac{B}{B_d},\frac{BP^v}{c_0c_{1,0}}\right\}\min\left\{\frac{B}{B_d},\frac{BP^v}{c_0c_{2,0}}\right\}c_0c_{1,0}\nonumber\\
    &\times \sup_{u_1,u_2\ll BP^v}u_1u_2\left(\frac{M}{u_1}+1\right)\left(\frac{M}{u_2}+1\right)\nonumber\\
    \ll&P^\varepsilon\frac{BN}{d^3}\left(M+BP^v\right)^2.
\end{align}

Similarly for the case $\chi\neq\chi_0, v=1$, we have the contribution of $n\neq0$ to $\mathcal{S}_{\chi,H}$ is \begin{align}
    \mathcal{S}_{\chi,H,\neq0}\ll P^\varepsilon\frac{B^2PM^2N}{d^3}.
\end{align}

\subsubsection{Conclusion of the Second Case}

Hence inserting (\ref{SH0M0Bound}), (\ref{SH0Mnon0Bound}) from both sub cases into (\ref{SuvAfterCauchy0split}), we get \begin{align*}
    \left(S_{v,M,\tilde{N},N_1,B,B_d}^{\pm,\eta}\right)^2\ll&\frac{a^2P^\varepsilon}{B_d^2P^{1-v}N_1^2}\left(M+BP^v\right)^2\left(\frac{B^2P^vk^2}{B_d^2N_1}+\frac{BN}{d^3}\right).
\end{align*}
Putting this into (\ref{SvMDyadic2}) in Theorem \ref{SvAfterVDThm}, we get for \begin{align*}
    \frac{1}{2}\leq M\leq \frac{(Bd)^2P^{1+v+\varepsilon}}{N},
\end{align*}
we have \begin{align*}
    S_{v,M}\ll& P^\varepsilon\sum_a\sup_{B\ll\frac{C}{aP^v}}\frac{1}{B_dP^{\frac{1-v}{2}}N_1}\left(M+BP^v\right)\sqrt{\frac{B^2P^vk^2}{B_d^2N_1}+\frac{BN}{d^3}}\nonumber\\
    \ll&P^\varepsilon\sum_a\sup_{B\ll\frac{C}{aP^v}}\frac{1}{B_dP^{\frac{1-1v}{2}}N_1}\left(\frac{(Bd)^2P^{1+v}}{N}+BP^v\right)\sqrt{\frac{B^2P^vk^2}{B_d^2N_1}+\frac{BN}{d^3}}\nonumber\\
    \ll&P^\varepsilon\frac{C^\frac{3}{2}\sqrt{N}}{d^\frac{3}{2}\sqrt{P}}\left(1+\frac{Cd^2P}{N}\right)\sqrt{1+\frac{Cd^3k^2}{N}}.
\end{align*}
Choosing $C\gg P$ with the restriction $C^2d^2\leq N$, we get \begin{align}\label{SuvM0MBound}
    S_{v,M}\ll P^\varepsilon\left(\frac{C^\frac{3}{2}\sqrt{N}}{d^\frac{3}{2}\sqrt{P}}+\frac{C^2k}{\sqrt{P}}\right).
\end{align}
Note that by the restriction $C^2d^2\leq N$, (\ref{SuvM0MBound}) is dominated by (\ref{SuvMSmallMBound}).

Similarly for the case $\chi\neq\chi_0, v=1$, we have \begin{align}\label{SuvMchi0MBound}
    S_{1,M}^{\eta}\ll& P^\varepsilon\sum_a\sup_{B\ll\frac{C}{aP}}\frac{1}{B_dN_1}\left(\frac{B^\frac{3}{2}PMk}{B_d\sqrt{N_1}}+\frac{B\sqrt{P}M\sqrt{N}}{d^\frac{3}{2}}\right)\nonumber\\
    \ll&P^\varepsilon\left(\frac{C\sqrt{N}}{d^\frac{3}{2}\sqrt{P}}\frac{C^2d^2}{N}+\frac{C^\frac{3}{2}k}{\sqrt{P}}\frac{C^2d^2}{N}\right).
\end{align}
Note that by the restriction $C^2d^2\leq N$, (\ref{SuvMchi0MBound}) is dominated by (\ref{SuvM0MBound}).

\subsection{Third Case: Large \texorpdfstring{$M$}{M}}

For $\frac{(Bd)^2P^{1+v+\varepsilon}k^2}{N}\ll M\ll\frac{P^{1-v+\varepsilon}N}{C^2d^2}$, the treatment is similar to the above two cases. Note that for $M$ to have such size, we have \begin{align*}
    B\ll\frac{\sqrt{MN}}{dP^\frac{1+v}{2}k}\ll\frac{N}{Cd^2P^vk}P^\varepsilon.
\end{align*}
For $A>0$, recalling (\ref{PsiDef}) and (\ref{UltIntegralDef}), we apply Lemma \ref{HugeIntegralAnalysis2} on $\mathcal{J}$ to get the following:

If $\frac{dP^\frac{1+v}{2}k^2}{B_d^3N_1^2\sqrt{N}M^\frac{3}{2}}\left|b_1m_1(b_2,d)^3n_2^2-b_2m_2(b_1,d)^3n_1^2\right|\ll P^\varepsilon$, \begin{align}\label{IntegralConclusionLargeMSmallH}
    \mathcal{J}\ll &\frac{(MN)^\frac{3}{2}}{B^3d^3P^\frac{3+3v}{2}}\delta\left(\tilde{N}\sim \frac{\sqrt{N}M^\frac{3}{2}}{B_d^3N_1^2P^\frac{3-3v}{2}}, \left|n\pm \frac{P(b_1m_1(b_2,d)^3n_2^2-b_2m_2(b_1,d)^3n_1^2)}{(b_1,d)(b_2,d)dm_1m_2n_1n_2}\right|\ll\frac{B^2B_dP^{\frac{3+v}{2}+\varepsilon}}{\sqrt{N}M^\frac{3}{2}}\right)+P^{-A}\nonumber\\
    \ll&\frac{(MN)^\frac{3}{2}}{B^3d^3P^\frac{3+3v}{2}}\delta\left(\tilde{N}\sim \frac{\sqrt{N}M^\frac{3}{2}}{B_d^3N_1^2P^\frac{3-3v}{2}}, |n|\ll\frac{B_dP^{\frac{1-v}{2}+\varepsilon}\sqrt{N}}{d^2\sqrt{M}k^2}\right)+P^{-A},
\end{align}
and if $\frac{dP^\frac{1+v}{2}k^2}{B_d^3N_1^2\sqrt{N}M^\frac{3}{2}}\left|b_1m_1(b_2,d)^3n_2^2-b_2m_2(b_1,d)^3n_1^2\right|\gg P^\varepsilon$, \begin{align}\label{IntegralConclusionLargeMLargeH}
    \mathcal{J}\ll &\frac{B_d^\frac{3}{2}N_1N^\frac{7}{4}M^\frac{9}{4}}{B^3d^\frac{7}{2}P^\frac{7+7v}{4}k}\left|b_1m_1(b_2,d)^3n_2^2-b_2m_2(b_1,d)^3n_1^2\right|^{-\frac{1}{2}}\nonumber\\
    &\times\delta\left(\tilde{N}\sim \frac{\sqrt{N}M^\frac{3}{2}}{B_d^3N_1^2P^\frac{3-3v}{2}}, \left|n\right|\sim\frac{P}{B_d^2dN_1^2M^2}\left|b_1m_1(b_2,d)^3n_2^2-b_2m_2(b_1,d)^3n_1^2\right|\right)+P^{-A}.
\end{align}

\subsubsection{Sub Case: Diagonal $n=0$}

By (\ref{IntegralConclusionLargeMLargeH}), the sub case $n=0$ only occurs when $H\ll\frac{B_d^3N_1^2\sqrt{N}M^\frac{3}{2}}{dP^\frac{1+v}{2}k^2}P^\varepsilon$ including $H=0$, and in such ranges of $H$ (\ref{IntegralConclusionLargeMSmallH}) gives us the contribution to $\mathcal{S}_H$ is \begin{align*}
    \mathcal{S}_{H,0}=&\sum_{\substack{B\leq b_1\leq 2B\\ B_d\leq (b_1,d)\leq 2B_d\\
    (P^{1-v},b_1)=1}}\sum_{\substack{B\leq b_2\leq 2B\\ B_d\leq (b_2,d)\leq 2B_d\\
    (P^{1-v},b_2)=1}}\sum_{u_1|b_1P^v}\sum_{u_2|b_2P^v}u_1u_2\sum_{\substack{N_1\leq n_1\leq 2N_1\\
    n_1|b_{1,d}P^v}}\sum_{\substack{N_1\leq n_2\leq 2N_1\\
    n_2|b_{2,d}P^v}}\sumast_{\beta_1\Mod{\frac{b_{1,d}P^v}{n_1}}}\sumast_{\beta_2\Mod{\frac{b_{2,d}P^v}{n_2}}}\nonumber\\
    &\times\sum_n\sum_{m_1\sim M}\sum_{m_2\sim M}\delta\left(H\leq \left|b_1m_1(b_2,d)^3n_2^2-b_2m_2(b_1,d)^3n_1^2\right|\leq 2H, n=0, (*)\right)\left|\mathcal{J}\right|\\
    \ll& \frac{(MN)^\frac{3}{2}}{B^3d^3P^\frac{3+3v}{2}}\sum_{\substack{B\leq b_1\leq 2B\\ B_d\leq (b_1,d)\leq 2B_d\\
    (P^{1-v},b_1)=1}}\sum_{\substack{B\leq b_2\leq 2B\\ B_d\leq (b_2,d)\leq 2B_d\\
    (P^{1-v},b_2)=1}}\sum_{u_1|b_1P^v}\sum_{u_2|b_2P^v}u_1u_2\sum_{\substack{N_1\leq n_1\leq 2N_1\\
    n_1|b_{1,d}P^v}}\sum_{\substack{N_1\leq n_2\leq 2N_1\\
    n_2|b_{2,d}P^v}}\sumast_{\beta_1\Mod{\frac{b_{1,d}P^v}{n_1}}}\sumast_{\beta_2\Mod{\frac{b_{2,d}P^v}{n_2}}}\\
    &\times\sum_n\sum_{m_1\sim M}\sum_{m_2\sim M}\delta\left(H\leq \left|b_1m_1(b_2,d)^3n_2^2-b_2m_2(b_1,d)^3n_1^2\right|\leq 2H, n=0, (*)\right)+P^{-A}
\end{align*}
for any $A>0$, with $(*)$ the set of congruences in (\ref{*}). Applying the exact same analysis as the $n=0$ sub case in the previous two cases, we get \begin{align}\label{SHLargeM0Bound}
     \mathcal{S}_{H,0}\ll& P^\varepsilon\frac{(MN)^\frac{3}{2}}{B^3d^3P^\frac{3+3v}{2}}\frac{B}{B_d}\frac{BP^v}{B_dN_1}\sup_{u_1,u_2\ll BP^v}u_1u_2\left(\frac{M}{u_1}+1\right)\left(\min\left\{\frac{H}{u_2BB_d^3N_1^2},\frac{M}{u_2}\right\}+1\right)\nonumber\\
     \ll&P^\varepsilon\frac{(MN)^\frac{3}{2}}{BB_d^2d^3N_1P^\frac{3+v}{2}}\left(M+BP^v\right)\left(\frac{M^\frac{3}{2}\sqrt{N}}{BdP^\frac{1+v}{2}k^2}+BP^v\right).
\end{align}

Similarly for the case $\chi\neq\chi_0, v=1$, we have the contribution of $n=0$ to $\mathcal{S}_{\chi,H}$ is \begin{align}\label{SHchiLargeM0Bound}
    \mathcal{S}_{\chi,H,0}\ll P^\varepsilon\frac{M^\frac{5}{2}N^\frac{3}{2}}{B_d^2d^3N_1P}\left(\frac{M^\frac{3}{2}\sqrt{N}}{BdPk^2}+1\right).
\end{align}

\subsubsection{Sub Case: Offdiagonal $n\neq0$}

For $H\ll\frac{B_d^3N_1^2\sqrt{N}M^\frac{3}{2}}{dP^\frac{1+v}{2}k^2}P^\varepsilon$ including $H=0$,  (\ref{IntegralConclusionLargeMSmallH}) gives us for any $A>0$, \begin{align*}
    \mathcal{S}_{H,\neq 0}=&\sum_{\substack{B\leq b_1\leq 2B\\ B_d\leq (b_1,d)\leq 2B_d\\
    (P^{1-v},b_1)=1}}\sum_{\substack{B\leq b_2\leq 2B\\ B_d\leq (b_2,d)\leq 2B_d\\
    (P^{1-v},b_2)=1}}\sum_{u_1|b_1P^v}\sum_{u_2|b_2P^v}u_1u_2\sum_{\substack{N_1\leq n_1\leq 2N_1\\
    n_1|b_{1,d}P^v}}\sum_{\substack{N_1\leq n_2\leq 2N_1\\
    n_2|b_{2,d}P^v}}\sumast_{\beta_1\Mod{\frac{b_{1,d}P^v}{n_1}}}\sumast_{\beta_2\Mod{\frac{b_{2,d}P^v}{n_2}}}\nonumber\\
    &\times\sum_{n\neq0}\sum_{m_1\sim M}\sum_{m_2\sim M}\delta\left(H\leq \left|b_1m_1(b_2,d)^3n_2^2-b_2m_2(b_1,d)^3n_1^2\right|\leq 2H, (*)\right)\left|\mathcal{J}\right|\\
    \ll& \frac{(MN)^\frac{3}{2}}{B^3d^3P^\frac{3+3v}{2}}\sum_{\substack{B\leq b_1\leq 2B\\ B_d\leq (b_1,d)\leq 2B_d\\
    (P^{1-v},b_1)=1}}\sum_{\substack{B\leq b_2\leq 2B\\ B_d\leq (b_2,d)\leq 2B_d\\
    (P^{1-v},b_2)=1}}\sum_{u_1|b_1P^v}\sum_{u_2|b_2P^v}u_1u_2\sum_{\substack{N_1\leq n_1\leq 2N_1\\
    n_1|b_{1,d}P^v}}\sum_{\substack{N_1\leq n_2\leq 2N_1\\
    n_2|b_{2,d}P^v}}\sumast_{\beta_1\Mod{\frac{b_{1,d}P^v}{n_1}}}\sumast_{\beta_2\Mod{\frac{b_{2,d}P^v}{n_2}}}\\
    &\times\sum_{1\leq |n|\ll\frac{B_dP^{\frac{1-v}{2}+\varepsilon}\sqrt{N}}{d^2\sqrt{M}k^2}}\sum_{m_1\sim M}\sum_{m_2\sim M}\delta\left(H\leq \left|b_1m_1(b_2,d)^3n_2^2-b_2m_2(b_1,d)^3n_1^2\right|\leq 2H,(*)\right)+P^{-A}.
\end{align*}
Note that $n\neq0$ gives arbitrary saving unless \begin{align*}
    M\ll\frac{B_d^2P^{1-v}N}{d^4k^4}P^\varepsilon.
\end{align*}
This together with $M\gg \frac{(Bd)^2P^{1+v+\varepsilon}k^2}{N}$ implies \begin{align*}
    B\ll \frac{B_dN}{d^3P^vk^3}P^\varepsilon.
\end{align*}
Applying the exact same analysis as the $n\neq0$ sub case in the first case, we get for $H\ll\frac{B_d^3N_1^2\sqrt{N}M^\frac{3}{2}}{dP^\frac{1+v}{2}k^2}P^\varepsilon$, \begin{align}\label{SHLargeMSmallHnon0Bound}
    \mathcal{S}_{H,\neq 0}\ll& P^\varepsilon\frac{(MN)^\frac{3}{2}}{B^3d^3P^\frac{3+3v}{2}}\sum_{c_0\ll\frac{BP^v}{B_dN_1}}\frac{B_dP^{\frac{1-v}{2}}\sqrt{N}}{c_0d^2\sqrt{M}k^2}\sup_{c_{1,0},c_{2,0}\ll \frac{BP^v}{c_0B_dN_1}}\min\left\{\frac{B}{B_d},\frac{BP^v}{c_0c_{1,0}}\right\}\min\left\{\frac{B}{B_d},\frac{BP^v}{c_0c_{2,0}}\right\}c_0c_{1,0}\nonumber\\
    &\times \sup_{u_1,u_2\ll BP^v}u_1u_2\left(\frac{M}{u_1}+1\right)\left(\min\left\{\frac{H}{u_2BB_d^3N_1^2},\frac{M}{u_2}\right\}+1\right)\delta\left(M\ll\frac{B_d^2P^{1-v}N}{d^4k^4}P^\varepsilon, B\ll \frac{B_dN}{d^3P^vk^3}\right)+P^{-A}\nonumber\\
    \ll&P^\varepsilon\frac{MN^2}{Bd^5P^{1+v}k^2}\left(M+BP^v\right)\left(\frac{M^\frac{3}{2}\sqrt{N}}{BdP^\frac{1+v}{2}k^2}+BP^v\right)\delta\left(M\ll\frac{B_d^2P^{1-v}N}{d^4k^4}P^\varepsilon, B\ll \frac{B_dN}{d^3P^vk^3}P^\varepsilon\right)+P^{-A}.
\end{align}

For $\frac{B_d^3N_1^2\sqrt{N}M^\frac{3}{2}}{dP^\frac{1+v}{2}k^2}P^\varepsilon\ll H\ll BB_d^3N_1^2MP^\varepsilon$, (\ref{IntegralConclusionLargeMLargeH}) gives us for any $A>0$, \begin{align*}
    \mathcal{S}_{H,\neq0}\ll&P^\varepsilon\frac{B_d^\frac{3}{2}N_1N^\frac{7}{4}M^\frac{9}{4}}{B^3d^\frac{7}{2}P^\frac{7+7v}{4}k\sqrt{H}}\sum_{\substack{B\leq b_1\leq 2B\\ B_d\leq (b_1,d)\leq 2B_d\\
    (P^{1-v},b_1)=1}}\sum_{\substack{B\leq b_2\leq 2B\\ B_d\leq (b_2,d)\leq 2B_d\\
    (P^{1-v},b_2)=1}}\sum_{u_1|b_1P^v}\sum_{u_2|b_2P^v}u_1u_2\sum_{\substack{N_1\leq n_1\leq 2N_1\\
    n_1|b_{1,d}P^v}}\sum_{\substack{N_1\leq n_2\leq 2N_1\\
    n_2|b_{2,d}P^v}}\sumast_{\beta_1\Mod{\frac{b_{1,d}P^v}{n_1}}}\\
    &\times\sumast_{\beta_2\Mod{\frac{b_{2,d}P^v}{n_2}}}\sum_{|n|\sim\frac{PH}{B_d^2dN_1^2M^2}}\sum_{m_1\sim M}\sum_{m_2\sim M}\delta\left(H\leq \left|b_1m_1(b_2,d)^3n_2^2-b_2m_2(b_1,d)^3n_1^2\right|\leq 2H,(*)\right)+P^{-A}.
\end{align*}
Applying the same analysis as above or the first case, we get \begin{align}\label{SHLargeMBigHnon0Bound}
    \mathcal{S}_{H,\neq 0}\ll& P^\varepsilon\frac{B_d^\frac{3}{2}N_1N^\frac{7}{4}M^\frac{9}{4}}{B^3d^\frac{7}{2}P^\frac{7+7v}{4}k\sqrt{H}}\sum_{c_0\ll\frac{BP^v}{B_dN_1}}\frac{PH}{c_0B_d^2dN_1^2M^2}\sup_{c_{1,0},c_{2,0}\ll \frac{BP^v}{c_0B_dN_1}}\min\left\{\frac{B}{B_d},\frac{BP^v}{c_0c_{1,0}}\right\}\min\left\{\frac{B}{B_d},\frac{BP^v}{c_0c_{2,0}}\right\}c_0c_{1,0}\nonumber\\
    &\times \sup_{u_1,u_2\ll BP^v}u_1u_2\left(\frac{M}{u_1}+1\right)\left(\min\left\{\frac{H}{u_2BB_d^3N_1^2},\frac{M}{u_2}\right\}+1\right)\nonumber\\
    \ll&P^\varepsilon\frac{M^\frac{1}{4}N^\frac{7}{4}\sqrt{H}}{BB_d^\frac{3}{2}d^\frac{9}{2}N_1P^\frac{3+3v}{4}k}\left(M+BP^v\right)^2\ll P^\varepsilon\frac{M^\frac{3}{4}N^\frac{7}{4}}{\sqrt{B}d^\frac{9}{2}P^\frac{3+3v}{4}k}\left(M+BP^v\right)^2.
\end{align}

Similarly for the case $\chi\neq\chi_0, v=1$, we have the contribution of $n\neq 0$ to $\mathcal{S}_{\chi,H}$ is \begin{align}
    \mathcal{S}_{\chi,H,\neq 0}\ll P^\varepsilon\frac{M^2N^2}{d^5Pk^2}\left(\frac{M^\frac{3}{2}\sqrt{N}}{BdPk^2}+1\right)\delta\left(M\ll\frac{B_d^2N}{d^4k^4}P^\varepsilon, B\ll \frac{B_dN}{d^3Pk^3}P^\varepsilon\right)+P^{-A}
\end{align}
for $H\ll\frac{B_d^3N_1^2\sqrt{N}M^\frac{3}{2}}{dPk^2}P^\varepsilon$, and \begin{align}
    \mathcal{S}_{\chi,H,\neq 0}\ll P^\varepsilon\frac{\sqrt{B}M^\frac{11}{4}N^\frac{7}{4}}{d^\frac{9}{2}\sqrt{P}k}
\end{align}
for $\frac{B_d^3N_1^2\sqrt{N}M^\frac{3}{2}}{dPk^2}P^\varepsilon\ll H\ll BB_d^3N_1^2MP^\varepsilon$.

\subsubsection{Conclusion of the Third Case}

Inserting (\ref{SHLargeM0Bound}), (\ref{SHLargeMSmallHnon0Bound}) and (\ref{SHLargeMBigHnon0Bound}) into (\ref{SuvAfterCauchy0split}), we get \begin{align*}
    \left(S_{v,M,\tilde{N},N_1,B,B_d}^{\pm,\eta}\right)^2\ll&\frac{a^4BdP^\varepsilon}{B_d^2P^{\frac{1-3v}{2}}\sqrt{MN}N_1^2}\left(M+BP^v\right)\\
    &\times\left(\left(\frac{(MN)^\frac{3}{2}}{BB_d^2d^3N_1P^\frac{3+v}{2}}+\frac{MN^2}{Bd^5P^{1+v}k^2}\delta\left(M\ll\frac{B_d^2P^{1-v}N}{d^4k^4}P^\varepsilon, B\ll \frac{B_dN}{d^3P^vk^3}P^\varepsilon\right)\right)\left(\frac{M^\frac{3}{2}\sqrt{N}}{BdP^\frac{1+v}{2}k^2}+BP^v\right)\right.\\
    &\left.+\frac{M^\frac{3}{4}N^\frac{7}{4}}{\sqrt{B}d^\frac{9}{2}P^\frac{3+3v}{4}k}\left(M+BP^v\right)\right).
\end{align*}
Inserting the restriction $$\frac{(Bd)^2P^{1+v+\varepsilon}k^2}{N}\ll M\ll\frac{P^{1-v-\varepsilon}N}{a^2C^2d^2} \text{ and } B\ll\frac{\sqrt{MN}}{dP^\frac{1+v}{2}k}P^\varepsilon\ll\frac{N}{Cd^2P^vk}P^\varepsilon,$$
into the above bound and (\ref{SvMDyadic2}) in Theorem \ref{SvAfterVDThm}, we get \begin{align*}
    S_{v,M}\ll& P^\varepsilon\sup_{B\ll\frac{\sqrt{MN}}{dP^\frac{1+v}{2}k}P^\varepsilon}\frac{\sqrt{Bd}}{B_dP^{\frac{1-3v}{4}}M^{\frac{1}{4}}N^\frac{1}{4}}\sqrt{M+BP^v}\\
    &\times\left(\left(\frac{(MN)^\frac{3}{4}}{\sqrt{B}d^\frac{3}{2}P^\frac{3+v}{4}}+\sqrt{\frac{MN^2}{Bd^5P^{1+v}k^2}}\delta\left(M\ll\frac{B_d^2P^{1-v}N}{d^4k^4}P^\varepsilon, B\ll \frac{B_dN}{d^3P^vk^3}P^\varepsilon\right)\right)\left(\frac{M^\frac{3}{4}N^\frac{1}{4}}{\sqrt{Bd}P^\frac{1+v}{4}k}+\sqrt{BP^v}\right)\right.\\
    &\left.+\sqrt{\frac{M^\frac{3}{4}N^\frac{7}{4}}{\sqrt{B}d^\frac{9}{2}P^\frac{3+3v}{4}k}\left(M+BP^v\right)}\right)\\
    \ll&P^\varepsilon \left(\frac{M^\frac{7}{4}N^\frac{3}{4}}{d^\frac{3}{2}P^\frac{5-v}{4}k}+\frac{M^\frac{5}{4}N^\frac{3}{4}}{d^\frac{3}{2}P^\frac{5-3v}{4}k}+\frac{(MN)^\frac{3}{4}}{d^\frac{3}{2}P^\frac{5-3v}{4}\sqrt{k}}\sqrt{M+\frac{\sqrt{MN}}{dP^\frac{1-v}{2}k}}+\frac{M^\frac{1}{4}N^\frac{3}{4}}{d^2P^\frac{3-v}{4}k^\frac{3}{4}}\left(M+\frac{\sqrt{MN}}{dP^\frac{1-v}{2}k}\right)\right)\\
    \ll&P^\varepsilon\left(\frac{\sqrt{P}N^\frac{5}{2}}{C^\frac{7}{2}d^5k}+\frac{N^2}{C^\frac{5}{2}d^4\sqrt{k}}\sqrt{1+\frac{C}{Pk}}+\frac{\sqrt{P}N^2}{C^\frac{5}{2}d^\frac{9}{2}k^\frac{3}{4}}\left(1+\frac{C}{Pk}\right)+\frac{N^2}{d^\frac{9}{2}\sqrt{P}k^2(C^\frac{3}{2}+k^3)}\right).
\end{align*}

\begin{remark}
To get the second inequality, the first three terms come from choosing $\frac{(MN)^\frac{3}{4}k}{\sqrt{B}d^\frac{3}{2}P^\frac{3+v}{4}}$, and the last term comes from choosing $\sqrt{\frac{M^\frac{3}{4}N^\frac{7}{4}}{\sqrt{B}d^\frac{9}{2}P^\frac{3+3v}{4}k}\left(M+BP^v\right)}$. Note that no more terms are needed as $\sqrt{\frac{MN^2}{Bd^5P^{1+v}k^2}}$ is dominated by other terms when $M\gg \frac{B_d^2P^{1-v}N}{d^4k^4}P^\varepsilon$.
\end{remark}

By the restriction $C\leq\sqrt{\frac{N}{d^3}}\leq P^\frac{3}{4}k\leq Pk$ and restricting $k<P^2$, the above bound simplies to \begin{align}\label{SuvMlargeMBound}
    S_{v,M}\ll&P^\varepsilon\left(\frac{\sqrt{P}N^\frac{5}{2}}{C^\frac{7}{2}d^5k}+\frac{\sqrt{P}N^2}{C^\frac{5}{2}d^4k^\frac{3}{4}}+\frac{N^2}{d^\frac{9}{2}\sqrt{P}k^2(C^\frac{3}{2}+k^3)}\right).
\end{align}

Similarly for the case $\chi\neq\chi_0, v=1$, we have \begin{align}
    S_{1,M}^{\eta}\ll& P^\varepsilon\sup_{B\ll\frac{\sqrt{MN}}{dPk}P^\varepsilon}\frac{\sqrt{BdP}}{B_dM^{\frac{1}{4}}N^\frac{1}{4}}\nonumber\\
    &\times \left(\left(\frac{M^\frac{3}{4}N^\frac{1}{4}}{\sqrt{BdP}k}+1\right)\left(\frac{M^\frac{5}{4}N^\frac{3}{4}}{B_dd^\frac{3}{2}\sqrt{N_1P}}+\frac{B\sqrt{PMN}}{d^\frac{3}{2}}\delta\left(M\ll\frac{B^\frac{4}{3}B_d^\frac{2}{3}P^\frac{4}{3}}{N^\frac{1}{3}}P^\varepsilon, B\ll \frac{B_dN}{d^3Pk^3}\right)\right)+\sqrt{\frac{\sqrt{B}M^\frac{11}{4}N^\frac{7}{4}}{d^\frac{9}{2}\sqrt{P}k}}\right)\nonumber\\
    \ll& P^\varepsilon\left(\frac{N^\frac{5}{2}}{C^\frac{7}{2}d^5\sqrt{P}}+\frac{N^2}{C^\frac{5}{2}d^4\sqrt{Pk}}+\frac{\sqrt{P}N^\frac{5}{2}}{C^\frac{7}{2}d^\frac{11}{2}k}\frac{\sqrt{C}}{Pk^\frac{1}{4}}\right).
\end{align}
Note that this is dominated by (\ref{SuvMlargeMBound}) as we have the restriction $k<P^\frac{5}{4}$ giving us $C\leq\frac{\sqrt{N}}{d}\ll P^\frac{3}{4}k\leq P^2\sqrt{k}$.

\subsection{Fourth Case: Transitional \texorpdfstring{$M$}{M}}

Finally we deal with the transitional range $\frac{(Bd)^2P^{1+v-\varepsilon}k^2}{N}\ll M\ll \frac{(Bd)^2P^{1+v+\varepsilon}k^2}{N}$. Recalling (\ref{S2TransMAfterPoisson}), inserting the size of $M$, bounding everything by absolute value, (\ref{DeligneBound}) and (\ref{CharSumBound}) gives us \begin{align*}
    \left(S_{v,M,\tilde{N},N_1,B,B_d}^{\pm,\eta}\right)^2\ll&\frac{a^2P^\varepsilon}{B_d^2P^{1-v}N_1^2k}\nonumber\\
    &\times\sum_{\substack{B\leq b_1\leq 2B\\ B_d\leq (b_1,d)\leq 2B_d\\
    (P^{1-v},b_1)=1}}\sum_{\substack{B\leq b_2\leq 2B\\ B_d\leq (b_2,d)\leq 2B_d\\
    (P^{1-v},b_2)=1}}\sum_{u_1|b_1P^v}\sum_{u_2|b_2P^v}u_1u_2\sum_{\substack{N_1\leq n_1\leq 2N_1\\
    n_1|b_{1,d}P^v}}\sum_{\substack{N_1\leq n_2\leq 2N_1\\
    n_2|b_{2,d}P^v}}\\
    &\times \sumast_{\beta_1\Mod{\frac{b_{1,d}P^v}{n_1}}}\sumast_{\beta_2\Mod{\frac{b_{2,d}P^v}{n_2}}}\sum_{n}\sum_{m_1\sim M}\sum_{m_2\sim M}\delta\left((*)\right)\sup_{\sigma=\pm1}\sup_{\substack{T=0 \text{ or }\\ P^{-\varepsilon}\leq T\leq k^\frac{1}{3}P^{-\varepsilon}}}\left|\mathcal{J}_{\sigma T}\right|,
\end{align*}
with $(*)$ the set of congruences in (\ref{*}).

Similarly, for the case $\chi\neq \chi_0, v=1$, we have \begin{align*}
    \left(S_{1,M,\tilde{N},N_1,B,B_d}^{\pm,\eta}\right)^2\ll&\frac{a^2P^\varepsilon}{B_d^2N_1^2k}BP\sum_{\substack{B\leq b_1\leq 2B\\ B_d\leq (b_1,d)\leq 2B_d}}\sum_{\substack{B\leq b_2\leq 2B\\ B_d\leq (b_2,d)\leq 2B_d}}\sum_{\substack{N_1\leq n_1\leq 2N_1\\
    n_1|b_{1,d}P}}\sum_{\substack{N_1\leq n_2\leq 2N_1\\
    n_2|b_{2,d}P}}\\
    &\times \sumast_{\beta_1\Mod{\frac{b_{1,d}P}{n_1}}}\sumast_{\beta_2\Mod{\frac{b_{2,d}P}{n_2}}}\sum_{n}\sum_{m_1\sim M}\sum_{m_2\sim M}\delta\left((**)\right)\sup_{\sigma=\pm1}\sup_{\substack{T=0 \text{ or }\\ P^{-\varepsilon}\leq T\leq k^\frac{1}{3}P^{-\varepsilon}}}\left|\mathcal{J}_{\sigma T}\right|,
\end{align*}
with \begin{align*}
    (**)=n\equiv\mp\beta_1\frac{b_{2,d}P^v}{n_2}\pm\beta_2\frac{b_{1,d}P^v}{n_1}\Mod{\frac{b_{1,d}b_{2,d}P^{2v}}{n_1n_2}}
\end{align*}
as defined above.

Denote $S_{v,M,\tilde{N},N_1,B,B_d,0}^{\pm,\eta}$ and $S_{v,M,\tilde{N},N_1,B,B_d,\sigma T}^{\pm,\eta}$ be the contribution of $\mathcal{J}_0$ and $\mathcal{J}_{\sigma T}$ to \newline $S_{v,M,\tilde{N},N_1,B,B_d}^{\pm,\eta}$ respectively. The rest of treatment is almost identical to the first and third case, with some slight complications.

\subsubsection{Sub Case: $T=0$}

For any $A>0$, recalling (\ref{PsiDef}) and (\ref{UltIntegralDef}), we apply Lemma \ref{HugeIntegralAnalysisTransit} on $\mathcal{J}_0$ to get the following:

If $\frac{N}{B^4d^2P^{1+v}k}\left|m_2t_2^2b_1^2-m_1t_1^2b_2^2\right|\ll P^\varepsilon$, we have
\begin{align}\label{IntegralConclusionTransM0SmallH}
    \mathcal{J}_0\ll k^\frac{10}{3}P^\varepsilon\delta\left(|n|\ll\frac{B_dN}{Bd^3P^vk^3}P^\varepsilon, \frac{(BdP^vk)^3}{B_d^3NN_1^2}P^{-\varepsilon}\ll\tilde{N}\ll\frac{(BdP^vk)^3}{B_d^3NN_1^2}P^\varepsilon\right)+O\left(P^{-A}\right),
\end{align}
and if $\frac{N}{B^4d^2P^{1+v}k}\left|m_2t_2^2b_1^2-m_1t_1^2b_2^2\right|\gg P^\varepsilon$, we have \begin{align}\label{IntegralConclusionTransM0LargeH}
    \mathcal{J}_0\ll& k^\frac{10}{3}P^\varepsilon\left(\frac{\sqrt{N}}{B^3\sqrt{M}dP^\frac{1+v}{2}}\left|m_2t_2^2b_1^2-m_1t_1^2b_2^2\right|\right)^{-\frac{1}{2}}\delta\left(\frac{(BdP^vk)^3}{B_d^3NN_1^2}P^{-\varepsilon}\ll\tilde{N}\ll\frac{(BdP^vk)^3}{B_d^3NN_1^2}P^\varepsilon,\right.\nonumber\\
    &\left.\frac{B_dN^2}{B^5d^5P^{1+2v}k^4}P^{-\varepsilon}\left|m_2t_2^2b_1^2-m_1t_1^2b_2^2\right|\ll |n|\ll \frac{B_dN^2}{B^5d^5P^{1+2v}k^4}P^\varepsilon\left|m_2t_2^2b_1^2-m_1t_1^2b_2^2\right|\right)+O\left(P^{-A}\right).
\end{align}

Separating the case $\left|m_2t_2^2b_1^2-m_1t_1^2b_2^2\right|\leq\frac{B^4d^2P^{1+v+\varepsilon}k}{N}$ and performing dyadic subdivision for $\left|m_2t_2^2b_1^2-m_1t_1^2b_2^2\right|\geq \frac{B^4d^2P^{1+v+\varepsilon}k}{N}$, we get for any $A>0$, \begin{align}\label{SuvTransM0AfterCauchy0split}
    \left(S_{v,M,\tilde{N},N_1,B,B_d,0}^{\pm,\eta}\right)^2\ll&\frac{a^2P^\varepsilon}{B_d^2P^{1-v}N_1^2k}\left(\mathcal{S}_{v,M,\tilde{N},N_1,B,B_d,0,0}^{\pm,\eta}+\sup_{\frac{B^4d^2P^{1+v+\varepsilon}k}{N}\leq H_0\leq \frac{B^4d^2P^{1+v+\varepsilon}k^2}{N}}\mathcal{S}_{v,M,\tilde{N},N_1,B,B_d,0,H_0}^{\pm,\eta}\right),
\end{align}
where \begin{align}
    \mathcal{S}_{0,0}:=&\mathcal{S}_{v,M,\tilde{N},N_1,B,B_d,0,0}^{\pm,\eta}\nonumber\\
    =&\sum_{\substack{B\leq b_1\leq 2B\\ B_d\leq (b_1,d)\leq 2B_d\\
    (P^{1-v},b_1)=1}}\sum_{\substack{B\leq b_2\leq 2B\\ B_d\leq (b_2,d)\leq 2B_d\\
    (P^{1-v},b_2)=1}}\sum_{u_1|b_1P^v}\sum_{u_2|b_2P^v}u_1u_2\sum_{\substack{N_1\leq n_1\leq 2N_1\\
    n_1|b_{1,d}P^v}}\sum_{\substack{N_1\leq n_2\leq 2N_1\\
    n_2|b_{2,d}P^v}}\sumast_{\beta_1\Mod{\frac{b_{1,d}P^v}{n_1}}}\sumast_{\beta_2\Mod{\frac{b_{2,d}P^v}{n_2}}}\nonumber\\
    &\times\sum_n\sum_{m_1\sim M}\sum_{m_2\sim M}\delta\left( \left|m_2t_2^2b_1^2-m_1t_1^2b_2^2\right|\leq \frac{B^4d^2P^{1+v+\varepsilon}k}{N}, (*)\right)\left|\mathcal{J}_0\right|.
\end{align}
and \begin{align}
    \mathcal{S}_{0,H_0}:=&\mathcal{S}_{v,M,\tilde{N},N_1,B,B_d,0,H_0}^{\pm,\eta}\nonumber\\
    =&\sum_{\substack{B\leq b_1\leq 2B\\ B_d\leq (b_1,d)\leq 2B_d\\
    (P^{1-v},b_1)=1}}\sum_{\substack{B\leq b_2\leq 2B\\ B_d\leq (b_2,d)\leq 2B_d\\
    (P^{1-v},b_2)=1}}\sum_{u_1|b_1P^v}\sum_{u_2|b_2P^v}u_1u_2\sum_{\substack{N_1\leq n_1\leq 2N_1\\
    n_1|b_{1,d}P^v}}\sum_{\substack{N_1\leq n_2\leq 2N_1\\
    n_2|b_{2,d}P^v}}\sumast_{\beta_1\Mod{\frac{b_{1,d}P^v}{n_1}}}\sumast_{\beta_2\Mod{\frac{b_{2,d}P^v}{n_2}}}\nonumber\\
    &\times\sum_n\sum_{m_1\sim M}\sum_{m_2\sim M}\delta\left(H_0\leq \left|m_2t_2^2b_1^2-m_1t_1^2b_2^2\right|\leq 2H_0, (*)\right)\left|\mathcal{J}_0\right|.
\end{align}

By (\ref{IntegralConclusionTransM0SmallH}) and (\ref{IntegralConclusionTransM0LargeH}), $n=0$ can only occur in $\mathcal{S}_{0,0}$ and with the same analysis done in the diagonal $n=0$ sub case in the previous cases, we have the contribution of $n=0$ is bounded by \begin{align}\label{S00Bound}
    \ll &k^\frac{10}{3}P^\varepsilon\frac{B}{B_d}\frac{BP^v}{B_dN_1}\sup_{u_1,u_2\ll BP^v}u_1u_2\left(\frac{M}{u_1}+1\right)\left(\min\left\{\frac{B^2d^2P^{1+v}k}{u_2N},\frac{M}{u_2}\right\}+1\right)\nonumber\\
    \ll&P^\varepsilon\frac{B^2P^vk^\frac{10}{3}}{B_d^2N_1}\left(\frac{B^2d^2P^{1+v}k^2}{N}+BP^v\right)\left(\frac{B^2d^2P^{1+v}k}{N}+BP^v\right).
\end{align}
Here we used \begin{align*}
    \left|m_2t_2^2b_1^2-m_1t_1^2b_2^2\right|\leq \frac{B^4d^2P^{1+v+\varepsilon}k}{N} \Rightarrow \left|m_2-\frac{m_1t_1^2b_2^2}{t_2^2b_1^2}\right|\leq \frac{B^2d^2P^{1+v+\varepsilon}k}{N}.
\end{align*}
For $n\neq0$, same argument as the sub case $n\neq0$ in previous cases but with (\ref{IntegralConclusionTransM0SmallH}) and (\ref{IntegralConclusionTransM0LargeH}) gives us the contribution of $n\neq0$ to $\mathcal{S}_{0,0}$ is bounded by \begin{align}\label{temp}
    \ll& k^\frac{10}{3}P^\varepsilon\sum_{c_0\ll\frac{BP^v}{B_dN_1}}\frac{B_dN}{c_0Bd^3P^vk^3}\sup_{c_{1,0},c_{2,0}\ll \frac{BP^v}{c_0B_dN_1}}\min\left\{\frac{B}{B_d},\frac{BP^v}{c_0c_{1,0}}\right\}\min\left\{\frac{B}{B_d},\frac{BP^v}{c_0c_{2,0}}\right\}c_0c_{1,0}\nonumber\\
    &\times \sup_{u_1,u_2\ll BP^v}u_1u_2\left(\frac{M}{u_1}+1\right)\left(\min\left\{\frac{B^2d^2P^{1+v}k}{u_2N},\frac{M}{u_2}\right\}+1\right)\delta\left(B\ll\frac{B_dN}{d^3P^vk^3}P^\varepsilon\right)\nonumber\\
    \ll&P^\varepsilon\frac{BNk^\frac{1}{3}}{d^3}\left(\frac{B^2d^2P^{1+v}k^2}{N}+BP^v\right)\left(\frac{B^2d^2P^{1+v}k}{N}+BP^v\right)\delta\left(B\ll\frac{B_dN}{d^3P^vk^3}P^\varepsilon\right).
\end{align}
and the contribution of $n\neq0$ to $\mathcal{S}_{0,H_0}$ is bounded by \begin{align}\label{S0non0Bound}
    \ll&k^\frac{10}{3}P^\varepsilon\left(\frac{\sqrt{N}H_0}{B^3\sqrt{M}dP^\frac{1+v}{2}}\right)^{-\frac{1}{2}}\sum_{c_0\ll\frac{BP^v}{B_dN_1}}\frac{B_dN^2H_0}{c_0B^5d^5P^{1+2v}k^4}\sup_{c_{1,0},c_{2,0}\ll \frac{BP^v}{c_0B_dN_1}}\min\left\{\frac{B}{B_d},\frac{BP^v}{c_0c_{1,0}}\right\}\min\left\{\frac{B}{B_d},\frac{BP^v}{c_0c_{2,0}}\right\}c_0c_{1,0}\nonumber\\
    &\times \sup_{u_1,u_2\ll BP^v}u_1u_2\left(\frac{M}{u_1}+1\right)\left(\min\left\{\frac{H_0}{u_2B^2},\frac{M}{u_2}\right\}+1\right)\nonumber\\
    \ll&P^\varepsilon\frac{BNk^\frac{5}{6}}{d^3}\left(\frac{B^2d^2P^{1+v}k^2}{N}+BP^v\right)^2.
\end{align}
Here we used \begin{align*}
    \left|m_2t_2^2b_1^2-m_1t_1^2b_2^2\right|\leq H_0 \Rightarrow \left|m_2-\frac{m_1t_1^2b_2^2}{t_2^2b_1^2}\right|\leq \frac{H_0}{B^2}.
\end{align*}
Notice that (\ref{S0non0Bound}) dominates (\ref{temp}). Putting (\ref{S00Bound}) and (\ref{S0non0Bound}) into $\mathcal{S}_{0,0}, \mathcal{S}_{0,H_0}$ and (\ref{SuvTransM0AfterCauchy0split}), we get \begin{align}\label{SuvTransM0Bound}
    \left(S_{v,M,\tilde{N},N_1,B,B_d,0}^{\pm,\eta}\right)^2\ll&\frac{a^2P^\varepsilon}{B_d^2P^{1-v}N_1^2k}\nonumber\\
    &\times\left(\frac{B^2P^vk^\frac{10}{3}}{B_d^2N_1}\left(\frac{B^2d^2P^{1+v}k}{N}+BP^v\right)+\frac{BNk^\frac{5}{6}}{d^3}\left(\frac{B^2d^2P^{1+v}k^2}{N}+BP^v\right)\right)\left(\frac{B^2d^2P^{1+v}k^2}{N}+BP^v\right).
\end{align}

Performing the analogous treatment similar to the previous cases for the case $\chi\neq\chi_0, v=1$, we have \begin{align}\label{SuvchiTransM0Bound}
    \left(S_{1,M,\tilde{N},N_1,B,B_d,0}^{\pm,\eta}\right)^2\ll&\frac{a^2P^\varepsilon}{B_d^2N_1^2k}\left(\frac{B^5d^2P^4k^\frac{16}{3}}{B_d^2N_1N}\left(\frac{B^2d^2P^2k}{N}+1\right)+\frac{B^4P^3k^\frac{17}{6}}{d}\left(\frac{B^2d^2P^2k^2}{N}+1\right)\right).
\end{align}

\subsubsection{Sub Case: $P^{-\varepsilon}\leq T\leq k^\frac{1}{3}P^{-\varepsilon}$}

For any $A>0$, recalling (\ref{PsiDef}) and (\ref{UltIntegralDef}), we apply Lemma \ref{HugeIntegralAnalysisTransit} on $\mathcal{J}_{\sigma T}$ to get the following:

If $\frac{N\left|b_1m_1(b_2,d)^3n_2^2-b_2m_2(b_1,d)^3n_1^2\right|}{B^3B_d^3d^2N_1^2P^{1+v}k}\ll P^\varepsilon$, we have \begin{align}\label{IntegralConclusionTransMnon0SmallH}
    \mathcal{J}_{\sigma T}\ll&k^3TP^\varepsilon\delta\left(|n|\ll\frac{B_dN}{Bd^3P^vk^3}P^\varepsilon, \frac{(BdP^vk)^3}{B_d^3NN_1^2}P^{-\varepsilon}\ll\tilde{N}\ll\frac{(BdP^vk)^3}{B_d^3NN_1^2}P^\varepsilon\right)+P^{-A},
\end{align}
and if $\frac{N\left|b_1m_1(b_2,d)^3n_2^2-b_2m_2(b_1,d)^3n_1^2\right|}{B^3B_d^3d^2N_1^2P^{1+v}k}\gg P^\varepsilon$, we have \begin{align}\label{IntegralConclusionTransMnon0LargeH}
    \mathcal{J}_{\sigma T}\ll&\frac{B^\frac{3}{2}B_d^\frac{3}{2}dN_1P^\frac{1+v}{2}k^\frac{7}{2}T}{\sqrt{N}}\left|b_1m_1(b_2,d)^3n_2^2-b_2m_2(b_1,d)^3n_1^2\right|^{-\frac{1}{2}}P^\varepsilon\nonumber\\
    &\times \delta\left(\frac{N^2\left|b_1m_1(b_2,d)^3n_2^2-b_2m_2(b_1,d)^3n_1^2\right|}{B^4B_d^2d^5N_1^2P^{1+2v}k^4}P^{-\varepsilon}\ll |n|\ll \frac{N^2\left|b_1m_1(b_2,d)^3n_2^2-b_2m_2(b_1,d)^3n_1^2\right|}{B^4B_d^2d^5N_1^2P^{1+2v}k^4}P^\varepsilon, \right.\nonumber\\
    &\left.\frac{(BdP^vk)^3}{B_d^3NN_1^2}P^{-\varepsilon}\ll\tilde{N}\ll\frac{(BdP^vk)^3}{B_d^3NN_1^2}P^\varepsilon\right)+P^{-A}.
\end{align}

Now we apply dyadic subdivision as in the first three cases and apply the same analysis as all the above cases. Separating the case $\left|b_1m_1(b_2,d)^3n_2^2-b_2m_2(b_1,d)^3n_1^2\right|=0$ and performing dyadic subdivision for \newline $\left|b_1m_1(b_2,d)^3n_2^2-b_2m_2(b_1,d)^3n_1^2\right|\neq 0$, we get for any $A>0$, \begin{align}\label{SuvTransMAfterCauchy0split}
    \left(S_{v,M,\tilde{N},N_1,B,B_d,\sigma T}^{\pm,\eta}\right)^2\ll&\frac{a^2P^\varepsilon}{B_d^2P^{1-v}N_1^2k}\left(\mathcal{S}_{v,M,\tilde{N},N_1,B,B_d,\sigma T,0}^{\pm,\eta}+\sup_{1\leq H\leq BB_d^3MP^\varepsilon}\mathcal{S}_{v,M,\tilde{N},N_1,B,B_d,\sigma T,H}^{\pm,\eta}\right),
\end{align}
where \begin{align}\label{SHTransMT0Bound}
    \mathcal{S}_{\sigma T,H}:=&\mathcal{S}_{v,M,\tilde{N},N_1,B,B_d,\sigma T,H}^{\pm,\eta}\nonumber\\
    =&\sum_{\substack{B\leq b_1\leq 2B\\ B_d\leq (b_1,d)\leq 2B_d\\
    (P^{1-v},b_1)=1}}\sum_{\substack{B\leq b_2\leq 2B\\ B_d\leq (b_2,d)\leq 2B_d\\
    (P^{1-v},b_2)=1}}\sum_{u_1|b_1P^v}\sum_{u_2|b_2P^v}u_1u_2\sum_{\substack{N_1\leq n_1\leq 2N_1\\
    n_1|b_{1,d}P^v}}\sum_{\substack{N_1\leq n_2\leq 2N_1\\
    n_2|b_{2,d}P^v}}\sumast_{\beta_1\Mod{\frac{b_{1,d}P^v}{n_1}}}\sumast_{\beta_2\Mod{\frac{b_{2,d}P^v}{n_2}}}\nonumber\\
    &\times\sum_n\sum_{m_1\sim M}\sum_{m_2\sim M}\delta\left(H\leq \left|b_1m_1(b_2,d)^3n_2^2-b_2m_2(b_1,d)^3n_1^2\right|\leq 2H, (*)\right)\left|\mathcal{J}_{\sigma T}\right|.
\end{align}
Here the definition includes $H=0$. Applying the same analysis as the previous cases but with (\ref{IntegralConclusionTransMnon0SmallH}) and (\ref{IntegralConclusionTransMnon0LargeH}), we get $n=0$ only occurs when $\frac{NH}{B^3B_d^3d^2N_1^2P^{1+v}k}\ll P^\varepsilon$ and the contribution of $n=0$ to $\mathcal{S}_{\sigma T,H}$ is bounded by \begin{align}
    \ll& k^3TP^\varepsilon\frac{B}{B_d}\frac{BP^v}{B_dN_1}\sup_{u_1,u_2\ll BP^v}u_1u_2\left(\frac{M}{u_1}+1\right)\left(\min\left\{\frac{H}{u_2BB_d^3N_1^2},\frac{M}{u_2}\right\}+1\right)\nonumber\\
     \ll&P^\varepsilon\frac{B^2P^vk^\frac{10}{3}}{B_d^2N_1}\left(\frac{B^2d^2P^{1+v}k^2}{N}+BP^v\right)\left(\frac{B^2d^2P^{1+v}k}{N}+BP^v\right),
\end{align}
and the contribution of $n\neq0$ to $\mathcal{S}_{\sigma T,H}$ is bounded by:
for $\frac{NH}{B^3B_d^3d^2N_1^2P^{1+v}k}\ll P^\varepsilon$, \begin{align}\label{temp2}
    \ll&P^\varepsilon k^3T\sum_{c_0\ll\frac{BP^v}{B_dN_1}}\frac{B_dN}{c_0Bd^3P^vk^3}\sup_{c_{1,0},c_{2,0}\ll \frac{BP^v}{c_0B_dN_1}}\min\left\{\frac{B}{B_d},\frac{BP^v}{c_0c_{1,0}}\right\}\min\left\{\frac{B}{B_d},\frac{BP^v}{c_0c_{2,0}}\right\}c_0c_{1,0}\nonumber\\
    &\times \sup_{u_1,u_2\ll BP^v}u_1u_2\left(\frac{M}{u_1}+1\right)\left(\min\left\{\frac{H}{u_2BB_d^3N_1^2},\frac{M}{u_2}\right\}+1\right)\delta\left(\right)\nonumber\\
    \ll&P^\varepsilon\frac{BNk^\frac{1}{3}}{d^3}\left(\frac{B^2d^2P^{1+v}k^2}{N}+BP^v\right)\left(\frac{B^2d^2P^{1+v}k}{N}+BP^v\right)\delta\left(B\ll\frac{B_dN}{d^3P^vk^3}P^\varepsilon\right),
\end{align}
and for $\frac{NH}{B^3B_d^3d^2N_1^2P^{1+v}k}\gg P^\varepsilon$
\begin{align}\label{SHTransMTnon0Bound}
    \ll&P^\varepsilon\frac{B^\frac{3}{2}B_d^\frac{3}{2}dN_1P^\frac{1+v}{2}k^\frac{7}{2}T}{\sqrt{NH}}\sum_{c_0\ll\frac{BP^v}{B_dN_1}}\frac{N^2H}{c_0B^4B_d^2d^5N_1^2P^{1+2v}k^4}\sup_{c_{1,0},c_{2,0}\ll \frac{BP^v}{c_0B_dN_1}}\min\left\{\frac{B}{B_d},\frac{BP^v}{c_0c_{1,0}}\right\}\min\left\{\frac{B}{B_d},\frac{BP^v}{c_0c_{2,0}}\right\}c_0c_{1,0}\nonumber\\
    &\times \sup_{u_1,u_2\ll BP^v}u_1u_2\left(\frac{M}{u_1}+1\right)\left(\min\left\{\frac{H}{u_2BB_d^3N_1^2},\frac{M}{u_2}\right\}+1\right)\nonumber\\
    \ll&P^\varepsilon\frac{BNk^\frac{5}{6}}{d^3}\left(\frac{B^2d^2P^{1+v}k^2}{N}+BP^v\right)^2.
\end{align}

Note that (\ref{temp2}) is being dominated by (\ref{SHTransMTnon0Bound}), hence (\ref{SHTransMT0Bound}) and (\ref{SHTransMTnon0Bound}) gives us \begin{align}\label{SuvTransMTBound}
    \left(S_{v,M,\tilde{N},N_1,B,B_d,\sigma T}^{\pm,\eta}\right)^2\ll&\frac{a^2P^\varepsilon}{B_d^2P^{1-v}N_1^2k}\nonumber\\
    &\times\left(\frac{B^2P^vk^\frac{10}{3}}{B_d^2N_1}\left(\frac{B^2d^2P^{1+v}k}{N}+BP^v\right)+\frac{BNk^\frac{5}{6}}{d^3}\left(\frac{B^2d^2P^{1+v}k^2}{N}+BP^v\right)\right)\left(\frac{B^2d^2P^{1+v}k^2}{N}+BP^v\right),
\end{align}
which is the exact same bound as (\ref{SuvTransM0Bound}).

Performing the analogous treatment similar to the previous cases for the case $\chi\neq\chi_0, v=1$, we have \begin{align}\label{SuvchiTransMTBound}
    \left(S_{1,M,\tilde{N},N_1,B,B_d,\sigma T}^{\pm,\eta}\right)^2\ll&\frac{a^2P^\varepsilon}{B_d^2N_1^2k}\left(\frac{B^5d^2P^4k^\frac{16}{3}}{B_d^2N_1N}\left(\frac{B^2d^2P^2k}{N}+1\right)+\frac{B^4P^3k^\frac{17}{6}}{d}\left(\frac{B^2d^2P^2k^2}{N}+1\right)\right),
\end{align}
which again is the exact same bound as (\ref{SuvchiTransM0Bound}).

\subsubsection{Conclusion of the Transitional Case}

Combining (\ref{SuvTransM0Bound}) and (\ref{SuvTransMTBound}), we get together with $M\ll\frac{P^{1-v+\varepsilon}N}{a^2C^2d^2}$ and the size of $M$ in this case implying $B\ll\frac{N}{aCd^2P^vk}P^\varepsilon$, \begin{align*}
    \left(S_{v,M,\tilde{N},N_1,B,B_d}^{\pm,\eta}\right)^2\ll&\frac{P^\varepsilon}{P^{1-v}k}\left(\frac{N^2k^\frac{4}{3}}{C^2d^4P^v}\left(\frac{NP^{1-v}}{C^2d^2k}+\frac{N}{Cd^2k}\right)+\frac{N^2}{Cd^5P^vk^\frac{1}{6}}\left(\frac{NP^{1-v}}{C^2d^2}+\frac{N}{Cd^2k}\right)\right)\left(\frac{NP^{1-v}}{C^2d^2}+\frac{N}{Cd^2k}\right)\nonumber\\
    \ll&P^\varepsilon\frac{N^4}{C^5d^8P^vk^\frac{2}{3}}\left(1+\frac{P^{1-v}}{C}+\frac{P^{1-v}}{d\sqrt{k}}+\frac{C}{dk^\frac{3}{2}}\right)\left(1+\frac{C}{P^{1-v}k}\right).
\end{align*}
Putting this into (\ref{SvMDyadic2}) in Theorem \ref{SvAfterVDThm}, with the restriction $C\ll\frac{\sqrt{N}}{d}\leq P^\frac{3}{4}k$ and choosing $C\gg P$, for $$\frac{(Bd)^2P^{1+v-\varepsilon}k^2}{N}\ll M\ll \frac{(Bd)^2P^{1+v+\varepsilon}k^2}{N},$$
we get \begin{align}\label{SuvMTransMBound}
    S_{v,M}\ll P^\varepsilon\left(\frac{N^2}{C^\frac{5}{2}d^4k^\frac{1}{3}}+\frac{\sqrt{P}N^2}{C^\frac{5}{2}d^\frac{9}{2}k^\frac{7}{12}}\right).
\end{align}

Similarly for the case $\chi\neq\chi_0, v=1$, we have \begin{align}
    S_{1,M}^{\eta}\ll P^\varepsilon\left(\frac{N^2}{C^\frac{5}{2}d^4\sqrt{P}k^\frac{1}{3}}+\frac{N^\frac{5}{2}}{C^\frac{7}{2}d^5\sqrt{P}k^\frac{5}{6}}+\frac{N^\frac{5}{2}}{C^3d^\frac{11}{2}\sqrt{P}k^\frac{13}{12}}\right).
\end{align}
Note that this is dominated by \ref{SuvMTransMBound} as restriction of $N\ll P^{\frac{3}{2}+\varepsilon}k^2, k<P^\frac{5}{4}$ and choosing $C\gg P$ together gives us \begin{align*}
    \frac{\sqrt{N}}{CPk^\frac{1}{4}}\ll\frac{k^\frac{3}{4}}{CP^\frac{1}{4}}P^\varepsilon\ll \frac{P^{\frac{11}{16}+\varepsilon}}{C}\ll P^{-\varepsilon} \text{ and } \sqrt{\frac{N}{CP^2k}}\ll P^\varepsilon\sqrt{\frac{k}{C\sqrt{P}}}\ll P^{-\varepsilon}.
\end{align*}

\subsection{Combining all Bounds and Conclusion}

Putting (\ref{SuvMSmallMBound}), (\ref{SuvM0MBound}), (\ref{SuvMlargeMBound}) and (\ref{SuvMTransMBound}) into (\ref{SuvDyadic}) and then (\ref{SNDeltaSplit}), clearing the terms are are smaller, we get for $\max\left\{\frac{\sqrt{N}}{d\sqrt{P}},P\right\}\ll C\ll\frac{\sqrt{N}}{d}$, $d^3<k<P^\frac{5}{4}$, \begin{align}
    S_d(N)\ll& P^\varepsilon\left(\frac{C^\frac{3}{2}\sqrt{N}}{d^\frac{3}{2}\sqrt{P}}+\frac{\sqrt{CN}k}{d}+\frac{\sqrt{P}N^\frac{5}{4}}{Cd^3}+\frac{C^2k}{\sqrt{P}}+\frac{\sqrt{P}N^\frac{5}{2}}{C^\frac{7}{2}d^5k}+\frac{N^2}{C^\frac{5}{2}d^4k^\frac{1}{3}}+\frac{\sqrt{P}N^2}{C^\frac{5}{2}d^\frac{9}{2}k^\frac{7}{12}}\right).
\end{align}
Note that this includes the case $\chi\neq\chi_0$ as the bound for this case is the same when $v=0$, and smaller when $v=1$ as shown in the previous subsections. By further restricting $P>d^\frac{8}{3}$ and $P^\frac{1}{4}<k<P^\frac{21}{17}$ and using $N\ll P^{\frac{3}{2}+\varepsilon}k^2$ implies that the above bound simplifies to \begin{align}\label{temptemp}
    S_d(N)\ll& P^\varepsilon\left(\frac{C^\frac{3}{2}\sqrt{N}}{d^\frac{3}{2}\sqrt{P}}+\frac{\sqrt{CN}k}{d}+\frac{\sqrt{P}N^\frac{5}{4}}{Cd^3}+\frac{C^2k}{\sqrt{P}}+\frac{\sqrt{P}N^2}{C^\frac{5}{2}d^\frac{9}{2}k^\frac{7}{12}}\right).
\end{align}
To optimise the bound and achieve Theorem \ref{mainthm}, we separate into two cases.

Case 1: For $P^\frac{1}{4}< k\leq P^\frac{6}{5}$, we choose \begin{align}\label{CChoice}
    C=\min\left\{\frac{P^\frac{2}{5}N^\frac{3}{10}}{d^\frac{3}{5}},\frac{P^\frac{1}{3}\sqrt{N}}{dk^\frac{2}{3}},\frac{P^\frac{1}{3}N^\frac{5}{12}}{dk^\frac{1}{3}},\frac{\sqrt{N}}{d}\right\}.
\end{align}
Restricting $P^\frac{1}{4}<k\leq P^\frac{6}{5}$, one can check with direct computations that such choice satisfies \begin{align}
    \frac{\sqrt{N}}{d\sqrt{P}}+\frac{\sqrt{N}}{dk^\frac{7}{18}}+P\ll C\ll\frac{\sqrt{N}}{d} \text{ when } (Pk)^\frac{4}{3}d^2+P^\frac{8}{5}k^\frac{4}{5}d^\frac{12}{5}+P^2d^2\ll N\ll P^{\frac{3}{2}+\varepsilon}k^2.
\end{align}
Observe that such a choice is almost obtained by minimising the first four terms, with the power of $d$ adjusted slightly for ease of computation. The restriction of $k\leq P^\frac{6}{5}$ ensures that $C\gg \frac{\sqrt{N}}{dk^\frac{7}{18}}$ and that makes the fifth term in (\ref{temptemp}) smaller than the maximum of the first four terms. Such choice of $C$ is the biggest $C$ such that the third term $\frac{\sqrt{P}N^\frac{5}{4}}{Cd^3}$ is the biggest term. Hence such choice of $C$ gives us \begin{align*}
    S_d(N)\ll P^\varepsilon\frac{\sqrt{P}N^\frac{5}{4}}{Cd^3}.
\end{align*}
Substituting (\ref{CChoice}) into the above we get \begin{align}
    S_d(N)\ll P^\varepsilon\left(\frac{P^\frac{1}{10}N^\frac{19}{20}}{d^\frac{12}{5}}+\frac{P^\frac{1}{6}N^\frac{3}{4}k^\frac{2}{3}}{d^\frac{3}{2}}+\frac{P^\frac{1}{6}N^\frac{5}{6}k^\frac{1}{3}}{d^2}+\frac{\sqrt{P}N^\frac{3}{4}}{d^2}\right).
\end{align}

Case 2: For $P^\frac{6}{5}\leq k< P^\frac{21}{17}$, we choose \begin{align}
    C=\frac{P^\frac{1}{6}\sqrt{N}}{d^\frac{7}{6}k^\frac{19}{36}}.
\end{align}
Such choice is made by balancing $\frac{\sqrt{CN}k}{d}$ and $\frac{\sqrt{P}N^2}{C^\frac{5}{2}d^\frac{9}{2}k^\frac{7}{12}}$ and one can check by direction computations that such choice satisfies \begin{align*}
    \frac{\sqrt{N}}{d\sqrt{P}}+P\ll C=\frac{P^\frac{1}{6}\sqrt{N}}{d^\frac{7}{6}k^\frac{19}{36}}\ll\min\left\{\frac{\sqrt{N}}{dk^\frac{7}{18}},\frac{P^\frac{2}{5}N^\frac{3}{10}}{d^\frac{3}{5}},\frac{P^\frac{1}{3}N^\frac{5}{12}}{dk^\frac{1}{3}}\right\}
\end{align*}
when \begin{align}
    P>d^{24} \text{ and } (Pk)^\frac{4}{3}d^2+P^\frac{5}{3}k^\frac{19}{18}d^\frac{7}{3}+P^2d^2\ll N\ll P^{\frac{3}{2}+\varepsilon}k^2.
\end{align}
Therefore we have the bound \begin{align}
    S_d(N)\ll P^\varepsilon\frac{P^\frac{1}{12}N^\frac{3}{4}k^\frac{53}{72}}{d^\frac{19}{12}}.
\end{align}

This concludes the proof of Theorem \ref{mainthm}.

\section{Proof of Corollary \ref{maincor}}\label{maincorSect}

Finally, we apply Theorem \ref{mainthm} and Lemma \ref{SetupLemma} to prove our goal, Corollary \ref{maincor}.

For the case $P^\frac{1}{4}<k\leq P^\frac{6}{5}$, (\ref{trivialBound}) and (\ref{mainBound1}) in Theorem \ref{mainthm} gives us \begin{align*}
    \sup_{\frac{1}{2}\leq N\leq P^{3/2+\varepsilon}k^2}\sum_{(d,P)=1}\frac{S_d(N)}{\sqrt{N}}\ll& P^\varepsilon\sum_{d}\left(\frac{P^\frac{31}{40}k^\frac{9}{10}}{d^\frac{12}{5}}+\frac{P^\frac{13}{24}k^\frac{7}{6}}{d^\frac{3}{2}}+\frac{P^\frac{2}{3}k}{d^2}+\frac{P^\frac{7}{8}\sqrt{k}}{d^2}\right)\\
    &+P^\varepsilon\sum_{d}\frac{P^\frac{3}{4}k}{d^3}\delta\left(d\gg k^\frac{1}{3} \text{ or } d\gg P^\frac{3}{8}\right)+P^\varepsilon\sum_{d}\frac{(Pk)^\frac{2}{3}d+P^\frac{4}{5}k^\frac{2}{5}d^\frac{6}{5}+Pd}{d^3}\\
    \ll&P^\varepsilon\left(P^{\frac{31}{40}}k^{\frac{9}{10}}+P^{\frac{13}{24}}k^{\frac{7}{6}}+P^{\frac{2}{3}}k\right).
\end{align*}

Similarly for the case $P^\frac{6}{5}\leq k <P^\frac{21}{17}$, the (\ref{trivialBound}) and (\ref{mainBound2}) in Theorem \ref{mainthm} yields \begin{align*}
    \sup_{\frac{1}{2}\leq N\leq P^{3/2+\varepsilon}k^2}\sum_{(d,P)=1}\frac{S_d(N)}{\sqrt{N}}\ll& P^\varepsilon\sum_{d}\frac{P^\frac{11}{24}k^\frac{89}{72}}{d^\frac{19}{12}}+P^\varepsilon\sum_{d}\frac{P^\frac{3}{4}k}{d^3}\delta\left(d\gg k^\frac{1}{3} \text{ or } d\gg P^\frac{1}{24}\right)\\
    &+P^\varepsilon\sum_{d}\frac{(Pk)^\frac{2}{3}d+P^\frac{5}{6}k^\frac{19}{36}d^\frac{7}{6}+Pd}{d^3}\\
    \ll&P^{\frac{11}{24}+\varepsilon}k^{\frac{89}{72}}.
\end{align*}

Combining the above bounds with Lemma \ref{SetupLemma} yields Corollary \ref{maincor}.

\appendix

 \section{Reformulation of the Delta Method}\label{DeltaAppendix}

Here we introduce and develop a new simple form of the delta method that is more general than what we need in this paper. To better demonstrate the idea of this delta method, we first take a look at a simple delta method that has been successfully applied to prove various boundary problems. The idea behind this elementary method in detecting equality is to first reduce the size of the equation to some parameter $Y>0$, and then detect the equality by a congruence modulus a prime $p>Y$, i.e. \begin{align}\label{trivialdelta}
    \delta(n=0)=\delta\left(n\equiv 0\Mod{p}\right)h\left(\frac{n}{Y}\right)
\end{align}
for some function $h\left(\frac{n}{Y}\right)$ controlling the size of $|n|$ to be less than $Y$.

While there have been recent successes in using such a simple delta method (see \cite{taspecttrivial}, \cite{besseldelta}, \cite{burgesstrivial}), we would also like to weaken the restriction of $p>Y$. In application, one often wishes to choose a smaller modulus $p$ if possible, which has led to the development of various other delta methods. Inspired by the Duke-Friedlander-Iwaniec's delta method \cite[\S 3]{DFIdelta}, here we use a slight modification. By incorporating a hyperbola or divisor trick similar to the start of the proof of D-F-I's delta method, one gains the freedom of choosing moduli as shown in the following theorem.

\begin{thm}\label{DeltaMethod}
    Let $\varepsilon>0$, $n,q$ be integers such that $q>0$ and $|n|\ll N\rightarrow\infty$. Let $C,D>0$ be parameters such that $C>N^\varepsilon$. Let $U, W\neq 0$ be non-negative smooth even functions such that $U(0)=1, W(0)=0$ and $U$ decays exponentially at $\infty$. Then for any $A>0$, \begin{align*}
        \delta(n=0)=\mathcal{S}_1-\mathcal{S}_2+O_A\left(N^{-A}\right),
    \end{align*}
    where \begin{align*}
        \mathcal{S}_1=U_C\sum_{1\leq c<CN^\varepsilon}\delta\left(n\equiv 0\Mod{cq}\right)W(c)U\left(\frac{n}{cDq}\right)U\left(\frac{c}{C}\right)
    \end{align*}
    and \begin{align*}
        \mathcal{S}_2=U_C\sum_{1\leq d<DN^\varepsilon}\delta(n\equiv 0\Mod{dq})W\left(\frac{n}{dq}\right)U\left(\frac{n}{Cdq}\right)U\left(\frac{d}{D}\right),
    \end{align*}
    with \begin{align*}
        U_C=\left(\sum_{c\geq1}W(c)U\left(\frac{c}{C}\right)\right)^{-1}.
    \end{align*}
\end{thm}

The reader should observe that the role of $U\left(\frac{n}{cDq}\right)$ is to control $|n|$ to be of size less than $cDqN^\varepsilon$. Furthermore, we are not detecting equality using a single congruence modulus $p>CDq$, but instead we are averaging over a collection of congruences where we have freedom to control their sizes. Finally, the reader should also observe that $\mathcal{S}_1$ is the sum that carries the delta term $n=0$ naturally and $\mathcal{S}_2$ is identically $0$ when $n=0$.

\begin{proof}
Let \begin{align*}
    \mathcal{D}=\delta(n=0)\sum_{c\geq1}W(c)U\left(\frac{c}{C}\right),
\end{align*}
\begin{align*}
    S_2=\sum_{0\neq d\in\Z}\sum_{c\geq1}\delta(n=cdq)W(c)U\left(\frac{n}{cDq}\right)U\left(\frac{n}{Cdq}\right)
\end{align*}
and define \begin{align*}
    S_1=\mathcal{D}+S_2.
\end{align*}

We start by rewriting $d=\frac{n}{cq}$ in $S_2$, we get \begin{align*}
    S_2=&\sum_{c\geq1}\delta\left(n\equiv 0\Mod{cq}, n\neq 0\right)W(c)U\left(\frac{n}{cDq}\right)U\left(\frac{c}{C}\right).
\end{align*}
Hence adding up $\mathcal{D}$ and $S_2$ with $U(0)=1$, we get \begin{align}
    S_1=\sum_{c\geq1}\delta\left(n\equiv 0\Mod{cq}\right)W(c)U\left(\frac{n}{cDq}\right)U\left(\frac{c}{C}\right).
\end{align}
Instead, if we start by rewriting $c=\frac{n}{dq}$ in $S_2$, we get \begin{align*}
    S_2=\sum_{0\neq d\in\Z}\delta\left(n\equiv 0\Mod{dq}, \frac{n}{d}>0\right)W\left(\frac{n}{dq}\right)U\left(\frac{n}{Cdq}\right)U\left(\frac{d}{D}\right).
\end{align*}
If $n>0$, we get \begin{align*}
    S_2=\sum_{d\geq 1}\delta\left(n\equiv 0\Mod{dq}\right)W\left(\frac{n}{dq}\right)U\left(\frac{n}{Cdq}\right)U\left(\frac{d}{D}\right).
\end{align*}
and if $n<0$, we get \begin{align*}
    S_2=&\sum_{d\geq1}\delta\left(n\equiv 0\Mod{-dq}\right)W\left(-\frac{n}{dq}\right)U\left(-\frac{n}{Cdq}\right)U\left(-\frac{d}{D}\right)\\
    =&\sum_{d\geq 1}\delta\left(n\equiv 0\Mod{dq}\right)W\left(\frac{n}{dq}\right)U\left(\frac{n}{Cdq}\right)U\left(\frac{d}{D}\right)
\end{align*}
as $U$ and $W$ are even. Combining both cases above with the case $n=0$ giving $0$ in $S_2$ as $W(0)=0$, \begin{align}
    S_2=\sum_{d\geq 1}\delta\left(n\equiv 0\Mod{dq}\right)W\left(\frac{n}{dq}\right)U\left(\frac{n}{Cdq}\right)U\left(\frac{d}{D}\right).
\end{align}
Truncating the $c$ and $d$-sum in $S_1$ and $S_2$ by the exponential decay of $U$ and defining \begin{align*}
    U_C:=\left(\sum_{c\geq1}W(c)U\left(\frac{c}{C}\right)\right)^{-1},
\end{align*}
we obtain Theorem \ref{DeltaMethod}.
\end{proof}

In this paper, we apply this delta method with the standard use of additive characters to detect the congruence condition, i.e. \begin{align*}
    \delta(n\equiv 0\Mod{c})=\frac{1}{c}\sum_{\alpha\Mod{c}}e\left(\frac{\alpha n}{c}\right),
\end{align*}
with the choice of parameters $C=D, q=1$, and $W(x)=W'\left(\frac{x}{C}\right)$ with $W'\in C_c^\infty([-2,-1]\bigcup [1,2])$. This gives us the following corollary.

\begin{cor}\label{DeltaCor}
    Let $\varepsilon>0$, $n,q$ be integers such that $q>0$, $|n|\ll N\rightarrow\infty$ and let $C>N^\varepsilon$ be a parameter. Let $U\in C_c^\infty(\R), W\in C_c^\infty([-2,-1]\bigcup [1,2])$ be a fixed non-negative even function such that $U(x)=1$ for $-2\leq x\leq 2$. Then we have \begin{align*}
        \delta(n=0)=\frac{1}{\mathcal{C}}\sum_{c\geq1}\frac{1}{c}\sum_{\alpha\Mod{cq}}e\left(\frac{\alpha n}{cq}\right)h\left(\frac{c}{C},\frac{n}{cCq}\right),
    \end{align*}
    with $\mathcal{C}=\displaystyle\sum_{c\geq1}W\left(\frac{c}{C}\right)\sim C$ and \begin{align*}
        h\left(x,y\right)=W\left(x\right)U\left(x\right)U\left(y\right)-W(y)U(x)U(y).
    \end{align*}
    In particular, $h$ is a fixed smooth function satisfying $h(x,y)\ll \delta(|x|,|y|\ll 1).$
\end{cor}

\section{Gamma Functions}

The first important tool to study Gamma functions is its asymptotics given by Stirling's approximation. See, for example, \cite{stirling1}, \cite{stirling2} for details.

\begin{lemma}\label{Stirling}
    There exists a sequence of constants $(\gamma_n)$, called the Stirling coefficient, such that for any $z\in\C$ with $|\arg(z)|\leq\frac{\pi}{2}$, $N\geq1$, we have \begin{align*}
        \Gamma(z)=\sqrt{2 \pi} z^{z-\frac{1}{2}} e^{-z}\left(\sum_{n=0}^{N-1}(-1)^{n} \frac{\gamma_{n}}{z^{n}}+O\left(|z|^{-N}\right)\right).
    \end{align*}
\end{lemma}

Applying the Stirling's approximation, we get the following technical lemma.

\begin{lemma}\label{GammaRatio}
    Let $s=\sigma+i\tau$ with $\sigma, \tau\in\R$, $\nu>0$ such that $\nu+\sigma>0$. For any $N\geq1$, we have $$\frac{\Gamma\left(\nu+i\tau\right)}{\Gamma\left(\nu-i\tau\right)}=\left(\frac{\nu^2+\tau^2}{e^2}\right)^{i\tau}e\left(\frac{2\nu-1}{2\pi}\arctan\left(\frac{\tau}{\nu}\right)\right)W_{\nu,N}(\tau)+O\left(\nu^N+\tau^N\right),$$
    for some $1$-flat function $W_{\nu,N}$.
\end{lemma}
\begin{proof}
    Apply Stirling's approximation, we get for any $N\geq1$, \begin{align*}
        \frac{\Gamma\left(\nu+i\tau\right)}{\Gamma\left(\nu-i\tau\right)}=&\frac{\left(\nu+i\tau\right)^{\nu+i\tau-\frac{1}{2}}e^{-(\nu+i\tau)}}{\left(\nu-i\tau\right)^{\nu-i\tau-\frac{1}{2}}e^{-(\nu-i\tau)}}\frac{\displaystyle\sum_{n=0}^{N-1} \frac{(-1)^{n}\gamma_{n}}{(\nu+i\tau)^{n}}+O\left(|\nu+i\tau|^{-N}\right)}{\displaystyle\sum_{n=0}^{N-1} \frac{(-1)^{n}\gamma_{n}}{(\nu-i\tau)^{n}}+O\left(|\nu-i\tau|^{-N}\right)}\\
        =&\left(\frac{\nu^2+\tau^2}{e^2}\right)^{i\tau}e^{i(2\nu-1)\arctan\left(\frac{\tau}{\nu}\right)}W_{\nu,N}(\tau)+O\left(\nu^N+\tau^N\right),
    \end{align*}
    where \begin{align*}
        W_{\nu,N}(\tau)=\sum_{n=0}^{N-1} \frac{(-1)^{n}\gamma_{n}}{(\nu+i\tau)^{n}}\left(\sum_{n=0}^{N-1} \frac{(-1)^{n}\gamma_{n}}{(\nu-i\tau)^{n}}\right)^{-1}
    \end{align*}
    is $1$-flat.
\end{proof}

Finally we have the following technical lemma.
\begin{lemma}\label{GammaDiff}
    Let $A>0, B\in\R$ such that $\pm B>0$. Then \begin{align*}
        F^\pm(B):=\frac{\Gamma\left(\frac{3}{4}+iB\right)}{\Gamma\left(\frac{3}{4}-iB\right)}\mp i\frac{\Gamma\left(\frac{1}{4}+iB\right)}{\Gamma\left(\frac{1}{4}-iB\right)}\ll_A B^{-A}.
    \end{align*}
\end{lemma}
\begin{proof}
    Apply the reflection formula \begin{align*}
        \Gamma(1-z) \Gamma(z)=\frac{\pi}{\sin (\pi z)}
    \end{align*}
    which holds for any $z\not\in\Z$, we have \begin{align*}
        F^\pm(B)=&\frac{\Gamma\left(\frac{3}{4}+iB\right)\Gamma\left(\frac{1}{4}-iB\right)\mp i\Gamma\left(\frac{1}{4}+iB\right)\Gamma\left(\frac{3}{4}-iB\right)}{\Gamma\left(\frac{3}{4}-iB\right)\Gamma\left(\frac{1}{4}-iB\right)}\\
        =&\pi\left(\Gamma\left(\frac{3}{4}-iB\right)\Gamma\left(\frac{1}{4}-iB\right)\right)^{-1}\left(\frac{1}{\sin\left(\pi\left(\frac{3}{4}+iB\right)\right)}\mp i \frac{1}{\sin\left(\pi\left(\frac{3}{4}-iB\right)\right)}\right)\\
        =&2\pi i\left(\Gamma\left(\frac{3}{4}-iB\right)\Gamma\left(\frac{1}{4}-iB\right)\left(e^{-B+\frac{3\pi i}{4}}-e^{B-\frac{3\pi i}{4}}\right)\left(e^{B+\frac{3\pi i}{4}}-e^{-B-\frac{3\pi i}{4}}\right)\right)^{-1}\\
        &\times \left(e^{B+\frac{3\pi i}{4}}-e^{-B-\frac{3\pi i}{4}}\mp ie^{-B+\frac{3\pi i}{4}}+\pm ie^{B-\frac{3\pi i}{4}}\right),
    \end{align*}
    and \begin{align*}
        &e^{B+\frac{3\pi i}{4}}-e^{-B-\frac{3\pi i}{4}}\mp ie^{-B+\frac{3\pi i}{4}}+\pm ie^{B-\frac{3\pi i}{4}}\\
        =&e^B\left(e^\frac{3\pi i}{4}\pm ie^{-\frac{3\pi i}{4}}\right)-e^{-B}\left(e^{-\frac{3\pi i}{4}}\pm i e^{\frac{3\pi i}{4}}\right)\\
        =&2e^{B+\frac{3\pi i}{4}}\delta\left(\pm=-\right)+2e^{-B+\frac{-3\pi i}{4}}\delta\left(\pm=+\right).
    \end{align*}
    This gives exponential decay as $\pm B>0$ and concludes the proof of Lemma since \begin{align*}
        \Gamma\left(\frac{3}{4}-iB\right)\Gamma\left(\frac{1}{4}-iB\right)\left(e^{-B+\frac{3\pi i}{4}}-e^{B-\frac{3\pi i}{4}}\right)\left(e^{B+\frac{3\pi i}{4}}-e^{-B-\frac{3\pi i}{4}}\right)\neq0
    \end{align*}
    as $B\neq0$.
\end{proof}

\section{Bessel Functions}

Here we collect the standard facts of Bessel functions, all of the facts listed here can be found in, for example, \cite{watsonBessel}.

Fix $s\in\C$, we have the asymptotic expansion of Bessel functions as follows. For any $x>0$, there exists $P^\varepsilon$-inert function $W^\pm$ such that \begin{align}\label{BesselAsymptotics}
    J_s(4\pi x)=&\sum_\pm\frac{1}{\sqrt{x}+1}e\left(\pm 2x\right)W^\pm(x).
\end{align}
This asymptotic expansion is usually written only when $x\gg P^\varepsilon$. However, by Lemma \ref{besselBoundsDerivatives} below, the Bessel functions are $P^\varepsilon$-inert when $x\ll P^\varepsilon$ and bounded by $O(1)$ when $s$ is fixed, justifying (\ref{BesselAsymptotics}) for all $x>0$.

For the derivative properties, we have the following lemma.
\begin{lemma}\label{besselBoundsDerivatives}
    Let $\nu\in\C$, then we have \begin{enumerate}
        \item \begin{align*}
            J_\nu'(x)=&\frac{1}{2}\left(J_{\nu-1}(x)-J_{\nu+1}(x)\right)
        \end{align*}
        \item For any $a>0, j\in\N$, \begin{align*}
            \left(\frac{2}{a}\right)^j\frac{d^j}{dx^j}\left(x^{\frac{\nu+j}{2}}J_{\nu+j}(a\sqrt{x})\right)=x^\frac{\nu}{2}J_\nu(a\sqrt{x}).
        \end{align*}
    \end{enumerate}
\end{lemma}
\begin{proof}
    To get (2), use the relation \begin{align*}
        \left(y^\nu J_\nu(y)\right)'=&y^\nu J_{\nu-1}(y).
    \end{align*}
\end{proof}

\section{Oscillatory Integrals}

The first tool to study oscillatory integrals is stationary phase analysis as shown in the following lemma, which is Lemma 3.1 in \cite{inert}.

\begin{lemma}\label{stat}
    Suppose that $w$ is a $X$-inert function compact support on $[Z,2Z]$, so that $w^{(j)}(t)\ll \left(Z/X\right)^{-j}$. Also,suppose that $\phi$ is real-valued, smooth and satisfies $\phi^{(j)}(t)\ll\frac{Y}{Z^j}$ for some $\frac{Y}{X^2}\geq R\geq1$ and all $t$ in the support of $w$. Let \begin{align*}
        I=\int_{-\infty}^\infty w(t)e^{i\phi(t)}dt.
    \end{align*}
    \begin{enumerate}
        \item If $|\phi'(t)|\gg\frac{Y}{Z}$ for all $t$ in the support of $w$, then $I\ll_A ZR^{-A}$ for any $A>0$.
        \item If $\phi''(t)\gg\frac{Y}{Z^2}$ for all $t$ in the support of $w$, and there exists $t_0\in\R$ such that $\phi'(t_0)=0$, then \begin{align*}
            I=\frac{e^{i \phi\left(t_{0}\right)}}{\sqrt{\phi^{\prime \prime}\left(t_{0}\right)}} F\left(t_{0}\right)+O_{A}\left(Z R^{-A}\right)
        \end{align*}
        where $F$ is a $X$-inert function (depending on $A$) with the same support as $w$.
    \end{enumerate}
\end{lemma}

\begin{remark}
    In the statement written in \cite[Lemma 3.1]{inert}, $F$ is only supported on $t_0\sim Z$. However, that lemma is a restatement of \cite[Prop. 8.2]{BKY} and one can check its proof to see $F$ is actually supported in the support of $w$.
\end{remark}

Moreover, we have the following simple second derivative bound with fewer conditions by \cite[Lemma 5]{2ndDerBound}: Let $\phi$ be a real-valued function such that $\phi''(t)\gg R$ for all $t\in [a,b]$. Then \begin{align*}
    \int_a^b e(\phi(t))dt\ll\frac{1}{\sqrt{r}}.
\end{align*}
Together with integration by parts we get the following lemma.

\begin{lemma}\label{2ndDerB}
    Let $w$ be an $X$-inert function compactly supported on $[a,b]$ and $\phi$ be a real-valued function such that $\phi''(t)\gg R$ for $t\in [a,b]$. Then \begin{align*}
        \int_\R w(t)e(\phi(t))dt\ll\frac{\int_\R|w'(t)|dt}{\sqrt{R}}.
    \end{align*}
\end{lemma}

We have the following technical lemma to tackle the oscillatory integrals we encounter.

\begin{lemma}\label{HugeIntegralAnalysis}
    Let $A,C>0$, $B\in\R$ and let $V\in C_c^\infty(\R)$ be a $P^\varepsilon$-inert function supported on $[1,2]$. Let \begin{align*}
        G_k^\pm(s)=\left(\frac{\Gamma\left(\frac{2+s}{2}\right)}{\Gamma\left(\frac{1-s}{2}\right)}\mp i\frac{\Gamma\left(\frac{1+s}{2}\right)}{\Gamma\left(-\frac{s}{2}\right)}\right)\frac{\Gamma\left(\frac{k+s}{2}\right)\Gamma\left(\frac{k+1+s}{2}\right)}{\Gamma\left(\frac{k-s-1}{2}\right)\Gamma\left(\frac{k-s}{2}\right)},
    \end{align*}
    and define \begin{align*}
        I(B,C)=\frac{1}{2\pi i}\int_{\left(-\frac{1}{2}\right)}\left(\pi^{-3}C\right)^sG_k^\pm\left(s\right)\int_0^\infty V(y)e\left(B\sqrt{y}\right)y^{-s-1}dyds.
    \end{align*}
    For $B\ll P^\varepsilon$, there exists a $P^\varepsilon$-inert function $V_0^\pm(B,C)$ such that \begin{align*}
        I(B,C)=\frac{V_0^\pm(B,C)}{\sqrt{C}}\delta\left(C\gg k^{-2}P^{-\varepsilon}\right)+O\left(P^{-A}\right).
    \end{align*}
    For $B\gg P^\varepsilon$, there exists functions $W_{1,B,C}^\pm(t), W_{2,B,C}^\pm(t)$ supported on $[1,2]$, $P^\varepsilon$-inert in $B,C,t$ such that for $P^\varepsilon\ll B\ll kP^{-\varepsilon}$, \begin{align*}
        I(B,C)=&C^{-1/2}W_{1,B,C}^\pm\left(\frac{|B|Ck^2}{8\pi}\right)\\
        &\times e\left(\pm\frac{2}{B^2C}\left(1-\sqrt{1-\left(\frac{B^2Ck}{4\pi}\right)^2}\right)\mp\frac{k}{\pi}\arctan\left(\frac{4\pi}{B^2Ck}\left(1-\sqrt{1-\left(\frac{B^2Ck}{4\pi}\right)^2}\right)\right)\right)+O\left(P^{-A}\right),
    \end{align*}
    and for $B\gg kP^\varepsilon$, \begin{align*}
        I(B,C)=&C^{-1/2}W_{2,B,C}^\pm\left(\frac{4\pi}{|B|^3C}\right)\\
        &\times e\left(\pm\frac{2}{B^2C}\left(1+\sqrt{1-\left(\frac{B^2Ck}{4\pi}\right)^2}\right)\mp\frac{k}{\pi}\arctan\left(\frac{4\pi}{B^2Ck}\left(1+\sqrt{1-\left(\frac{B^2Ck}{4\pi}\right)^2}\right)\right)\right)+O\left(P^{-A}\right).
    \end{align*}
    For $kP^{-\varepsilon}\ll B\ll kP^\varepsilon$, there exists functions $W_{0,B,C}^\pm(t)$ supported on $(-1,1)$, $W_{0,B,C,T}^\pm(t), W_{0,B,C,-T}^\pm(t)$ supported on $\left(\frac{k}{|B|}-\frac{1}{T},\frac{k}{|B|}+\frac{1}{T}\right)$, $TP^\varepsilon$-inert in $t$ and $P^\varepsilon$-inert in $B,C$ such that
    \begin{align}
        I(B,C)=\delta\left(k^{-3}P^{-\varepsilon}\ll C\ll k^{-3}P^\varepsilon\right)\left(I_0(B,C)+\sum_{P^{-\varepsilon}\leq T\leq k^\frac{1}{3}P^{-\varepsilon} \text{Dyadic }}\left(I_T(B,C)+I_{-T}(B,C)\right)\right)+O\left(P^{-A}\right),
    \end{align}
    where \begin{align*}
        I_0(B,C)=&k^2\int_0^\infty W_{0,B,C}^\pm\left(k^\frac{1}{3}P^{-\varepsilon}\left(t-\frac{k}{|B|}\right)\right)e\left(\mp\frac{Bt}{2\pi}\log\left(\frac{|B|C(k^2+B^2t^2)}{8\pi et}\right)\mp\frac{k}{\pi}\arctan\left(\frac{Bt}{k}\right)\right)dt
    \end{align*}
    and for $\sigma=\pm1$, \begin{align*}
        I_{\sigma T}(B,C)=&k^\frac{3}{2}\sqrt{T}W_{0,B,C,\sigma T}^\pm\left(\frac{4\pi}{|B|^3C}\left(1+\sigma\sqrt{1-\left(\frac{B^2Ck}{4\pi}\right)^2}\right)\right)\\
        &\times e\left(\pm\frac{2}{B^2C}\left(1+\sigma\sqrt{1-\left(\frac{B^2Ck}{4\pi}\right)^2}\right)\mp\frac{k}{\pi}\arctan\left(\frac{4\pi}{B^2Ck}\left(1+\sigma\sqrt{1-\left(\frac{B^2Ck}{4\pi}\right)^2}\right)\right)\right).
    \end{align*}
\end{lemma}
\begin{proof}
    We start by analysing the $y$-integral. Writing $s=-\frac{1}{2}+it$, we have \begin{align*}
        \int_0^\infty V(y)e\left(B\sqrt{y}\right)y^{-s-1}dy=\int_0^\infty V(y)y^{-1/2}e(f(y))dy,
    \end{align*}
    where \begin{align*}
        f(y)=B\sqrt{y}-\frac{t}{2\pi}\log(y).
    \end{align*}
    Differentiating, we get \begin{align*}
        f'(y)=&\frac{B}{2\sqrt{y}}-\frac{t}{2\pi y}\\
        f''(y)=&-\frac{B}{4y^\frac{3}{2}}+\frac{t}{2\pi y^2}
    \end{align*}
    and for $y\sim1$, $j\geq2$, \begin{align*}
        f^{(j)}(y)\ll |B|+|t|.
    \end{align*}
    By (1) in Lemma \ref{stat}, we get arbitrary saving unless $|f'(y)|\ll \sqrt{|B|+|t|}P^\varepsilon$ for some $1\leq y\leq 2$. Hence we get arbitrary saving unless $|t|\ll |B|P^\varepsilon$. For $B\ll P^\varepsilon$, we get arbitrary saving unless $|t|\ll P^\varepsilon$. Hence we can freely change the order of integration and shift the contour to the right. This together with (1) in Lemma \ref{GammaRatio} yields for any $A\geq0$, \begin{align*}
        I(B,C)\ll_A \left(Ck^2\right)^A+P^{-A}.
    \end{align*}
    By taking $A$ to be sufficiently big, we get arbitrary saving unless $C\gg k^{-2}P^{-\varepsilon}$. Finally, shifting back to $-\frac{1}{2}$ and a change of variable gives us the desired result for the case $B\ll P^\varepsilon$, \begin{align*}
        I(B,C)=\frac{V_0^\pm(B,C)}{\sqrt{C}}\delta\left(C\gg k^{-2}P^{-\varepsilon}\right)+O\left(P^{-A}\right),
    \end{align*}
    where \begin{align*}
        V_0^\pm(B,C)=\frac{\sqrt{\pi}}{2}\left(1-U\left(Ck^2P^\varepsilon\right)\right)\int_{|t|\ll P^\varepsilon}\left(\pi^{-3}C\right)^{it}G_k^\pm\left(-\frac{1}{2}+it\right)\int_0^\infty V(y)e\left(B\sqrt{y}\right)y^{-\frac{1}{2}-it}dydt,
    \end{align*}
    with $U\in C_c^\infty(\R)$ is a fixed function with $U(x)=1$ for $-1\leq x\leq 1$. Then $V_1$ is a $P^\varepsilon$-inert function when $B\ll P^\varepsilon$.
    
    For $|B|\gg P^\varepsilon$, we have to analyse further. The proof is divided into 2 steps. Choose $\sigma=-\frac{1}{2}$ and write $t=\im(s)$.

    \textbf{Step 1}: Continuation of analysis on $y$-integral.
    
    As mentioned above, (1) in Lemma \ref{stat} gives arbitrary saving unless $|f'(y)|\ll \sqrt{|B|+|t|}P^\varepsilon$ for some $1\leq y\leq 2$. With the cost of negligible error, we focus on the case $|f'(y_1)|\ll \sqrt{|B|+|t|}P^\varepsilon$ for some $1\leq y\leq 2$. In such case, we have $Bt\geq 0, B=\frac{t}{\pi\sqrt{y_1}}+O\left(\sqrt{B}\right)$, and hence \begin{align*}
        f''(y)=&-\frac{B}{4y^\frac{3}{2}}+\frac{t}{2\pi y^2}=\frac{t}{2\pi y^\frac{3}{2}}\left(-\frac{1}{2\sqrt{y_1}}+\frac{1}{\sqrt{y}}\right)\\
        =&\frac{t(4y_1-y)}{4\pi y^2\sqrt{y_1}(2\sqrt{y_1}+\sqrt{y})}\sim B.
    \end{align*}
    Here we used $1\leq y_1,y\leq 2$ to show $4y_1-y\sim 1$. Now explicit computation shows \begin{align*}
        y_0=\frac{t^2}{\pi^2B^2} \hspace{1cm} \text{ and }\hspace{1cm} f(y_0)=\frac{t}{\pi}\log\frac{\pi eB}{t}
    \end{align*}
    satisfies $f'(y_0)=0$. Hence by (2) in Lemma \ref{stat}, for any $A>0$, there exists some $P^\varepsilon$-inert function $V_2$ supported on $[1,2]$ such that \begin{align*}
        \int_0^\infty V(y)y^{-1/2}e\left(B\sqrt{y}\right)y^{-it}dy=f''(y_0)^{-\frac{1}{2}}V_2\left(\frac{t^2}{\pi^2B^2}\right)e(f(y_0))+O\left(P^{-A}\right).
    \end{align*}
    Here we incorporated the case $|f'(y)|\gg \sqrt{|B|+|t|}P^\varepsilon$ for all $1\leq y\leq 2$ with the cost of negligible error. Defining \begin{align*}
        V_B\left(\frac{t^2}{B^2}\right)=\sqrt{\frac{B}{f''(y_0)}}V_2\left(\frac{t^2}{\pi^2B^2}\right),
    \end{align*}
    we get $V_B(t)$ is a function supported on $1\leq t\leq 2$, $P^\varepsilon$-inert on $t$ and $B$, and \begin{align}\label{conclusion2}
        \int_0^\infty V(y)y^{-1/2}e\left(B\sqrt{y}\right)y^{-it}dy=B^{-1/2}V_B\left(\frac{t^2}{B^2}\right)e(f(y_0))+O\left(P^{-A}\right).
    \end{align}
    
    \textbf{Step 2}: Analysis of the $t$-integral.
    
    Inserting (\ref{conclusion2}) into $I$, changing the order of integration and a change of variable yields \begin{align*}
        I(B,C)=&\frac{1}{2\pi\sqrt{B}}\int_\R V_B\left(\frac{t^2}{B^2}\right)\left(\pi^{-3}C\right)^{-\frac{1}{2}+it}\left(\frac{\pi e B}{t}\right)^{2it}G_k^\pm\left(-\frac{1}{2}+it\right)dt+O\left(P^{-A}\right)\\
        =&\frac{\sqrt{\pi B}}{2\sqrt{C}}\int_\R V_B\left(t^2\right)\left(\frac{e^2 C}{\pi t^2}\right)^{iBt}G_k^\pm\left(-\frac{1}{2}+iBt\right)dt+O\left(P^{-A}\right).
    \end{align*}
    Recall that \begin{align*}
        G_k^\pm\left(-\frac{1}{2}+iBt\right)=\left(\frac{\Gamma\left(\frac{3}{4}+\frac{iBt}{2}\right)}{\Gamma\left(\frac{3}{4}-\frac{iBt}{2}\right)}\mp i\frac{\Gamma\left(\frac{1}{4}+\frac{iBt}{2}\right)}{\Gamma\left(\frac{1}{4}-\frac{iBt}{2}\right)}\right)\frac{\Gamma\left(\frac{k-\frac{1}{2}+iBt}{2}\right)\Gamma\left(\frac{k+\frac{1}{2}+iBt}{2}\right)}{\Gamma\left(\frac{k-\frac{1}{2}-iBt}{2}\right)\Gamma\left(\frac{k+\frac{1}{2}-iBt}{2}\right)}.
    \end{align*}
    As $B\gg P^\varepsilon$, applying Lemma \ref{GammaDiff} on $\frac{\Gamma\left(\frac{3}{4}+\frac{iBt}{2}\right)}{\Gamma\left(\frac{3}{4}-\frac{iBt}{2}\right)}\mp i\frac{\Gamma\left(\frac{1}{4}+\frac{iBt}{2}\right)}{\Gamma\left(\frac{1}{4}-\frac{iBt}{2}\right)}$ gives arbitrary saving unless $\mp Bt>0$. On the other hand, Lemma \ref{GammaRatio} implies for any $A>0$, there exists some functions $V_3(t), V_4(t)$ such that they are $P^\varepsilon$-inert when $t\sim1$ and \begin{align*}
        G_k^\pm\left(-\frac{1}{2}+iBt\right)=&\left(V_3(t)\left(\frac{\frac{9}{4}+B^2t^2}{4e^2}\right)^{\frac{iBt}{2}}e\left(\frac{1}{4\pi}\arctan\left(\frac{2Bt}{3}\right)\right)\mp iV_4(t)\left(\frac{\frac{1}{4}+B^2t^2}{4e^2}\right)^{\frac{iBt}{2}}e\left(-\frac{1}{4\pi}\arctan \left(2Bt\right)\right)\right)\\
        &\times \left(\frac{\left(k-\frac{1}{2}\right)^2+B^2t^2}{4e^2}\frac{\left(k+\frac{1}{2}\right)^2+B^2t^2}{4e^2}\right)^{\frac{iBt}{2}}e\left(\frac{k-\frac{3}{2}}{2\pi}\arctan\left(\frac{Bt}{k-\frac{1}{2}}\right)+\frac{k-\frac{1}{2}}{2\pi}\arctan\left(\frac{Bt}{k+\frac{1}{2}}\right)\right)\\
        &+O\left(P^{-A}\right).
    \end{align*}
    Defining \begin{align*}
        V_5^\pm(\mp Bt)=&\left(V_3(Bt)\left(\frac{9}{4B^2t^2}+1\right)^{\frac{iBt}{2}}e\left(\frac{1}{4\pi}\arctan\left(\frac{2Bt}{3}\right)\right)\mp iV_4(Bt)\left(\frac{1}{4B^2t^2}+1\right)^{\frac{iBt}{2}}e\left(-\frac{1}{4\pi}\arctan \left(2Bt\right)\right)\right)\\
        &\times \left(\frac{\left(k-\frac{1}{2}\right)^2+B^2t^2}{k^2+B^2t^2}\frac{\left(k+\frac{1}{2}\right)^2+B^2t^2}{k^2+B^2t^2}\right)^{\frac{iBt}{2}}e\left(-\frac{3}{4\pi}\arctan\left(\frac{Bt}{k-\frac{1}{2}}\right)-\frac{1}{4\pi}\arctan\left(\frac{Bt}{k+\frac{1}{2}}\right)\right)\\
        &\times e\left(\frac{k-\frac{3}{2}}{2\pi}\arctan\left(\frac{Bt}{k(2k-1)+B^2t^2}\right)+\frac{k-\frac{1}{2}}{2\pi}\arctan\left(-\frac{Bt}{k(2k+1)-B^2t^2}\right)\right),
    \end{align*}
    then $V_5^\pm(Bt)$ is $P^\varepsilon$-inert when $t\sim1$ and, \begin{align*}
        G_k^\pm\left(-\frac{1}{2}+iBt\right)=&V_5^\pm(Bt)\left(\frac{|Bt|(k^2+B^2t^2)}{8e^3}\right)^{iBt}e\left(\frac{k}{\pi}\arctan\left(\frac{Bt}{k}\right)\right)+O\left(P^{-A}\right).
    \end{align*}
    
    This gives us \begin{align*}
        I(B,C)=&\frac{\sqrt{\pi B}}{2\sqrt{C}}\int_\R V_B\left(t^2\right)V_5^\pm(\mp Bt)\delta(\mp Bt>0)\left(\frac{|Bt|C(k^2+B^2t^2)}{8\pi et^2}\right)^{iBt}e\left(\frac{k}{\pi}\arctan\left(\frac{Bt}{k}\right)\right)dt+O\left(P^{-A}\right)\\
        =&\frac{\sqrt{\pi B}}{2\sqrt{C}}\int_0^\infty V_B\left(t^2\right)V_5^\pm(|B|t)e(g(t))dt+O\left(P^{-A}\right),
    \end{align*}
    where \begin{align*}
        g(t)=\mp\frac{Bt}{2\pi}\log\left(\frac{|B|C(k^2+B^2t^2)}{8\pi et}\right)\mp\frac{k}{\pi}\arctan\left(\frac{Bt}{k}\right).
    \end{align*}
    Differentiating, we get \begin{align*}
        g'(t)=&\mp\frac{B}{2\pi}\log\left(\frac{|B|C(k^2+B^2t^2)}{8\pi t}\right)\\
        g''(t)=&\mp\frac{B(-k^2+B^2t^2)}{2\pi t(k^2+B^2t^2)}=\mp\frac{B(k+|B|t)}{2\pi t(k^2+B^2t^2)}(|B|t-k),
    \end{align*}
    and for $t\sim 1$, $j\geq2$, \begin{align*}
        g^{(j)}(t)\ll B.
    \end{align*}
    Note that (1) in Lemma \ref{stat} shows that we get arbitrary saving unless $C\sim \min\left\{|B|^{-1}k^{-2},|B|^{-3}\right\}$.
    
    Here we have to separate into cases due to the size of $g''(t)$ can be small if $k$ and $B$ are close. For the case $B\ll kP^{-\varepsilon}$, $g''(t)\sim B$ for $t\sim 1$. Computing the phase point $g'
    (t_0)=0$, \begin{align*}
        t_0=&\frac{4\pi}{|B|^3C}\left(1-\sqrt{1-\left(\frac{B^2Ck}{4\pi}\right)^2}\right)\\
        g(t_0)=&\pm\frac{Bt_0}{2\pi}\mp\frac{k}{\pi}\arctan\left(\frac{|B|t_0}{k}\right)\\
        =&\pm\frac{2}{B^2C}\left(1-\sqrt{1-\left(\frac{B^2Ck}{4\pi}\right)^2}\right)\mp\frac{k}{\pi}\arctan\left(\frac{4\pi}{B^2Ck}\left(1-\sqrt{1-\left(\frac{B^2Ck}{4\pi}\right)^2}\right)\right).
    \end{align*}
    Hence by (2) in Lemma \ref{stat}, for any $A>0$, there exists a function $V_{1,B}^\pm(t)$ supported on $1\leq t\leq 2$ such that it is $P^\varepsilon$-inert in both $t$ and $B$, and \begin{align*}
        I(B,C)=\frac{\sqrt{\pi B}}{2\sqrt{Cg''(t_0)}}V_{1,B}^\pm\left(t_0\right)e(g(t_0))+O\left(P^{-A}\right).
    \end{align*}
    Defining \begin{align*}
        W_{1,B,C}^\pm\left(\frac{|B|Ck^2}{8\pi}\right)=\frac{\sqrt{\pi B}}{2\sqrt{g''(t_0)}}V_{1,B}^\pm\left(t_0\right),
    \end{align*}
    we get $W_{1,B,C}\left(t\right)$ is $P^\varepsilon$-inert in $B,C,t$, supported on $(1/2,5/2)$ and the desired result \begin{align*}
        I(B,C)=&C^{-1/2}W_{1,B,C}^\pm\left(\frac{|B|Ck^2}{8\pi}\right)\\
        &\times e\left(\pm\frac{2}{B^2C}\left(1-\sqrt{1-\left(\frac{B^2Ck}{4\pi}\right)^2}\right)\mp\frac{k}{\pi}\arctan\left(\frac{4\pi}{B^2Ck}\left(1-\sqrt{1-\left(\frac{B^2Ck}{4\pi}\right)^2}\right)\right)\right)+O\left(P^{-A}\right).
    \end{align*}
    
    Similarly for $B\gg kP^\varepsilon$, we have $g''(t)\sim B$ for $t\sim1$. Computing the phase point $g'(t_0)=0$, \begin{align*}
        t_0=&\frac{4\pi}{|B|^3C}\left(1+\sqrt{1-\left(\frac{B^2Ck}{4\pi}\right)^2}\right)\\
        g(t_0)=&\pm\frac{Bt_0}{2\pi}\mp\frac{k}{\pi}\arctan\left(\frac{|B|t_0}{k}\right)\\
        =&\pm\frac{2}{B^2C}\left(1+\sqrt{1-\left(\frac{B^2Ck}{4\pi}\right)^2}\right)\mp\frac{k}{\pi}\arctan\left(\frac{4\pi}{B^2Ck}\left(1+\sqrt{1-\left(\frac{B^2Ck}{4\pi}\right)^2}\right)\right).
    \end{align*}
    Applying the same exact procedure as above yields the desired result \begin{align*}
        I(B,C)=&C^{-1/2}W_{2,B,C}^\pm\left(\frac{4\pi}{|B|^3C}\right)\\
        &\times e\left(\pm\frac{2}{B^2C}\left(1+\sqrt{1-\left(\frac{B^2Ck}{4\pi}\right)^2}\right)\mp\frac{k}{\pi}\arctan\left(\frac{4\pi}{B^2Ck}\left(1+\sqrt{1-\left(\frac{B^2Ck}{4\pi}\right)^2}\right)\right)\right)+O\left(P^{-A}\right)
    \end{align*}
    for some function $W_{2,B,C}(t)$ supported on $(1/2,5/2)$ and $P^\varepsilon$-inert on $B,C,t$.
    
    Finally, for the most complicated case $kP^{-\varepsilon}\ll B\ll kP^\varepsilon$, we have to analyse the integral depending on the size of $|B|t-k$. Recall that in such case we get arbitrary saving unless $k^{-3}P^{-\varepsilon}\ll C\ll k^{-3}P^\varepsilon$. Perform a smooth dyadic subdivision for $P^{-\varepsilon}\leq \left(t-\frac{k}{|B|}\right)\sim T\leq k^\frac{1}{3}P^\varepsilon$, i.e. take a fixed $\varphi\in C_c^\infty(\R^+)$ such that \begin{align*}
        1-\varphi_0(x):=\sum_{P^{-\varepsilon}\leq T\leq k^\frac{1}{3}P^{-\varepsilon} \text{ Dyadic }}\left(\varphi(Tx)+\varphi(-Tx)\right)
    \end{align*}
    satisfies $\varphi_0(x)=0$ for $k^{-\frac{1}{3}}P^{-\varepsilon}\ll |x|\ll P^\varepsilon$. Inserting this into $I$ yields \begin{align}
        I(B,C)=I_0(B,C)+\sum_{P^{-\varepsilon}\leq T\leq k^\frac{1}{3}P^{-\varepsilon} \text{ Dyadic}}\left(I_T(B,C)+I_{-T}(B,C)\right)+O\left(P^{-A}\right),
    \end{align}
    where \begin{align*}
        I_0(B,C)=&\frac{\sqrt{\pi B}}{2\sqrt{C}}\int_0^\infty V_B\left(t^2\right)V_5^\pm(|B|t)\varphi_0\left(t-\frac{k}{|B|}\right)e(g(t))dt
    \end{align*}
    and for $\sigma=\pm$, \begin{align*}
        I_{\sigma T}(B,C)=\frac{\sqrt{\pi B}}{2\sqrt{C}}\int_0^\infty V_B\left(t^2\right)V_5^\pm(|B|t)\varphi\left(\sigma T\left(t-\frac{k}{|B|}\right)\right)e(g(t))dt.
    \end{align*}
    For $I_0$, define \begin{align*}
        W_{0,B,C}^\pm\left(k^\frac{1}{3}P^{-\varepsilon}\left(t-\frac{k}{|B|}\right)\right)=\frac{\sqrt{\pi B}}{2k^2\sqrt{C}}V_B\left(t^2\right)V_5^\pm(|B|t)\varphi_0\left(t-\frac{k}{|B|}\right),
    \end{align*}
    then $W_{0,B,C}^\pm$ is a function supported on $(1/2,5/2)$, $P^\varepsilon$-inert on $B,C,t$ and \begin{align*}
        I_0(B,C)=&k^2\int_0^\infty W_{0,B,C}^\pm\left(k^\frac{1}{3}P^{-\varepsilon}\left(t-\frac{k}{|B|}\right)\right)e(g(t))dt.
    \end{align*}
    
    For $I_{\sigma T}(B,C)$, we do a similar treatment as the previous two cases. With the restriction $\left|t-\frac{k}{|B|}\right|\sim \frac{1}{T}$, we have from previous computation \begin{align*}
        g'(t)=&\mp\frac{B}{2\pi}\log\left(\frac{|B|C(k^2+B^2t^2)}{8\pi t}\right)\\
        g''(t)=&\mp\frac{B(-k^2+B^2t^2)}{2\pi t(k^2+B^2t^2)}=\mp\frac{B(k+|B|t)}{2\pi t(k^2+B^2t^2)}(|B|t-k)\gg \frac{kP^{-\varepsilon}}{T}=\frac{kP^{-\varepsilon}}{T^3}T^2,
    \end{align*}
    and for $t\sim 1$, $j\geq3$, \begin{align*}
        g^{(j)}(t)\ll kP^\varepsilon\ll \frac{k}{T^3}T^jP^\varepsilon.
    \end{align*}
    For $|B|t<k$, same computation as above shows the phase point $g'(t_0)=0$ is \begin{align*}
        t_0=&\frac{4\pi}{|B|^3C}\left(1-\sqrt{1-\left(\frac{B^2Ck}{4\pi}\right)^2}\right)\\
        g(t_0)=&\pm\frac{2}{B^2C}\left(1-\sqrt{1-\left(\frac{B^2Ck}{4\pi}\right)^2}\right)\mp\frac{k}{\pi}\arctan\left(\frac{4\pi}{B^2Ck}\left(1-\sqrt{1-\left(\frac{B^2Ck}{4\pi}\right)^2}\right)\right).
    \end{align*}
    Since $T\ll k^\frac{1}{3}P^{-\varepsilon}$, this implies $\frac{k}{T}P^{-\varepsilon}\gg T^2P^\varepsilon$ and hence the condition of (2) in Lemma \ref{stat} is satisfied to yield \begin{align*}
        I_{-T}(B,C)=\frac{\sqrt{\pi B}}{2\sqrt{Cg''(t_0)}}V_{0,B,-T}^\pm\left(t_0\right)e(g(t_0))+O\left(P^{-A}\right)
    \end{align*}
    for some $TP^\varepsilon$-inert function $V_{0,B,-T}^\pm(t)$ supported on $\left(\frac{k}{|B|}-\frac{1}{T},\frac{k}{|B|}+\frac{1}{T}\right)$ and $P^\varepsilon$-inert with respect to $B$. Defining \begin{align*}
        W_{0,B,C,-T}^\pm\left(t_0\right)=\frac{\sqrt{\pi B}}{2k^\frac{3}{2}\sqrt{CTg''(t_0)}}V_{0,B,T}^\pm\left(t_0\right),
    \end{align*}
    we get $W_{0,B,C,-T}^\pm\left(t\right)$ is $TP^\varepsilon$-inert in $t$, $P^\varepsilon$-inert in $B,C$, supported on $\left(\frac{k}{|B|}-\frac{1}{T},\frac{k}{|B|}+\frac{1}{T}\right)$ and the desired result \begin{align*}
        I_{-T}(B,C)=&k^\frac{3}{2}\sqrt{T}W_{0,B,C,-T}^\pm\left(\frac{4\pi}{|B|^3C}\left(1-\sqrt{1-\left(\frac{B^2Ck}{4\pi}\right)^2}\right)\right)\\
        &\times e\left(\pm\frac{2}{B^2C}\left(1-\sqrt{1-\left(\frac{B^2Ck}{4\pi}\right)^2}\right)\mp\frac{k}{\pi}\arctan\left(\frac{4\pi}{B^2Ck}\left(1-\sqrt{1-\left(\frac{B^2Ck}{4\pi}\right)^2}\right)\right)\right)+O\left(P^{-A}\right).
    \end{align*}
    
    Similarly for $|B|t>k$, same computation shows the phase point $g'(t_0)=0$ is \begin{align*}
        t_0=&\frac{4\pi}{|B|^3C}\left(1+\sqrt{1-\left(\frac{B^2Ck}{4\pi}\right)^2}\right)\\
        g(t_0)=&\pm\frac{2}{B^2C}\left(1+\sqrt{1-\left(\frac{B^2Ck}{4\pi}\right)^2}\right)\mp\frac{k}{\pi}\arctan\left(\frac{4\pi}{B^2Ck}\left(1+\sqrt{1-\left(\frac{B^2Ck}{4\pi}\right)^2}\right)\right).
    \end{align*}
    The same procedure implies there exists some function $W_{0,B,C,T}^\pm(t)$ that is supported on $[1,2]$ and $P^\varepsilon$-inert in $t$, $P^\varepsilon$-inert in $B,C$, supported on $\left(\frac{k}{|B|}-\frac{1}{T},\frac{k}{|B|}+\frac{1}{T}\right)$ such that \begin{align*}
        I_T(B,C)=&k^\frac{3}{2}\sqrt{T}W_{0,B,C,T}^\pm\left(\frac{4\pi}{|B|^3C}\left(1+\sqrt{1-\left(\frac{B^2Ck}{4\pi}\right)^2}\right)\right)\\
        &\times e\left(\pm\frac{2}{B^2C}\left(1+\sqrt{1-\left(\frac{B^2Ck}{4\pi}\right)^2}\right)\mp\frac{k}{\pi}\arctan\left(\frac{4\pi}{B^2Ck}\left(1+\sqrt{1-\left(\frac{B^2Ck}{4\pi}\right)^2}\right)\right)\right)+O\left(P^{-A}\right).
    \end{align*}
\end{proof}

Applying Lemma \ref{HugeIntegralAnalysis}, we prove the following two lemmas.

\begin{lemma}\label{HugeIntegralAnalysis2}
    Let $A,C_1,C_2>0$, $B_1,B_2,D\in\R$ and let $V,W\in C_c^\infty(\R)$ be a $P^\varepsilon$-inert function supported on $[1,2]$. Define \begin{align*}
        J=\int_0^\infty W(x)I\left(B_1,C_1x^{-1}\right)\overline{I\left(B_2,C_2x^{-1}\right)}e(Dx)dx,
    \end{align*}
    with \begin{align*}
        I(B,C)=\frac{1}{2\pi i}\int_{\left(-\frac{1}{2}\right)}\left(\pi^{-3}C\right)^sG_k^\pm\left(s\right)\int_0^\infty V(y)e\left(B\sqrt{y}\right)y^{-s-1}dyds
    \end{align*}
    as in Lemma \ref{HugeIntegralAnalysis}. Then we have the following.
    \begin{itemize}
        \item For $B\ll P^\varepsilon$, we have           \begin{align*}
                J\ll \frac{P^\varepsilon}{\sqrt{C_1C_2}}\delta\left(C_1,C_2\gg k^{-2}P^{-\varepsilon}, |D|\ll P^\varepsilon\right)+P^{-A}.
            \end{align*}
        \item For $P^\varepsilon\ll B\ll kP^{-\varepsilon}$, if $k^2\left|B_1^2C_1-B_2^2C_2\right|\ll P^\varepsilon$, we have \begin{align*}
            J\ll\frac{P^\varepsilon}{\sqrt{C_1C_2}}\delta\left(B_1C_1k^2, B_2C_2k^2\sim 1, |D|\ll P^\varepsilon\right)+P^{-A},
        \end{align*}
        and if $k^2\left|B_1^2C_1-B_2^2C_2\right|\gg P^\varepsilon$, we have \begin{align*}
            J\ll P^\varepsilon\left(C_1C_2k^2\left|B_1^2C_1-B_2^2C_2\right|\right)^{-1/2}\delta\left(B_1C_1k^2, B_2C_2k^2\sim 1, |D|\sim k^2\left|B_1^2C_1-B_2^2C_2\right|\right)+P^{-A}.
        \end{align*}
        \item For $B\gg kP^\varepsilon$, if $k^2\left|B_1^2C_1-B_2^2C_2\right|\ll P^\varepsilon$, we have \begin{align*}
            J\ll\frac{P^\varepsilon}{\sqrt{C_1C_2}}\delta\left(B_1^3C_1, B_2^3C_2\sim 1, \left|D\pm 4\left((B_1^2C_1)^{-1}-(B_2^2C_2)^{-1}\right)\right|\ll P^\varepsilon\right)+P^{-A},
        \end{align*}
        and if $k^2\left|B_1^2C_1-B_2^2C_2\right|\gg P^\varepsilon$, we have \begin{align*}
            J\ll& P^\varepsilon\left(C_1C_2k^2\left|B_1^2C_1-B_2^2C_2\right|\right)^{-1/2}\\
            &\times\delta\left(B_1^3C_1, B_2^3C_2\sim 1, \left|D\pm 4\left((B_1^2C_1)^{-1}-(B_2^2C_2)^{-1}\right)\right|\sim k^2\left|B_1^2C_1-B_2^2C_2\right|\right)+P^{-A}.
        \end{align*}
    \end{itemize}
\end{lemma}
\begin{proof}
    For $B\ll P^\varepsilon$, Lemma \ref{HugeIntegralAnalysis} gives us \begin{align*}
        J=\frac{1}{\sqrt{C_1C_2}}\delta\left(C_1,C_2\gg k^{-2}P^{-\varepsilon}\right)\int_0^\infty W(x)xV_0^\pm(B_1,C_1x^{-1})\overline{V_0^\pm(B_2,C_2x^{-1})}e\left(Dx\right)dx+O\left(P^{-A}\right),
    \end{align*}
    for some $P^\varepsilon$-inert function $V_0^\pm$. Repeated integration by parts gives arbitrary saving unless $|D|\ll P^\varepsilon$, and hence we get the desired statement.
    
    For $P^\varepsilon\ll B\ll kP^{-\varepsilon}$, Lemma \ref{HugeIntegralAnalysis} gives us \begin{align*}
        J=\frac{1}{\sqrt{C_1C_2}}\int_0^\infty W(x)xW_{1,B_1,C_1x^{-1}}^\pm\left(\frac{|B_1|C_1k^2}{8\pi x}\right)\overline{W_{1,B_2,C_2x^{-1}}^\pm\left(\frac{|B_2|C_2k^2}{8\pi x}\right)}e(f_1(x))dx+O\left(P^{-A}\right),
    \end{align*}
    with $W_{1,B,Cx^{-1}}^\pm(t)$ as defined in Lemma \ref{HugeIntegralAnalysis} is supported on $[1,2]$ and $P^\varepsilon$-inert in $B,Cx^{-1},t$, and \begin{align*}
        f_1(x)=&\pm\frac{2x}{B_1^2C_1}\left(1-\sqrt{1-\left(\frac{B_1^2C_1k}{4\pi x}\right)^2}\right)\mp\frac{k}{\pi}\arctan\left(\frac{4\pi x}{B_1^2C_1k}\left(1-\sqrt{1-\left(\frac{B_1^2C_1k}{4\pi x}\right)^2}\right)\right)\\
        &\mp\frac{2x}{B_2^2C_2}\left(1-\sqrt{1-\left(\frac{B_2^2C_2k}{4\pi x}\right)^2}\right)\pm\frac{k}{\pi}\arctan\left(\frac{4\pi x}{B_2^2C_2k}\left(1-\sqrt{1-\left(\frac{B_2^2C_2k}{4\pi x}\right)^2}\right)\right)+Dx.
    \end{align*}
    Note that the support of $W_{1,B,Cx^{-1}}^\pm$ implies we get arbitrary saving unless \begin{align*}
        B_1C_1k^2, B_2C_2k^2\sim 1.
    \end{align*}
    Applying Taylor expansions, we get for $Z\ll P^{-\varepsilon}$, there exists $a_{1,j}\in\R$ with $a_{1,1}=\frac{1}{4}, a_{1,2}=-\frac{5}{48}$ such that \begin{align*}
        \pm\frac{k}{2\pi Z}\left(1-\sqrt{1-Z^2}\right)\mp\frac{k}{\pi}\arctan\left(Z^{-1}\left(1-\sqrt{1-Z^2}\right)\right)=\mp\frac{k}{\pi}\sum_{j=1}^\infty a_{1,j}Z^{2j-1}.
    \end{align*}
    Hence we have \begin{align*}
        f_1(x)=\mp\frac{k}{\pi}\sum_{j=1}^\infty a_{1,j}\left(\frac{k}{4\pi x}\right)^{2j-1}\left((B_1^2C_1)^{2j-1}-(B_2^2C_2)^{2j-1}\right)+Dx.
    \end{align*}
    Differentiating gives us \begin{align*}
        f_1'(x)=\pm\frac{k}{\pi}\sum_{j=1}^\infty a_{1,j}\frac{2j-1}{x}\left(\frac{k}{4\pi x}\right)^{2j-1}\left((B_1^2C_1)^{2j-1}-(B_2^2C_2)^{2j-1}\right)+D
    \end{align*}
    and for $j\geq2$, \begin{align*}
        f_1^{(j)}(x)\sim_j k^2\left|B_1^2C_1-B_2^2C_2\right|.
    \end{align*}
    For $k^2\left|B_1^2C_1-B_2^2C_2\right|\ll P^\varepsilon$, repeated integration by parts gives us arbitrary saving unless $$|D|\ll P^\varepsilon.$$
    For $k^2\left|B_1^2C_1-B_2^2C_2\right|\gg P^\varepsilon$, (1) in Lemma \ref{stat} gives us arbitrary saving unless \begin{align*}
        |D|\sim k^2\left|B_1^2C_1-B_2^2C_2\right|
    \end{align*}
    and (2) in Lemma \ref{stat} gives us \begin{align*}
        J\ll P^\varepsilon\left(C_1C_2k^2\left|B_1^2C_1-B_2^2C_2\right|\right)^{-1/2},
    \end{align*}
    giving us the desired statement.
    
    For $B\gg kP^\varepsilon$, Lemma \ref{HugeIntegralAnalysis} gives us \begin{align*}
        J=\frac{1}{\sqrt{C_1C_2}}\int_0^\infty W(x)xW_{2,B_1,C_1x^{-1}}^\pm\left(\frac{4\pi}{|B_1|^3C_1}\right)\overline{W_{2,B_2,C_2x^{-1}}^\pm\left(\frac{4\pi}{|B_2|^3C_2}\right)}e(f_2(x))dx+O\left(P^{-A}\right),
    \end{align*}
    with $W_{2,B,Cx^{-1}}^\pm(t)$ as defined in Lemma \ref{HugeIntegralAnalysis} is supported on $[1,2]$ and $P^\varepsilon$-inert in $B,Cx^{-1},t$, and \begin{align*}
        f_2(x)=&\pm\frac{2x}{B_1^2C_1}\left(1+\sqrt{1-\left(\frac{B_1^2C_1k}{4\pi x}\right)^2}\right)\mp\frac{k}{\pi}\arctan\left(\frac{4\pi x}{B_1^2C_1k}\left(1+\sqrt{1-\left(\frac{B_1^2C_1k}{4\pi x}\right)^2}\right)\right)\\
        &\mp\frac{2x}{B_2^2C_2}\left(1+\sqrt{1-\left(\frac{B_2^2C_2k}{4\pi x}\right)^2}\right)\pm\frac{k}{\pi}\arctan\left(\frac{4\pi x}{B_2^2C_2k}\left(1+\sqrt{1-\left(\frac{B_2^2C_2k}{4\pi x}\right)^2}\right)\right)+Dx.
    \end{align*}
    Applying the properties of $\arctan$ and Taylor expansions, we get for $Z\ll P^{-\varepsilon}$, there exists $a_{2,j}\in\R$ with $a_{2,1}=\frac{1}{4}$ such that \begin{align*}
        &\pm\frac{k}{2\pi Z}\left(1+\sqrt{1-Z^2}\right)\mp\frac{k}{\pi}\arctan\left(Z^{-1}\left(1+\sqrt{1-Z^2}\right)\right)\\
        =&\mp\frac{k}{2}\pm\frac{k}{2\pi Z}\left(1+\sqrt{1-Z^2}\right)\pm\frac{k}{\pi}\arctan\left(Z\left(1+\sqrt{1-Z^2}\right)^{-1}\right)\\
        =&\mp\frac{k}{2}\pm\frac{k}{2\pi Z}\left(1+\sqrt{1-Z^2}\right)\pm\frac{k}{\pi}\arctan\left(Z^{-1}\left(1-\sqrt{1-Z^2}\right)\right)=\mp\frac{k}{2}\pm\frac{k}{\pi Z}\pm\frac{k}{\pi}\sum_{j=1}^\infty a_{2,j}Z^{2j-1}.
    \end{align*}
    Hence we have \begin{align*}
        f_2(x)=\pm 4x\left((B_1^2C_1)^{-1}-(B_2^2C_2)^{-1}\right)\pm\frac{k}{\pi}\sum_{j=1}^\infty a_{2,j}\left(\frac{k}{4\pi x}\right)^{2j-1}\left((B_1^2C_1)^{2j-1}-(B_2^2C_2)^{2j-1}\right)+Dx.
    \end{align*}
    Differentiating gives us \begin{align*}
        f_2'(x)=\pm 4\left((B_1^2C_1)^{-1}-(B_2^2C_2)^{-1}\right)\mp\frac{k}{\pi}\sum_{j=1}^\infty a_{2,j}\frac{2j-1}{w}\left(\frac{k}{4\pi x}\right)^{2j-1}\left((B_1^2C_1)^{2j-1}-(B_2^2C_2)^{2j-1}\right)+D
    \end{align*}
    and for $j\geq2$, \begin{align*}
        f_2^{(j)}(x)\sim_j k^2\left|B_1^2C_1-B_2^2C_2\right|.
    \end{align*}
    For $k^2\left|B_1^2C_1-B_2^2C_2\right|\ll P^\varepsilon$, repeated integration by parts gives us arbitrary saving unless $$\left|D\pm 4\left((B_1^2C_1)^{-1}-(B_2^2C_2)^{-1}\right)\right|\ll P^\varepsilon.$$
    For $k^2\left|B_1^2C_1-B_2^2C_2\right|\gg P^\varepsilon$, (1) in Lemma \ref{stat} gives us arbitrary saving unless \begin{align*}
        \left|D\pm 4\left((B_1^2C_1)^{-1}-(B_2^2C_2)^{-1}\right)\right|\sim k^2\left|B_1^2C_1-B_2^2C_2\right|
    \end{align*}
    and (2) in Lemma \ref{stat} gives us \begin{align*}
        J\ll P^\varepsilon\left(C_1C_2k^2\left|B_1^2C_1-B_2^2C_2\right|\right)^{-1/2},
    \end{align*}
    giving us the desired statement.
\end{proof}

\begin{lemma}\label{HugeIntegralAnalysisTransit}
    Let $\sigma=\pm1, A,C_1,C_2>0, T\geq0, B_1,B_2,D\in\R$ such that $kP^{-\varepsilon}\ll B_1,B_2\ll kP^\varepsilon$ and $k^{-3}P^{-\varepsilon}\ll C_1,C_2\ll k^{-3}P^\varepsilon$. Let $V,W\in C_c^\infty(\R)$ be a $P^\varepsilon$-inert function supported on $[1,2]$. Define \begin{align*}
        J_{\sigma T}=\int_0^\infty W(x)I_{\sigma T}(B_1,C_1x^{-1})\overline{I_{\sigma T}(B_2,C_2x^{-1})}e(Dx)dx,
    \end{align*}
    with \begin{align*}
        I_0(B,C)=&k^2\int_0^\infty W_{0,B,C}^\pm\left(k^\frac{1}{3}P^{-\varepsilon}\left(t-\frac{k}{|B|}\right)\right)e\left(\mp\frac{Bt}{2\pi}\log\left(\frac{|B|C(k^2+B^2t^2)}{8\pi et}\right)\mp\frac{k}{\pi}\arctan\left(\frac{Bt}{k}\right)\right)dt
    \end{align*}
    and for $T>0$, \begin{align*}
        I_{\sigma T}(B,C)=&k^\frac{3}{2}\sqrt{T}W_{0,B,C,\sigma T}^\pm\left(\frac{4\pi}{|B|^3C}\left(1+\sigma\sqrt{1-\left(\frac{B^2Ck}{4\pi}\right)^2}\right)\right)\\
        &\times e\left(\pm\frac{2}{B^2C}\left(1+\sigma\sqrt{1-\left(\frac{B^2Ck}{4\pi}\right)^2}\right)\mp\frac{k}{\pi}\arctan\left(\frac{4\pi}{B^2Ck}\left(1+\sigma\sqrt{1-\left(\frac{B^2Ck}{4\pi}\right)^2}\right)\right)\right)
    \end{align*}
    with $W_{0,B,C}^\pm(t)$ supported on $(-1,1)$, $W_{0,B,C,\sigma T}^\pm(t)$ supported on $\left(\frac{k}{|B|}-\frac{1}{T},\frac{k}{|B|}+\frac{1}{T}\right)$, $TP^\varepsilon$-inert in $t$ and $P^\varepsilon$-inert in $B,C$ as in Lemma \ref{HugeIntegralAnalysis}. Then we have \begin{itemize}
        \item If $|B_1t_1-B_2t_2|\ll P^\varepsilon$, we have \begin{align*}
            J_0\ll k^\frac{10}{3}P^\varepsilon\delta\left(|D|\ll P^\varepsilon\right)+P^{-A}.
        \end{align*}
        If $|B_1t_1-B_2t_2|\gg P^\varepsilon$, we have \begin{align*}
        J_0\ll k^\frac{10}{3}P^\varepsilon|B_1t_1-B_2t_2|^{-1/2}\delta\left(|D|\sim |B_1t_1-B_2t_2|\right)+P^{-A}.
        \end{align*}
        \item If $k^2\left|B_1^2C_1-B_2^2C_2\right|\ll P^\varepsilon$, we have \begin{align*}
            J_{\sigma T}\ll P^\varepsilon k^3T\delta\left(\left|D\right|\ll P^\varepsilon\right)+P^{-A}.
        \end{align*}
        If $k^2\left|B_1^2C_1-B_2^2C_2\right|\gg P^\varepsilon$, we have \begin{align*}
            J_{\sigma T}\ll P^\varepsilon k^3T\left(k^2\left|B_1^2C_1-B_2^2C_2\right|\right)^{-1/2}\delta\left(|D|\sim k^2\left|B_1^2C_1-B_2^2C_2\right|\right)+P^{-A}.
        \end{align*}
    \end{itemize}
\end{lemma}
\begin{proof}
    For the case $T=0$, changing the order of integration gives us \begin{align*}
        J_0=k^4\int_0^\infty\int_0^\infty\int_0^\infty W(x)W_{0,B_1,C_1x^{-1}}^\pm\left(k^\frac{1}{3}P^{-\varepsilon}\left(t_1-\frac{k}{|B_1|}\right)\right)\overline{W_{0,B_2,C_2}^\pm\left(k^\frac{1}{3}P^{-\varepsilon}\left(t_2-\frac{k}{|B_2|}\right)\right)}e(f_{0,0}(x,t_1,t_2))dxdt_1dt_2,
    \end{align*}
    where \begin{align*}
        f_{0,0}(x,t_1,t_2)=&\mp\frac{B_1t_1}{2\pi}\log\left(\frac{|B_1|C_1(k^2+B_1^2t_1^2)}{8\pi ext_1}\right)\mp\frac{k}{\pi}\arctan\left(\frac{B_1t_1}{k}\right)\\
        &\pm\frac{B_2t_2}{2\pi}\log\left(\frac{|B_2|C_2(k^2+B_2^2t_2^2)}{8\pi xet_2}\right)\pm\frac{k}{\pi}\arctan\left(\frac{B_2t_2}{k}\right)+Dx.
    \end{align*}
    Differentiating with respect to $x$ gives us \begin{align*}
        \frac{d}{dx}f_{0,0}(x,t_1,t_2)=\pm\frac{1}{2\pi x}(B_1t_1-B_2t_2)+D
    \end{align*}
    and for $j\geq2$, \begin{align*}
        \frac{d^j}{dx^j}f_{0,0}(x,t_1,t_2)\sim_j |B_1t_1-B_2t_2|.
    \end{align*}
    For $|B_1t_1-B_2t_2|\ll P^\varepsilon$, (1) in Lemma \ref{stat} gives us arbitrary saving unless \begin{align*}
        |D|\ll P^\varepsilon.
    \end{align*}
    Note that the support of $W_{0,B,Cx^{-1}}^\pm$ restricts the range of $t_1$ and $t_2$, giving us \begin{align*}
        J_0\ll k^\frac{10}{3}P^\varepsilon.
    \end{align*}
    For $|B_1t_1-B_2t_2|\gg P^\varepsilon$, (1) in Lemma \ref{stat} gives us arbitrary saving unless \begin{align*}
        |D|\sim |B_1t_1-B_2t_2|
    \end{align*}
    and (2) in Lemma \ref{stat} on the $x$-integral, together with the support of $W_{0,B,Cx^{-1}}^\pm$ restricting the range of $t_1$ and $t_2$ gives us \begin{align*}
        J_0\ll k^\frac{10}{3}P^\varepsilon|B_1t_1-B_2t_2|^{-1/2},
    \end{align*}
    giving us the desired statement.
    
    For the case $P^{-\varepsilon}\leq T\leq k^\frac{1}{3}P^{-\varepsilon}$, we have \begin{align*}
        J_{\sigma T}=&k^3T\int_0^\infty H(x)e(f_{0,\sigma T}(x))dx,
    \end{align*}
    where \begin{align*}
        H(x)=W(x)W_{0,B_1,C_1x^{-1},\sigma T_1}^\pm\left(\frac{4\pi x}{|B_1|^3C_1}\left(1+\sigma\sqrt{1-\left(\frac{B_1^2C_1k}{4\pi x}\right)^2}\right)\right)\overline{W_{0,B_2,C_2x^{-1},\sigma T}^\pm\left(\frac{4\pi x}{|B_2|^3C_2}\left(1+\sigma\sqrt{1-\left(\frac{B_2^2C_2k}{4\pi x}\right)^2}\right)\right)}
    \end{align*}
    and \begin{align*}
        f_{0,\sigma T}(x)=&\pm\frac{2x}{B_1^2C_1}\left(1+\sigma\sqrt{1-\left(\frac{B_1^2C_1k}{4\pi x}\right)^2}\right)\mp\frac{k}{\pi}\arctan\left(\frac{4\pi x}{B_1^2C_1k}\left(1+\sigma\sqrt{1-\left(\frac{B_1^2C_1k}{4\pi x}\right)^2}\right)\right)\\
        &\mp\frac{2x}{B_2^2C_2}\left(1+\sigma\sqrt{1-\left(\frac{B_2^2C_2k}{4\pi x}\right)^2}\right)\mp\frac{k}{\pi}\arctan\left(\frac{4\pi x}{B_2^2C_2k}\left(1+\sigma\sqrt{1-\left(\frac{B_2^2C_2k}{4\pi x}\right)^2}\right)\right)+Dx.
    \end{align*}
    If $k^2\left|B_1^2C_1-B_2^2C_2\right|\ll P^\varepsilon$, we apply the exact same treatment as done in the proof of Lemma \ref{HugeIntegralAnalysis2}, with the $\sigma=1$ case being the same as the $B\gg kP^\varepsilon$ case, and the $\sigma=-1$ case being the same as the $B\ll kP^\varepsilon$ case, giving us \begin{align*}
        J_{\sigma T}\ll P^\varepsilon k^3T\delta\left(|D|\ll P^\varepsilon\right)+P^{-A}.
    \end{align*}
    
    If $k^2\left|B_1^2C_1-B_2^2C_2\right|\ll P^\varepsilon$, we have to be slightly more careful because $W_{0,B,Cx^{-1},\sigma T}^\pm$ is $TP^\varepsilon$-inert instead of $P^\varepsilon$-inert. For $j=1,2$, the support of $W_{0,B_j,C_jx^{-1},\sigma T}^\pm$ gives us the restriction \begin{align*}
        \left|\frac{4\pi x}{|B_j|^3C_j}\left(1+\sigma\sqrt{1-\left(\frac{B_j^2C_jk}{4\pi x}\right)^2}\right)-\frac{b_jdP^\frac{1+v}{2}k}{2\sqrt{Nm_j}}\right|\ll\frac{1}{T}.
    \end{align*}
    This yields \begin{align*}
        \int_\R \left|H'(x)\right|dx\ll TP^\varepsilon\frac{1}{T}\ll P^\varepsilon.
    \end{align*}
    Performing the exact same differentiation as done in the proof of Lemma \ref{HugeIntegralAnalysis2}, with the $\sigma=1$ case being the same as the $B\gg kP^\varepsilon$ case, and the $\sigma=-1$ case being the same as the $B\ll kP^\varepsilon$ case, Lemma \ref{2ndDerB} gives us \begin{align*}
        J_{\sigma T}\ll P^\varepsilon k^3T\left(k^2\left|B_1^2C_1-B_2^2C_2\right|\right)^{-1/2}.
    \end{align*}
\end{proof}

\printbibliography

\end{document}